\documentclass[11pt,reqno]{amsart}
 \usepackage{graphicx}
\usepackage{amscd,amsmath,amsopn,amssymb,amsthm,multicol}
\usepackage[color,matrix, all, 2cell]{xy}
\usepackage{amscd}
\usepackage{lscape}
\usepackage{slashed}
\usepackage{graphicx}
\usepackage{setspace}
\usepackage{upgreek}
\usepackage{textgreek}
\usepackage{enumerate}
\usepackage{color}
\usepackage{lscape}
\usepackage{tikz}
\usepackage{multirow}
\usepackage{cancel}
\usepackage{soul}
\usepackage{harmony}
\usepackage{comment}
\usepackage{wasysym}
\usepackage{mathrsfs}
\usepackage{mathtools}
\usepackage{tcolorbox}
\usepackage{tikz}

\numberwithin{equation}{section}

\DeclareMathAlphabet{\mathscrbf}{OMS}{mdugm}{b}{n}

\DeclareMathOperator{\diag}{diag}
\DeclareMathOperator{\Ad}{Ad}

\DeclareMathOperator{\Id}{\mathsf{Id}}
\DeclareMathOperator{\tr}{tr}

\DeclareMathOperator{\Ric}{\mathsf{Ric}}

\DeclareMathOperator{\Iso}{\mathsf{Iso}}

\DeclareMathOperator{\R}{\mathbb{R}}

\DeclareMathOperator{\z}{\mathsf{z}}
\DeclareMathOperator{\ii}{\mathsf{i}}
\DeclareMathOperator{\jj}{\mathsf{j}}
\DeclareMathOperator{\kk}{\mathsf{k}}
\DeclareMathOperator{\dg}{\mathsf{deg}}

\DeclareMathOperator{\Span}{span}
\DeclareMathOperator{\rnk}{\mathsf{rank}}
\newcommand{\fr}{\mathfrak}
\newcommand{\al}{\alpha}
\newcommand{\be}{\beta}
\newcommand{\upa}{\upalpha}
\newcommand{\upb}{\upbeta}
\newcommand{\upg}{\upgamma}

\newcommand{\bb}{\mathbb}
\newcommand{\mc}{\mathcal}
\newcommand{\e}{\varepsilon}

\DeclareMathOperator{\SO}{\mathsf{SO}}
\DeclareMathOperator{\Sp}{\mathsf{Sp}}
 \DeclareMathOperator{\SU}{\mathsf{SU}}

\DeclareMathOperator{\U}{\mathsf{U}}

\DeclareMathOperator{\G}{\mathsf{G}}
\DeclareMathOperator{\F}{\mathsf{F}}
\DeclareMathOperator{\Oo}{\mathsf{O}}
\DeclareMathOperator{\E}{\mathsf{E}}
\DeclareMathOperator{\Ss}{\mathsf{S}}

\DeclareMathOperator{\Gl}{\mathsf{GL}}
\newcommand{\thickline}{\noalign{\hrule height 1pt}}

\theoremstyle{plain}
\newtheorem{lemma}{Lemma} [section]
\newtheorem{theorem}[lemma]{Theorem}
\newtheorem{corol}[lemma] {Corollary}
\newtheorem{prop} [lemma]{Proposition}

\theoremstyle{definition}

\newtheorem{remark}[lemma] {Remark}
\newtheorem*{remark*}{Remark}

\definecolor{dark}{rgb}{0.18,0.18,0.68}
\definecolor{mred}{rgb}{0.38,0.54,0.72}
\definecolor{crew}{rgb}{0.2,0.5,0.2}
\definecolor{mmg}{rgb}{0.31,0.50,0.23}
\definecolor{dblue}{rgb}{0.01,0.01,0.44}
\definecolor{red}{rgb}{0.57,0.11,0.15}
\definecolor{cobalt}{RGB}{12,89,178}
\usepackage[colorlinks,citecolor=cobalt,linkcolor=cobalt,urlcolor=cobalt,pdfpagemode=UseNone,backref = page]{hyperref}

\usepackage{tcolorbox}
\definecolor{mycolor}{rgb}{0.122, 0.435, 0.698}

\newtcbox{\mybox}{on line,
  colframe=mycolor,colback=mycolor!10!white,
  boxrule=0.5pt,arc=4pt,boxsep=0pt,left=6pt,right=6pt,top=6pt,bottom=6pt}
  
  \definecolor{mycol}{rgb}{0.422, 0.498, 0.135}
  \newtcbox{\myboxx}{on line,
  colframe=mycol,colback=mycol!10!white,
  boxrule=0.5pt,arc=4pt,boxsep=0pt,left=6pt,right=6pt,top=6pt,bottom=6pt}

 \input ulem.sty
 
\title[Homogeneous Einstein metrics on non-K\"ahler C-spaces]{Homogeneous Einstein metrics on non-K\"ahler C-spaces}

\author{Ioannis Chrysikos} 
\address{Faculty of Science, University of Hradec Kr\'alov\'e, Rokitanskeho 62, Hradec Kr\'alov\'e
50003, Czech Republic}
\email{ioannis.chrysikos@uhk.cz}

\author{Yusuke Sakane} 
\address{Osaka University, Department of Pure and Applied Mathematics, Graduate School of Information Science and Technology, Suita, Osaka 565-0871, Japan} 
\email{sakane@math.sci.osaka-u.ac.jp}

\begin{document}   

\begin{abstract}
We study homogeneous Einstein metrics on  indecomposable  non-K\"ahlerian C-spaces, i.e.  even-dimensional torus bundles $M=G/H$ with $\rnk G>\rnk H$ over   flag manifolds  $F=G/K$ of a compact simple Lie group $G$.  Based on the theory of painted  Dynkin diagrams  we  present the classification of such spaces. Next  we focus on the family 
\[
M_{\ell, m, n}:=\SU(\ell+m+n)/\SU(\ell)\times\SU(m)\times\SU(n)\,,\quad \ell, m, n\in\bb{Z}_{+}
\]
 and examine several of its geometric properties.  We show that invariant metrics on  $M_{\ell, m, n}$ are not diagonal and beyond certain exceptions their    parametrization depends on six real parameters. By using such an invariant Riemannian metric, we  compute the diagonal and the non-diagonal part of the Ricci tensor  and  present explicitly the     algebraic system of the homogeneous  Einstein equation.  For general positive integers  $\ell,  m, n$, by  applying  mapping degree theory we provide the existence of at least one $\SU(\ell+m+n)$-invariant Einstein metric on  $M_{\ell, m, n}$.  For $\ell=m$ we show the  existence of two $\SU(2m+n)$ invariant Einstein metrics on $M_{m, m, n}$, and for $\ell=m=n$  we obtain four $\SU(3n)$-invariant Einstein metrics on $M_{n, n, n}$. We also examine the isometry problem for these metrics, while for a plethora of  cases induced by fixed $\ell, m, n$, we provide the  numerical form of all non-isometric invariant Einstein metrics. 
\end{abstract}

\maketitle



\section*{Introduction}
\subsection*{Introduction}
A Riemannian  manifold $(M, g)$ is called Einstein if the metric $g$  is a solution of  the  Einstein equation $\Ric^{g}=\lambda g$, where  $\Ric^{g}$ is the Ricci tensor of $(M, g)$ and $\lambda$ is a real number, called the {Einstein constant}.   The Einstein equation is a system of  non-linear second order
  PDEs,   and there are no general   results available for  any dimension.  A traditional approach which makes the examination of this system  more tractable,   is based on the requirement of some additional  ``symmetry'' condition for  $g$.
 In this work we are concerned with invariant Einstein metrics on    homogeneous manifolds. Thus,  we assume that there is closed subgroup $G\subseteq \Iso(M, g)$ of the group of isometries   of $(M, g)$, acting transitively on $M$. Then, $M=G/H$ is a homogeneous space  and  $g$ is a $G$-invariant metric.  The study of homogeneous Einstein manifolds  is divided into cases according to the sign of $\lambda$.  For $\lambda>0$, the  Myers' theorem yields the compactness of      $M=G/H$  and $\pi_{1}(G/H)$ must be finite,    while for $\lambda<0$, $M=G/H$ must be non-compact (\cite[Thm.~7.56]{Bes}).
 Moreover, by \cite{AK75}  all Ricci-flat homogeneous manifolds ($\lambda=0$) are flat. From now on we focus on the compact case and refer to \cite{lafuente}  for  recent advances  about the non-compact case.

For a  $G$-invariant metric  the system  of non-linear PDEs  corresponding to the Einstein equation reduces to a  system of  non-linear algebraic equations.  This obviously makes the whole problem more accessible, but still  the classification of all compact homogeneous Einstein manifolds $(M=G/K, g)$ is a very difficult task, which at the present seems out of reach.  
However, nowadays it is available  a plethora of  existence or   classification results;  Most of them are related with certain classes of  compact homogeneous spaces, as for example symmetric spaces, spheres, Stiefel manifolds, generalized Wallach spaces, semisimple Lie groups, generalized flag manifolds, homogeneous spaces with a certain number of isotropy summands, and other.   Several  basic constructions of homogeneous Einstein manifolds can be found in \cite{Bes} and the references therein, while for more recent results we cite   the surveys \cite{Wa1, Wa3} and the papers  \cite{WZ2, Kim, PS,   Bohm1, Bohm2, BohmKerr, Gr,   DKer,   Chry2, stauros, ara, CS1, CS2,   cortes, statha, chen}.

\subsection*{Motivation}
The  purpose of this article is the investigation of homogeneous Einstein metrics on another class of compact homogeneous manifolds, namely on {\it non-K\"ahler C-spaces}.
A {\it C-space} is a compact simply connected   homogeneous complex manifold $M=G/H$ of a compact connected semisimple Lie group $G$.  Such   manifolds were introduced  during 50s by H.~C. Wang (\cite{Wang}).  The stability subgroup  $H$ of a C-space $M=G/H$  is a closed connected subgroup of $G$ whose semisimple part coincides with the semisimple part of the centralizer of a torus in $G$. Thus,  C-spaces fall into two categories with respect to the Euler characteristic $\chi$. If $\chi(G/H)>0$, then $M$ is a  so-called {\it generalized flag manifold}, and in this case the stabilizer $H$  is   the centralizer of a torus   in $G$. Recall that the generalized flag manifolds exhaust    all compact simply  connected homogeneous K\"ahler  manifolds  $F=G/H$  corresponding to a semisimple Lie group  $G$ (\cite{Ale, Forger, AP}).

Here we are interested in non-K\"ahlerian C-spaces, so we assume that $\rnk G>\rnk H$.  Any such space  is a 
  principal bundle over a flag manifold $F=G/K$ with structure group    a complex torus ${\rm T}^{2s}$  of real even dimension $2s:=\rnk G-\rnk H$ 
      \[
      {\rm T}^{2s}\cong K/H \longrightarrow M = G/H \overset{\pi}{\longrightarrow} F=G/K\,.
      \]
Note that non-K\"ahlerian  C-spaces may   admit     invariant   complex  structures   with zero  first Chern class, in contrast to    flag manifolds. In particular, certain types of such homogeneous  manifolds  provide  homogeneous examples of   Calabi-Yau structures with torsion (CYT),    hyper-K\"ahler structures with torsion (HKT),  or of further  Hermitian structures with torsion which are not K\"ahler.  Moreover,  these geometric structures have found  application also in  string theory compactifications. Therefore,  in  recent years  non-K\"ahlerian C-spaces  have   attracted  much  attention from both mathematicians and physicists   (see for example \cite{Pap, Fino, GGP, Podesta, Fino2}).   

	Next we will assume that  the Lie group $G$ is  simple, which is equivalent to  say that  $M=G/H$ is {\it not}   a product of homogeneous spaces, and similarly for $F=G/K$.   Such a  C-space (or  a flag manifold) will be called {\it indecomposable}, and  we should  mention that $M=G/H$ is {\it not} a circle bundle.  The last decades  there is a remarkable progress towards the  classification
  of  invariant Einstein metrics on the base space of the fibration $\pi : G/H\to G/K$ (see  \cite{Gr, Chry2, stauros, Chry1,  CS1}).  Surprisingly,   a similar systematic study of  invariant Einstein metrics  on the total space of $\pi$, is still missing.  For example, even for the low dimensional  Lie group $\SU(3)$, which is an indecomposable non-K\"ahlerian C-space fibered over the full flag manifold $\SU(3)/{\rm T}^{2}$, the full classification of invariant Einstein metrics is unknown.  This  gap is also reflected from the fact that a classification of generic non-K\"ahlerian C-spaces with respect to their isotropy representation is not known yet.  In principle, the  only known general result which applies to our discussion,  but   produces invariant Einstein metrics only on {\it decomposable} C-spaces,  is  \cite[Thm.~1.10, p.~224]{WZ2}.   This theorem was used  for example by C.~B\"ohm and M.~Kerr in  \cite{BohmKerr}, in combination with results of \cite{Bohm1}, for a description of  torus bundles   admitting an invariant Einstein metric up to dimension $11$; Such an example is the  8-dimensional manifold $\SU(2)\times\SU(2)\times\SU(2)/\U(1)_{k,\ell,m}$, which is a torus bundle of rank two  over $\Ss^2\times\Ss^2\times\Ss^2$, diffeomorphic to $\Ss^3\times\Ss^{3}\times\Ss^{2}$.  However,   one  can construct a large list of indecomposable C-spaces with $\rnk G>\rnk H$ (see  Section \ref{section2}). Hence,  it is fair to say that we know very little about all but a few small corners of the entire picture of non-isometric invariant Einstein metrics on non-K\"ahlerian C-spaces,  in contrast to the corresponding picture about flag manifolds (at least for the most meaningful  cosets).
  
 \subsection*{Outline} In this article we proceed with the first systematic examination of invariant Einstein metrics on   non-K\"ahler C-spaces $M=G/H$ with $G$ simple. Our aim is to highlight the difficulties of the problem and describe  some methods which can be considered as the first steps   towards the classification of invariant Einstein metrics on   non-K\"ahler indecomposable C-spaces.  The  main focus and most of our results, are about  the homogeneous Einstein equation on   the space
\[
M_{\ell, m, n}=G/H=\SU(N)/\SU(\ell)\times\SU(m)\times\SU(m)\,,
\]
 with $N:=\ell+m+n\geq 3$ and $\ell, m, n\in\bb{Z}_{+}$, where the embedding  $\SU(\ell)\times\SU(m)\times\SU(m)\to \SU(N)$ is the diagonal one.   This homogeneous  space is a  rank 2 torus bundle over  the classical flag manifold  $F=F_{\ell, m, n}=G/K=\SU(\ell+m+n)/\Ss(\U(\ell)\times\U(m)\times\U(n))$, 
\[
{\bb T}^{2}\cong K/H \longrightarrow G/H=\SU(N)/\SU(\ell)\times\SU(m)\times\SU(n)\overset{\pi}{\longrightarrow}G/K=\SU(N)/\Ss(\U(\ell)\times\U(m)\times\U(n))\,.
\]
Note that  the stability  subgroup $H$  is semisimple, and hence as we will see below, $M_{\ell, m, n}$ serves well to understand  the  difficulties of classifying invariant metrics and constructing the homogeneous Einstein equation,  for this class of indecomposable C-spaces $G/H$ whose stabilizer $H$ is semisimple. Such C-spaces in \cite{Wang} are called  {\it M-spaces}, but   this is a more general terminology which is in fact  not so common (the article \cite{Wang} lacks of a motivation for the use of this terminology) and can be confusing  (since it includes even non complex spaces).    Hence we will avoid to use it,  and from now on we agree to say that a  C-space  $M=G/H$ is of {\it semistrict type} when $G$ is simple and   the stabilizer $H$ is  {\it semisimple}  such that $\rnk G>\rnk H$.  If $M=G/H$ is a C-space, with  $G$ simple and  $H$  {\it reductive}, such that $\rnk G>\rnk H$, then $M$ will be called a C-space of   {\it strict type}.   Of course,  a C-space of strict or a semistrict type is always non-K\"ahlerian and   indecomposable.  A third class of indecomposable non-K\"ahlerian C-spaces arises when the stability group $H$ is just abelian, and such cases appear when the base space is a full flag manifold $F=G/T_{\rm max}$, with $G$ simple and $\rnk G\geq 3$.  In Section \ref{section2}, based on the well known  classification of flag manifolds in terms of {\it painted Dynkin diagrams} (\cite{Forger, Ale}) we present the classification of all non-K\"ahler C-spaces $M=G/H$ with $G$ simple, and indicate the corresponding type and fibration (see Theorem \ref{cspacesclas} and Appendix \ref{sectionapen}). 

    For a non-K\"ahler C-space $M=G/H$, the isotropy representation may  include  sub-representations  which not all of them are inequivalent each other (in contrast to the isotropy representation of a flag manifold). This fact has several geometric consequences; In comparison  with the invariant metrics on  the base space $F=G/K$ which are all {\it diagonal},  the invariant metrics on $M=G/H$ are more complicated objects, in particular the Ricci tensor $\Ric^{g}$  of a generic $G$-invariant metric $g$ on $M=G/H$  is {\it not   diagonal} and   this  adds a certain difficulty during the construction of the  homogeneous Einstein equation. 

To become more specific, let us  return back to the  fibration $\pi : M_{\ell, m, n}\to F_{\ell, m, n}$.   The isotropy representation of the base space $F_{\ell, m, n}$ decomposes with respect to $B$ (negative of Killing form), into three inequivalent and irreducible $\Ad(K)$-submodules $\fr{f}_1, \fr{f}_2, \fr{f}_3$. 
     Passing to a $B$-orthogonal reductive decomposition $\fr{g}=\fr{h}\oplus\fr{m}$ of    the total space $M_{\ell, m, n}=G/H$,  we see that the isotropy representation $\fr{m}\cong T_{o}M_{\ell, m, n}$ admits the following  $\Ad(H)$-invariant splitting  \[\fr{m}=\fr{f}_1\oplus\fr{f}_2\oplus\fr{f}_3\oplus\fr{f}_{0}=(\fr{f}_1\oplus\fr{f}_2\oplus\fr{f}_3)\oplus (\fr{f}_4\oplus\fr{f}_5)\,,\]
 which is {\it not}   unique, in the sense that the modules $\fr{f}_4, \fr{f}_5$ can be replaced by any pair $\hat{\fr{f}}_4, \hat{\fr{f}}_5$ of orthogonal 1-dimensional submodules of  $\fr{f}_{0}$.  As a conclusion, and after presenting  a suitable  parametrization  of the invariant products $( \ , \ )_{0}$ on $\fr{f}_{0}$ in terms of {\it nilpotent matrices} (see Proposition \ref{parametrize}), we show that the space $\mc{M}^{\SU(N)}_{M_{\ell, m, n}}$ of $\SU(N)$-invariant metrics on $M_{\ell, m, n}$ is  parametrized by six real parameters, under the  assumption that $\ell m \neq1$, $\ell n \neq 1$ and $m n \neq 1$ (see Corollary \ref{invRIEM}). This means that  beyond these values, general $G$-invariant metrics $g$ on   $M_{\ell, m, n}=G/H$  are specified by $\Ad(H)$-invariant inner products $\langle \ , \ \rangle$ on $\fr{m}=\fr{f}_1\oplus\fr{f}_2\oplus\fr{f}_3\oplus\hat{\fr{f}}_{4}\oplus\hat{\fr{f}}_{5}$, of  the following form
\begin{eqnarray*}
g\equiv\langle \ , \ \rangle&=&x_{1}B|_{\fr{f}_1}+x_2B|_{\fr{f}_2}+x_3B|_{\fr{f}_3}+( \ , \ )_{0}\\
&=&x_{1}B|_{\fr{f}_1}+x_2B|_{\fr{f}_2}+x_3B|_{\fr{f}_3}+v_4 ( \ , \ )|_{\hat{\frak f}_4} +v_5 ( \ , \ )|_{\hat{\frak f}_5},
\end{eqnarray*}
with $x_{i}>0, v_{j}>0$ for any $i=1, 2, 3$ and $j=4, 5$, while an extra parameter $c\in\bb{R}$ is appearing through the parametrization of $( \ , \ )_{0}$. 

Now, because $\hat{\fr{f}}_{4}\cong\hat{\fr{f}}_5$, the Ricci tensor of $g$ has both diagonal and  non-diagonal part, which we denote by $\Ric^{g}_1, \Ric^{g}_2, \Ric^{g}_3$ and $\Ric^{g}_4, \Ric^{g}_5, \Ric^{g}_{0}$, respectively, with $\Ric_{0}^{g}=\Ric^{g}(\hat{\fr{f}}_4, \hat{\fr{f}}_5)$.  A non-trivial task  which one faces during  the  description  of the Ricci tensor,  is the computation of the   non-zero $\langle \ , \  \rangle$-structure constants of $M_{\ell, m, n}$.  To overpass this problem  we   combine    the Lie algebraic structure underlying both of  the total space $M_{\ell, m, n}$ and  the base space $F_{\ell, m, n}$, with other techniques appearing also in \cite{ADN1, CS2}. On the other hand, the description of the non-diagonal part of $\Ric^{g}$ requires some extra care, and as we will see below the related computations are  long.

   The homogeneous  Einstein equation on  $(M_{\ell, m, n}, g=\langle \ , \ \rangle)$  is equivalent to the following system of equations
\[
\left\{\Ric^{g}_1=\lambda\,,  \quad 
\Ric^{g}_2=\lambda \,, \quad 
\Ric^{g}_3=\lambda \,,\quad
\Ric^{g}_4=\lambda \,, \quad 
\Ric^{g}_5=\lambda\,, \quad 
\Ric^{g}_{0}=0
\right\} \quad (\Sigma)\,.
\]
This is a system of  6 equations and 7 unknowns, namely $x_1, x_2, x_3, v_4, v_5, c$ and the Einstein constant  $\lambda$. We prove that the algebraic system  induced by $(\Sigma)$ is {\it linear in the variables $v_4, v_5$} (see Proposition \ref{newricci1}  and Corollary \ref{linearEIN}).  In particular, we may solve the equations $\Ric^{g}_{4}-\lambda=0=\Ric^{g}_5-\lambda$ with respect to $v_4$ and  $v_5$, respectively.  Also,  by the last equation we can  express $c$ in  terms of $x_2, x_3$. Then, a replacement  into  the first  three equations of $(\Sigma)$ shows that the Ricci components $\Ric^{g}_1, \Ric_{g}^{2}, \Ric^{g}_3$ of the diagonal part are homogeneous polynomials of degree $-1$ with respect the variables $x_1, x_2, x_3$.   In fact,  by  expressing  $\lambda$ in terms of three rational polynomials $\mathsf{t}_{i}(x_1, x_2, x_3)$ and by inserting a new parameter $t\in[0, 1]$, we may  consider   rational polynomials of the form
\[
T_{i}(t, x_1, x_2, x_3):=\frac{\mathsf{p}_{i}(x_1, x_2, x_3)}{1+t\cdot\mathsf{q}_{i}(x_1, x_2, x_3)}\,,\quad (i=1, 2, 3)\,.
\]
Note that for $t=0$,   $T_{i}(0, x_1, x_2, x_3)=\mathsf{p}_{i}(x_1, x_2, x_3)=\Ric^{\check g}(x_1, x_2, x_3)$ are the Ricci components of the base space $(F_{\ell, m, n}, {\check g})$, while for $t=1$, the expressions $T_{i}(1, x_1, x_2, x_3)$ coincide with the rational polynomials $\mathsf{t}_{i}(x_1, x_2, x_3)$, related to the Einstein constant on the total space $(M_{\ell, m, n}, g=\langle \ , \ \rangle)$. 
	Based on this notation one can introduce  a well defined homotopy $F_t : {\mathbb R}_+^3  \to \Ss^{2}$, given by
 \[
  F_t:= \frac{1}{\sqrt{{T_1}^2+{T_2}^2+{T_3}^2}} (T_1, T_2, T_3)\,.
 \]
This homotopy yields the following  characterization of  homogeneous Einstein metrics:  A point $({x_1}, {x_2}, {x_3})\in\bb{R}^{3}_{+}$  defines a $G$-invariant Einstein metric on the base space $F_{\ell, m, n}=G/K$, if and only if  $({x_1}, {x_2}, {x_3}) \in {F_0}^{-1}(p_0)$, where $p_0:= (1/\sqrt{3},  1/\sqrt{3}, 1/\sqrt{3})$. On the other hand,  for $t=1$  a point  $({x_1}, {x_2}, {x_3})\in\bb{R}^{3}_{+}$ will be a $G$-invariant Einstein metric on the $C$-space $M_{\ell,m, n}=G/H$, if and only if $({x_1}, {x_2}, {x_3}) \in {F_1}^{-1}(p_0)$ (see Proposition \ref{characterEIN}).
Motivated by \cite{Sak}, in order to find   $G$-invariant Einstein metrics on the $C$-space $M_{\ell, m, n}=G/H$ and the flag manifold  $F_{\ell, m , n}=G/K$,  we  may now  apply a method which is based on {\it mapping degree theory} with respect to  the map  
\[
f_t^{} : D_{+} \to {\mathbb R}^2\,,\quad f_t^{}(x_1, x_2, 1):= \psi\circ P \circ F_t(x_1, x_2, 1)\,.
\]
Here, without loss of generality we have normalized the metric by setting $x_3 = 1$ (see Remark \ref{homogpolyn}).  Moreover, $\psi : S^2-{(0,0,-1)} \to {\mathbb R}^2$ is the stereographic projection,    $P$ is a  rotation matrix which maps $p_0$ to $(0, 0, 1)$, and  $D_{+}$ is the domain in $\bb{R}^{3}_{+}$ defined by $D_{+}:=\{( x_1, x_2, 1) \in\bb{R}^3\,  |\,   x_1 > 0, \, x_2 > 0 \}$.  In terms of $f_{t}$, a point $ (x_1, x_2, 1)\in D_{+}$ 
corresponds to a  $G$-invariant Einstein metric on  $M_{\ell, m, n}=G/H$,  if and only if  $f_1^{}({x_1}, {x_2}, 1) =(0, 0)$, while a point $ ({x_1}, {x_2}, 1)\in  D_{+} $ 
defines a  $G$-invariant Einstein metric  on  the base space $F_{\ell, m, n}=G/K$, if and only if $f_0^{}({x_1}, {x_2}, 1) =(0, 0)$.    Then, based on arguments of mapping degree theory we  can  provide the existence of invariant Einstein metrics, in particular we obtain the following result (see Theorem \ref{MAINTHEM}).

\smallskip
{\bf Theorem 1.} {\it For general positive integers $\ell, m, n\in\bb{Z}_{+}$, the   C-space $M_{\ell, m, n}$ admits at least a  $\SU(\ell+m+n)$-invariant Einstein metric.}

 \medskip
Therefore,   the indecomposable C-spaces $M_{\ell, m, n}$  are homogeneous Einstein manifolds for arbitrary $\ell, m, n >0$. In order to improve this theorem and discuss some results of classification type,  we  restrict our attention to more special cases induced by  the family $M_{\ell, m, n}$. For example, for small $\ell\neq m\neq n$ we see that all members of $M_{\ell, m, n}$ admit, up to isometry, two or four invariant Einstein metrics, depending on the particular choice of  the parameters (see Tables  \ref{Table1},  \ref{Table21} and see also below). 
After that we examine the case  when  two of the three parameters $\ell, m, n$ coincide each other, e.g. $\ell=m$. We show that

\medskip
{\bf Theorem 2.}  
 {\it For $\ell=m$, 
there exist at least two $\SU(2m+n)$-invariant Einstein metrics on the C-space $M_{m, m, n}=\SU(2m+n)/\SU(m)\times\SU(m)\times\SU(n)$.
}

\medskip
Depending on the ordering between $m, n$, when $\ell=m$ we  can  actually present  the expressions of the Einstein metrics in  Theorem 2  (see Theorem \ref{theorem4.9}). 
Moreover, for small $m, n$ we  obtain the numerical forms  of all  invariant Einstein metrics and show that  there exist members of  $M_{m, m, n}$, which admit  one, two  or   four  invariant  Einstein metrics  (up to isometry). The numerical values of these metrics are given in   Tables  \ref{Table32}, \ref{Table31} and  \ref{Table41}, respectively.
    Finally, we examine  the case  $\ell=m=n$, where for $n\geq 2$ we obtain the  complete classification of  homogeneous Einstein metrics.

 \medskip
{\bf Theorem 3.}  
 {\it  For $n \geq 2$
there exist   exactly four $\SU(3n)$-invariant Einstein metrics on the C-space $M_{n, n, n}=\SU(3n)/\SU(n)\times\SU(n)\times\SU(n)$, given by 
  \begin{eqnarray*}  & & (x_1, x_2, x_3) = (1, 1, 1)\,, \quad (x_1, x_2, x_3) = (1, \al, 1)\,, \\ & &  (x_1, x_2, x_3) = (\al, 1, 1)\,, \quad (x_1, x_2, x_3) = (1/\al, 1/\al, 1)\,,
  \end{eqnarray*} 
where $\al$ is the solution of  $(2 n^2+1)\al^3-(2 n-1) (2 n+1)\al^2+4 (n^2+2)\al-4 (2 n^2+1)=0$.  Moreover, the invariant Einstein metrics   $(x_1, x_2, x_3) = (1, \al, 1)$,    $(x_1, x_2, x_3) = (\al, 1, 1)$ and $(x_1, x_2, x_3) = (1, 1, \al)$ are isometric each other.  Hence, up to isometry and scale, $M_{n, n, n}$ admits exactly two $\SU(3n)$-invariant Einstein metrics.  For  $n=1$ and the  C-space $M_{1, 1, 1}=\SU(3)$,  among the invariant metrics $g=\langle \ , \ \rangle$  defined above, only  the bi-invariant metric is an invariant Einstein metric.}

\medskip
For  arbitrary, distinct,  but {\it fixed}   $\ell, m, n$ which can be very large,  a  full classification of all invariant Einstein metrics is still possible. Such examples are presented in the following table, where we use  the normalization $x_3=1$ and   present only the values $x_1, x_2$ (for  the values of $v_4, v_5$ one can apply (\ref{eq_v4}) and (\ref{eq_v5}), respectively).
\begin{table}[ht]
\centering
{\small \vspace{0.1cm}
\renewcommand\arraystretch{1.4}
\begin{tabular}{ l | l | l | l }
$M_{\ell, m, n}$	& $\dim_{\R}M_{\ell, m, n}$ &	 $x_1$								& $x_2$	\\
\thickline
$M_{10000, 2, 3}$	       & 1000014	      &  0.49999812508758  & 0.5000039582837693\\ 
                                        &                      &  23333.9023351598 &  23333.902296584482 \\ 
$M_{10000, 99, 3}$         & 20400596     &  0.50024111495038 & 0.4997597438771520 \\
                                        &                     &  984.203593167392 & 984.36732072352000 \\                                    
$M_{2, 100000, 99999}$ & 20000599998 &  0.50000562505320 &  1.5000033749343452 \\
                                         &                      &  1.50000837497409 &  0.5000106250719541\\
              \hline
$M_{100000, 99999, 99998}$ & 59998800006 & 0.5000012500593759991 &    0.49999875004062503385 \\
                                                 &                        & 1.0000100001500034167 &   2.00000999994999841665 \\
          					& 			& 1.0000100001500055835  &  1.00000500007500279172 \\
                           		   &                           &   2.0000049998749947081 & 1.00000500007500170836       \\
		   \hline
\end{tabular}}
\end{table}
As one can see, here the conclusion   is exactly the same as for the low-dimensional cases described in  Tables  \ref{Table1} and \ref{Table21} in Section \ref{caseaaa}.   This means that  by fixing a C-space $M_{\ell, m, n}$, where all the three parameters $\ell, m, n$ are different each other, we   obtain exactly  two, or four non-isometric homogeneous Einstein metrics. Thus, we may conjecture that for   arbitrary but distinct parameters $\ell, m, n$, the total number of non-isometric invariant Einstein metrics on $M_{\ell, m, n}$ is the same as for these examples.  

Most of the ideas and methods developed in this article can be adapted  to several indecomposable non-K\"ahlerian C-space $M=G/H$ of semistrict type. Therefore,  in a forthcoming work we will present  invariant Einstein metrics  on further C-spaces of this type.  The investigation of invariant Einstein metrics on C-spaces of strict type remains an interesting open problem.

\subsection*{Structure of the article}  Section \ref{section1}  is about   the Ricci  tensor on a reductive  homogeneous space and  Section \ref{section2}  can be viewed as a short introduction to the  theory of indecomposable non-K\"ahler C-spaces. In the same section we present the classification of such homogeneous spaces (for exceptional Lie groups the results are given in Appendix \ref{sectionapen}). In  Section \ref{section3} we focus on the family $M_{\ell, m, n}$ and    compute the $B$-structure constants, classify invariant metrics and present the structure constants with respect to the  invariant metric $g=\langle \ , \ \rangle$, described above.   Section \ref{section4}  is devoted to the computation of the  Ricci tensor $\Ric^{g}$ of $g$, and finally in Section  \ref{section5}  we present  the homogeneous Einstein equation. In this section we prove {Theorems 1, 2,} and {3}, and  present several tables  including  the numerical forms of the new Einstein metrics for fixed $\ell, m, n$ (up to isometry for most of the cases).

\subsection*{Acknowledgements}  
 I.~C.~gratefully acknowledges   support via Czech Science Foundation (project no.~19-14466Y). Y.~S.~thanks the  Faculty of Science in Hradec Kr\'alov\'e University for hospitality, during a visiting in September 2019.


\section{Preliminaries}\label{section1}
In  this article we are interested in the formulation of the Einstein equation on compact homogeneous spaces $M=G/K$ of a compact semisimple Lie group $G$. 	Hence, it is useful to recall by \cite{PS}    an expression of the Ricci tensor for a  $G$-invariant metric $g$.  
\subsection{The Ricci tensor  of a reductive homogeneous space}
Let  us consider a compact homogeneous space $M=G/K$   a compact semisimple Lie group  $G$, where $K\subset G$ is  connected closed subgroup. 
We denote by $\fr{g}$ and $\fr{k}$   the corresponding Lie algebras, and by  $B$ the $\Ad(G)$-invariant inner product  induced from the {\it negative} of the Killing form $B_{\fr{g}}$ of $\fr{g}$. Let $\fr{g}$ = $\fr{k}\oplus\fr{m}$ be a reductive decomposition of $\fr{g}$ with respect to $B$,  i.e. $\fr{m}=\fr{k}^{\perp}$ with  $[\fr{k}, \fr{m}]\subset\fr{m}$. The normal metric on $M$ induced by the restriction $B|_{\fr{m}}$ is the so-called {\it Killing metric}. 

Assume  that the $\Ad(K)$-module $\fr{m}\cong T_{eK}$ decomposes into $r$ mutually inequivalent irreducible $\Ad(K)$-submodules, i.e. 
\begin{equation}\label{iso}
\fr{m} = {\frak m}_1 \oplus \cdots \oplus {\frak m}_r\,.
\end{equation} 
Then, by diagonalization and Schur's lemma,  it follows that any $G$-invariant metric $g$ on $G/K$ depends on $r$ positive real numbers $x_1, \dots, x_r$,  and is given by an $\Ad(K)$-invariant inner product $\langle \ , \ \rangle$ on $\fr{m}$ of the form
  \begin{equation} \label{invg}
\langle \ , \ \rangle =
x_1   B|_{\mbox{\footnotesize$ \frak m$}_1} + \cdots + 
 x_r  B|_{\mbox{\footnotesize$ \frak m$}_r}\,.
\end{equation}
Such invariant metrics  satisfy  $\langle \fr{m}_i, \fr{m}_j\rangle=0$ for any $i\neq j$ and are often called {\it diagonal}.   Obviously, any $G$-invariant symmetric covariant  tensor on $G/K$ having the same rank with an invariant Riemannian metric, is exactly of the same form, although not necessarily  positive definite.  
 In particular, the Ricci tensor $\Ric^{g}$ of a $G$-invariant Riemannian metric $g=\langle \ , \ \rangle$ on $G/K$ as above, can be expressed by $ \Ric^{g}=y_1 B|_{\mbox{\footnotesize$ \frak m$}_1}  + \cdots + y_{r} B|_{\mbox{\footnotesize$ \frak m$}_r}$, 
 for some real numbers $y_1, \ldots, y_r$, and thus the Ricci tensor is diagonal, i.e.  $\Ric^{g}(\fr{m}_{i}, \fr{m}_{j})=0$ whenever $i\neq j$, or in other terms 
\[
\Ric^{g}=\left(\begin{tabular}{l l l l }
$\Ric^g_1$ & 0  & $\ldots$  & 0\\
0 & $\Ric^g_2$  & $\ldots$  & 0 \\
$\vdots$ & $\vdots$ & $\ddots$ & $\vdots$  \\
0 & 0 & $\ldots$ & $\Ric^{g}_{r}$ 
\end{tabular}
\right)
\] 
with $\Ric^{g}_{k}:=\Ric^{g}(\fr{m}_k, \fr{m}_k)$, for any $1\leq k\leq r$.  
The computation of the Ricci components goes as follows: Given a  $g$-orthonormal basis $\{X_i\}$ of $\fr{m}$, the Ricci tensor is expressed by (see   \cite{Bes})
\begin{equation}\label{genric}
\Ric^{g}(X, Y) = -\frac12\sum_i\langle [X, X_i], [Y, X_i]\rangle +\frac12B(X, Y)
       +\frac14 \sum _{i, j}\langle [X_i, X_j], X\rangle\langle [X_i, X_j], Y\rangle\,,
\end{equation}
for any $X, Y\in\fr{m}$.   Since  we work with $B$, it is now convenient to consider a $B$-orthonormal basis  of $\fr{m}$, adapted to the decomposition (\ref{iso}).  Hence,   set  $d_{i}:=\dim_{\bb{R}}\fr{m}_{i}$ and let us denote by $\lbrace e_{\al}^{i} \rbrace_{\al=1}^{d_{i}}$ a $B$-orthonormal basis, 
such that 
$e_{\al}^{i} \in {\frak m}_i$ for some $i$, 
$\al < \be$ if $i<j$  with $e_{\al}^{i} \in {\frak m}_i$ and
$e_{\be}^{j} \in {\frak m}_j $, for any $0\leq i, j\leq r$.  
 Consider the numbers $A^{\gamma}_{\al\be}=B([e_{\al}^{i},e_{\be}^{j}],e_{\gamma}^{k})$, with
$[e_{\al}^{i},e_{\be}^{j}]
= {\sum_{\gamma}
A^{\gamma}_{\al\be} e_{\gamma}^{k}}$, and set 
\[
\displaystyle{k \brack {i \ j}}:=\sum (A^\gamma_{\alpha \beta})^2\,,
\]
 where the sum is
taken over all indices $\alpha, \beta, \gamma$,  with $e_\al^{i} \in
{\frak m}_i,\ e_\be^{j} \in {\frak m}_j,\ e_\gamma^{k}\in {\frak m}_k$. 
These quantities are non-negative real numbers, which are independent of the 
$B$-orthonormal bases chosen for ${\frak m}_i, {\frak m}_j, {\frak m}_k$. However,  they depend on the fixed decomposition of $\fr{m}$ and hence on (\ref{iso}), and finally they are symmetric in all three indices,  
$\displaystyle{k \brack {i \ j}} =\displaystyle{k \brack {j \ i}} =\displaystyle{j \brack {k \ i}}$. We shall refer to  $\displaystyle{k \brack {i \ j}}$ by the term {\it $B$-structure constants} of $M=G/K$.

\begin{lemma}\label{ric2}\textnormal{(\cite{PS})}
The components $\Ric^{g}_{1}, \dots, \Ric^{g}_{r}$ 
of the Ricci tensor $\Ric^{g}$ of the  invariant metric $g=\langle \ , \ \rangle$  on $M=G/K$ defined by $(\ref{invg})$, are expressed by
\begin{equation}   \label{Riccomp}
\Ric^{g}_{k} = \frac{1}{2x_k}+\frac{1}{4d_k}\sum_{j,i}
\frac{x_k}{x_j x_i} {k \brack {j \ i}}
-\frac{1}{2d_k}\sum_{j,i}\frac{x_j}{x_k x_i} {j \brack {k \ i}}\,,
 \quad\quad (k= 1,\ \dots, r)\,,
\end{equation}
where  the sum is taken over $i, j =1,\dots, r$.
\end{lemma} 
When   for any $i\neq j$ we have  $\fr{m}_i\ncong\fr{m}_j$ as $\Ad(K)$-representations, then any $G$-invariant Einstein metric on $M=G/K$ corresponds to a positive real solution $(x_1, \ldots, x_r)\in\bb{R}^{r}_{+}$  of the system   \[
\{\Ric^g_1=\lambda, \ \Ric^g_2=\lambda, \ \ldots, \ \Ric^{g}_{r}=\lambda\},\] 
  for some $\lambda\in\bb{R}_{+}$ (Einstein constant).
In contrast,  if some of the modules are equivalent as $\Ad(K)$-representations, i.e. $\fr{m}_{i}\simeq\fr{m}_j$ for some $i\neq j$ with $1\leq i, j\leq r$, then the  metric $g$ and the  Ricci tensor are  not diagonal.  Hence, in the related system of the homogeneous Einstein equation  one has to  include the  constraint obtained from the equation  $\Ric^g(\fr{m}_i, \fr{m}_j)=0$.  The description of the constraints imposed by this condition, can be formulated  in terms of (\ref{genric}), and this suits the examination which we   present below.


\section{The classification of indecomposable non-K\"ahlerian C-spaces}\label{section2}
\subsection{Material from C-spaces}
A C-space is a  compact simply connected  homogeneous complex manifold
  $M = G/H$ of a compact  semisimple Lie  group  $G$. Such spaces   were introduced by H.~C.~Wang in \cite{Wang}. In fact, when the Euler characteristic is non-zero, then $M$ is a compact homogeneous K\"ahler manifold and so a generalized flag manifold.   Next we are interested in non-K\"ahlerian C-spaces, i.e. $\rnk G>\rnk H$.
According to   \cite{Wang}, the stability subgroup  $H$ of $M$ is a closed connected subgroup of $G$ whose semisimple part coincides with the semisimple part of the centralizer of a torus in $G$. It turns out that  $M=G/H$ is the  total  space  of  a  principal    bundle 
      over a generalized flag manifold $F=G/K$ with structure group  a complex torus ${\rm T}^{2s}$  of real even dimension given by $2s:={\rm rnk} G-{\rm rnk} H$. 
      Consequently,  $F$ needs to be at least of second Betti number 2, or bigger, i.e. $b_{2}(F):=\sharp(\Pi_{B})=v\geq 2$ such that $F=G/K=G/K'\cdot {\rm T}^{v}$, where $K'$ denotes the semisimple part of $K$, $\Pi_{B}\subset\Pi$ is the set of painted black simple roots and $\Pi$ is a fundamental basis of the root system $R$ of $G$ (see \cite{AP, Chry1, CS1,  Podesta, AC1} for details). 

Let  $ \mathfrak{g}= \mathfrak{k}\oplus \mathfrak{f}$  be a reductive  decomposition  related to  the flag manifold $F=G/K$,  with respect to   $B$ and let us denote by  $\fr{z}=Z(\fr{k})$   the centre of $\fr{k}$.  We may identify $\fr{f}\simeq T_{eK}F$ and moreover we may assume that $G$ is simple, compact and simply connected, as we do from now on.  In these terms,   an indecomposable C-space  is  defined
   by  a  decomposition of the space $\fr{t}:=\fr{z}\cap i\fr{a}$, where $\fr{a}$ is a Cartan subalgebra of $\fr{k}$ (and hence also of $\fr{g}$),    into a  direct  sum     of a
 (commutative)  subalgebra  $\mathfrak{t}_0$  of even  dimension  $2s$, generating the torus ${\rm T}_{0}^{2s}\equiv {\rm T}^{2s}$ and   a complementary subalgebra  $\mathfrak{t}_1$,   which generates   a central subgroup  ${\rm T}_1\subset K$  with  $\dim_{\bb{R}} {\rm T}_{1}=v-2s$, that is  
 \[
 \fr{t}= \mathfrak{t}_0\oplus \mathfrak{t}_1\,.
 \]
   Then, $\rnk G=\rnk K=\dim {\rm T}^{v}+\rnk K'$, $\rnk H=\dim {\rm T}_{1}+\rnk K'$,  $H={\rm T}_1 \cdot K'\subset K$ is closed normal  subgroup  of $K$,   ${\rm T}_{1}\cap K'$ is finite and  $K'$ coincides  with the simisimple part of $H$.    
As a consequence,    the homogeneous manifold
 $M =  G/H := G/{\rm T}_1\cdot K'$ is    a C-space, and any indecomposable C-space is obtained in this way (see \cite{Wang}), i.e.  as  an even-dimensional  torus bundle over $F=G/K$
 \[
 {\rm T}^{2s}\cong\frac{{\rm T}^{v}\cdot K'}{{\rm T}_1\cdot K'}\cong \frac{{\rm T}^{v}}{{\rm T}_1}  \longrightarrow M =  \frac{G}{H} =\frac{G}{{\rm T}_1\cdot K'}\longrightarrow  F=\frac{G}{K}=\frac{G}{{\rm T}^{v}\cdot K'}\,.
 \]
Passing to the level of Lie algebras,  the reductive  decomposition of $M=G/H$ is expressed by 
\[
   \fr{g}= \fr{h}\oplus \fr{m}\,, \quad  \fr{h}=\fr{k}'\oplus i\fr{t}_1\,,\quad \fr{m}=\fr{f}\oplus i\mathfrak{t}_0\cong T_{eH}M\,.
\]

Note that any  complex  structure in $\mathfrak{t}_0$  together   with  an invariant  complex  structure  $J_{\fr{f}}$
in $F=G/K$  defines   an invariant  complex  structure  $J_{\fr{m}}$ in $M=G/H$  such  that   $\pi : M=G/H \to F = G/K$ is  a holomorphic  fibration with respect to the  complex  structures  $J_{\fr{m}}$  and  $J_{\fr{f}}$.   
We  also recall that
\begin{prop} \textnormal{(\cite{AC1})}
Let   $M =  G/H := G/ K'\cdot T_1$ be a C-space with $G$ simple. Then,  there is an isomorphism $H^{2}(M; \R)\simeq\fr{t}_{1}$ and hence the second  Betti number of $M$ is given by $b_{2}(M)=b_{2}(F)-2s=v-2s=\dim_{\R}\fr{t}_1$.
\end{prop}


\subsection{The classification of indecomposable non-K\"ahlerian C-spaces}
The classification of non-K\"ahlerian  C-spaces is based on the classification of K\"ahlerian C-spaces. The second one is given in  terms of {\it painted Dynkin diagrams} (PDD), see \cite{Forger, Ale} and see also \cite{AC1} for an explicit presentation of all flag manifolds.
Here we   use  the results of \cite{Forger, AC1} to present the complete classification of all non K\"ahlerian C-spaces $M=G/H$ of a compact simple Lie group $G$, or in other terms of all indecomposable   non-K\"ahlerian C-spaces.

So, from now on let us assume that $G$ is  compact and simple. We divide  C-spaces $M=G/H$ with $\rnk G>\rnk H$  into three types, depending on the type  of  the stability group $H:$
\begin{itemize}
\item $M$ is said to be of {\it semistrict type} when $\fr{t}_1$ is trivial, or equivalently  when   $H$  coincides with the semisimple part of $K$, i.e. $H=K'$ and $M=G/K'$;
\item  $M$ is said to be of {\it strict type} (or {\it pure  type}) when $\fr{t}_1\neq\emptyset$ is non-trivial, or equivalently  when $H$ is reductive, i.e. $H=K'\cdot{\rm T}_1$ and $M=G/K'\cdot{\rm T}_1$;
\item  $M$ is said to be of {\it abelian  type} when the semisimple part $\fr{k}'\subset\fr{k}$ is trivial, i.e. $H={\rm T}_1$ and $M=G/{\rm T}_1$. 
\end{itemize}
Note that in terms of \cite{Wang} a C-space of semistrict type is an {\it even-dimensional} M-space $G/H$, with $G$ simple.
Let us now state some direct conclusions.
 \begin{prop}\label{firstconcl}
   Let $G$ be a  compact simple Lie group. Then,\\
 1)  A C-space $M=G/H$ of semistrict type has trivial second Betti number, $b_{2}(M)=0$.\\
2) If the flag manifold $F=G/K'\cdot{\rm T}^{v}$ has second Betti number $b_{2}(F)=v=2$, then there is a unique non-K\"ahlerian C-space $M=G/H$ associated to $F$. In particular, $H=K'$ and $M$ is of semistrict type.\\
3)  If $G$ is  even dimensional, then it is an indecomposable non-K\"ahlerian C-space of semistrict type over the  corresponding full flag manifold $F=G/{\rm T}_{\rm max}$.\\
4) Non-K\"ahlerian C-spaces $M=G/H$   of abelian type, are induced only by full flag manifold $F=G/{\rm T}_{\rm max}$ with $b_{2}(F)=\rnk G\geq 3$.
\end{prop}

First we will describe the classification of indecomposable C-spaces associated to a flag manifold of  a classical Lie group. Hence,  let $G$ be one of $A_{\ell-1}=\SU(\ell)$, $B_{\ell}=\SO(2\ell+1)$,  $C_{\ell}=\Sp(\ell)$ and $D_{\ell}=\SO(2\ell)$.  Recall that there are 4 general families of classical flag manifolds given  by
 \begin{eqnarray*}
A(\ell_1, \ldots, \ell_p)&:=& \SU(\ell)/\Ss(\U(\ell_1) \times \cdots \times \U(\ell_p))\,,\\
B(\ell_1, \ldots, \ell_q, m)&:=&  \SO(2\ell+1)/\U(\ell_1)\times \cdots\times  \U(\ell_q) \times \SO(2m+1)\,,\\
C(\ell_1, \ldots, \ell_q, m)&:=& \Sp(\ell)/\U(\ell_1)\times \cdots\times  \U(\ell_q) \times \Sp(m)\,,\\
D(\ell_1, \ldots, \ell_q, m) &:=&  \SO(2\ell)/ \U(\ell_1)\times  \cdots \times \U(\ell_q) \times \SO(2m)\,,
\end{eqnarray*}
 where $\ell=\ell_1+\cdots+\ell_p$ with $\ell_1\geq\ell_2\geq\cdots\geq\ell_p\geq 1$ and $p\geq 1$ for $\SU(\ell)$ and  $\ell =\ell_1+\cdots+\ell_q+m$ with $\ell_1\geq\ell_2\geq\cdots\geq\ell_q\geq 1$ and  $q, m\geq 0$ for the other cases.  Then, with respect to the three types of indecomposable C-spaces specified above,  we obtain the following classification, which up to our knowledge is for first time presented  (the papers \cite{Ale, Forger} do not include the classification of non-K\"ahler C-spaces, while the classical paper \cite{Wang}  does not provide the full classification explicitly, and it has gaps, a fact which  is mentioned also in \cite{Pap}).

 \begin{theorem} \label{cspacesclas} Assume that $G$ is a  compact simple classical Lie group. \\
1) Let $M=G/H$ be a non-K\"ahlerian C-space of semistrict type associated to a flag manifold $F=G/K$ corresponding to $G$.  Then $M$ is diffeomorphic to one of the following homogeneous spaces:
 \[
 \begin{tabular}{l | l | l}
     $G$      & $M=G/H$ $\text{\bf semistrict type}$  & $\text{\sc conditions}$ \\
     \thickline
 $A_{\ell-1}$ & $\SU(\ell)/\SU(\ell_1)\times\cdots\times\SU(\ell_p)$    & $p\geq 3$ odd \\
$B_{\ell}$  & $\SO(2\ell+1)/\SU(\ell_1)\times\cdots\times\SU(\ell_q)\times\SO(2m+1)$    & $q\geq 2$ even\\
$C_{\ell}$  & $\Sp(\ell)/\SU(\ell_1)\times\cdots\times\SU(\ell_q)\times\Sp(m)$    & $q\geq 2$ even\\
$D_{\ell}$ & $\SO(2\ell)/\SU(\ell_1)\times\cdots\times\SU(\ell_q)\times\SO(2m)$    & $q\geq 2$ even\\
\hline
 \end{tabular}
 \] 
 \smallskip
 2) Let $M=G/H$ be a non-K\"ahlerian C-space of   strict type associated to a flag manifold $F=G/K$ corresponding to $G$.  Then $M$ is diffeomorphic to one of the following homogeneous spaces:
  \[
 \begin{tabular}{l | l | l}
     $G$      & $M=G/H$ $\text{\bf strict type}$  & $\text{\sc conditions}$ \\
     \thickline
 $A_{\ell-1}$ & $\SU(\ell)/\U(1)^{t}\cdot(\SU(\ell_1)\times\cdots\times\SU(\ell_p))$    & $p-t=\text{odd}$, $p\geq 4$, $0<t<p-1$ \\
$B_{\ell}$  & $\SO(2\ell+1)/\U(1)^{t}\cdot(\SU(\ell_1)\times\cdots\times\SU(\ell_q)\times\SO(2m+1))$    & $q-t=\text{even}$, $q\geq 3$, $0<t<q$ \\
$C_{\ell}$  & $\Sp(\ell)/\U(1)^{t}\cdot(\SU(\ell_1)\times\cdots\times\SU(\ell_q)\times\Sp(m))$    & $q-t=\text{even}$, $q\geq 3$, $0<t<q$\\
$D_{\ell}$ & $\SO(2\ell)/U(1)^{t}\cdot(\SU(\ell_1)\times\cdots\times\SU(\ell_q)\times\SO(2m))$    & $q-t=\text{even}$, $q\geq 3$, $0<t<q$ \\
\hline
 \end{tabular}
 \]
 \smallskip
 3) Let $M=G/H$ be a non-K\"ahlerian C-space of abelian type associated to a flag manifold $F=G/K$ corresponding to $G$.  Then $M$ is diffeomorphic to one of the following homogeneous spaces:
  \[
 \begin{tabular}{l | l | l}
     $G$      & $M=G/H$ $\text{\bf abelian  type}$  & $\text{\sc conditions}$ \\
     \thickline
 $A_{\ell-1}$ & $\SU(\ell)/\U(1)^{t}$    & $\ell-t=\text{odd}$, $\ell\geq 4$, $0<t<\ell-1$ \\
$B_{\ell}$  & $\SO(2\ell+1)/\U(1)^{t}$    & $\ell-t=\text{even}$, $\ell\geq 3$, $0<t<\ell$ \\
$C_{\ell}$  & $\Sp(\ell)/\U(1)^{t}$    & $\ell-t=\text{even}$, $\ell\geq 3$, $0<t<\ell$\\
$D_{\ell}$ & $\SO(2\ell)/U(1)^{t}$    & $\ell-t=\text{even}$, $\ell\geq 4$, $0<t<\ell$ \\
\hline
 \end{tabular}
 \]

 \end{theorem}
 \begin{proof}
 We  begin with $G=\SU(\ell)$ where  $\ell=\ell_1+\ldots+\ell_p$.  The stability group $K=\Ss(\U(\ell_1) \times \cdots \times \U(\ell_p))$ of the generalized flag manifold $F=A(\ell_1, \ldots, \ell_p)$ has rank  $\ell-1$ and  we have a diffeomorphism 
 \[
K= \Ss(\U(\ell_1) \times \cdots \times \U(\ell_p))	\cong\U(1)^{p-1}\cdot(\SU(\ell_1)\times\cdots\times\SU(\ell_p))\,.
 \]
 To produce a C-space $M=\SU(\ell)/H$ over $A(\ell_1, \ldots, \ell_p)$ of semistrict  type, we need $H=K'$ and hence all the  abelian part of $K$ must  be removed. In particular, $\SU(\ell)/\SU(\ell_1)\times\cdots\times\SU(\ell_p)$ is a ${\rm T}^{p-1}=\U(1)^{p-1}$ principal bundle over $A(\ell_1, \ldots, \ell_p)$ and we require  $p\geq 3$ odd,  in  order the rank of this torus bundle to be an even number $\geq 2$.    For  2)  assume that $M=\SU(\ell)/H$ is a C-space  of strict type. Then, the stabilizer $H$ must be reductive and hence $H=\U(1)^{t}\cdot K'$, for some $0<t<p-1$, where $K'=\SU(\ell_1)\times\cdots\times\SU(\ell_p)$ is the semisimple part  of $K$. Indeed,  we compute  $\rnk G-\rnk H=\ell-1-(t+\ell_1+\cdots+\ell_p-p)=p-t-1$,
 so we  get the following torus bundle
 \[
 {\rm T}^{p-t-1}\cong \U(1)^{p-1}/\U(1)^{t}\to \SU(\ell)/\U(1)^{t}\cdot K'\to \SU(\ell)/\U(1)^{p-1}\cdot K'\,,
 \]
where  $p-t-1$ must be an even number, i.e. $p-t=\text{odd}$.  If $t=p-1$, then we obtain trivial fiber and hence we must assume that $t<p-1$, while  the value $t=0$  induces a C-space of semistrict type. This finally gives the restriction $0<t<p-1$, while it is easy to see that for strict type we also need $p\geq 4$.   For the assertion in 3), by Proposition \ref{firstconcl}, non-K\"ahlerian C-space of abelian type appear as fibrations over full flag manifolds $F=G/{\rm T}_{\rm max}$ of a simple Lie group $G$ with rank $\rnk G=b_{2}(F)\geq 3$. For $G=\SU(\ell)$ we get the full flag manifold 	$\SU(\ell)/{\rm T}^{\ell-1}$, where ${\rm T}^{\ell-1}=\U(1)^{\ell-1}$ denotes a maximal torus of $G$. Since we want to construct a C-space $M=\SU(\ell)/H$ of abelian type, $H$ must be  abelian. Let us assume that $H=\U(1)^{t}$  for some $t$. This defines the fibration
\[
{\rm T}^{\ell-t-1}\cong \U(1)^{\ell-1}/\U(1)^{t}\to  \SU(\ell)/\U(1)^{t}\to \SU(\ell)/\U(1)^{\ell-1}\,,
\]
and we require $\ell-t=\text{odd}$ in order the fiber  to be even-dimensional.  Also, a direct computation gives the restrictions $\ell\geq 4$ and $0 <  t<\ell-1$.

 Let us now prove the assertions  for $G=\SO(2\ell+1)$ and the other cases are treated similarly.
We begin with 1). The stability group of the classical flag manifold $B(\ell_1, \ldots, \ell_q, m)$ corresponding to $\SO(2\ell+1)$ is the Lie group $K=\U(\ell_1)\times \cdots\times  \U(\ell_q) \times \SO(2m+1)$ and we have a diffeomorphism
\[
K=\U(\ell_1)\times \cdots\times  \U(\ell_q) \times \SO(2m+1)\cong\U(1)^{q}\cdot(\SU(\ell_1)\times \cdots\times  \SU(\ell_q) \times \SO(2m+1))\,,
\] 
where   $\ell=\ell_1+\ldots+\ell_q+m$. Consider a non-K\"ahlerian C-space $M=\SO(2\ell+1)/H$ of semistrict type. This means $H=K'$, in particular $\SO(2\ell+1)/\SU(\ell_1)\times \cdots\times  \SU(\ell_q) \times \SO(2m+1)$ is a ${\rm T}^{q}$-principal bundle over  $B(\ell_1, \ldots, \ell_q, m)$, where ${\rm T}^{q}=\U(1)^{q}$ and we require $q\geq 2$ such that ${\rm T}^{q}$ be an even dimensional torus.    For the assertion in 2), 
$M=\SO(2\ell+1)/H$ must be of strict type, so  we need to assume $H=\U(1)^{t}\cdot K'$ for some $t$, where $K'=\SU(\ell_1)\times\cdots\times\SU(\ell_q)\times\SO(2m+1)$ is the semisimple part  of $K$. Then we compute  $\rnk G-\rnk H=\ell-(t+\ell_1+\cdots+\ell_q-q+m)=q-t$,
 so we  get the following torus bundle
 \[
 {\rm T}^{q-t}\cong \U(1)^{q}/\U(1)^{t}\to \SO(2\ell+1)/\U(1)^{t}\cdot K'\to \SO(2\ell+1)/\U(1)^{q}\cdot K'\,.
 \]
Here, we require $q-t=\text{even}$,  $q\geq 3$ (since the value $q=2$ produces  a C-space of semistrict type) and $0<t<q$.   For 3) let $\SO(2\ell+1)/{\rm T}^{\ell}$ be the full flag manifold  of $\SO(2\ell+1)$. Consider the coset $M=\SO(2\ell+1)/\U(1)^{t}$ for some $t<\ell$. This induces the fibration
\[
{\rm T}^{\ell-t}\cong \U(1)^{\ell}/\U(1)^{t}\to  \SO(2\ell+1)/\U(1)^{t}\to \SO(2\ell+1)/\U(1)^{\ell}\,,
\]
so we ask $\ell-t=\text{even}$. Also, to get abelian type we need $\ell\geq 3$ and $0<t<\ell$.  
 \end{proof}

When $G$ is a compact exceptional simple Lie group, then we will refer to the base space $F=G/K$ by the term {\it exceptional flag manifold}. For non-K\"ahler C-spaces fibered over such flag manifolds we see that
\begin{prop}
1) All the associated non-K\"ahlerian C-spaces $M=G/H$ to an exceptional flag manifold $F=G/K'\cdot{\rm T}^{v}$ with    second Betti number $b_{2}(F)=v\geq 3$ odd, are of strict type.\\
2) An exceptional flag manifold  $F=G/K'\cdot{\rm T}^{v}$ with  second Betti number $b_{2}(F)=v\geq 4$ even, induces a  unique non-K\"ahlerian C-space of semistrict type and at least one  non-K\"ahlerian C-space of strict type. 
\end{prop}
 Based on this proposition and on the classification of exceptional flag manifolds, one can obtain the full classification of  all indecomposable non-K\"ahler C-spaces  associated to an exceptional Lie group. For convenience of the reader, we present this  classification in  Table \ref{Table6} in  Appendix \ref{sectionapen}.


\section{The indecomposable  C-space $\SU(\ell+m+n)/\SU(\ell)\times\SU(m)\times\SU(n)$}\label{section3}
In this section we examine a certain indecomposable C-space of semistrict type, namely   the homogeneous space  
\[
M=G/H=\SU(\ell+m+n)/\SU(\ell)\times\SU(m)\times\SU(m)
\]
 with $N=\ell+m+n\geq 3$, where $\ell, m, n\in\bb{Z}_{+}$ are positive integers. In fact, this is the first induced example from the family $\SU(\ell)/\SU(\ell_1)\times\cdots\times\SU(\ell_p)$  in Theorem \ref{cspacesclas}.  The  embedding  $\SU(\ell)\times\SU(m)\times\SU(m)\to \SU(\ell+m+n)$ is the diagonal one, and in terms of Lie algebras it is given by
\[
\fr{h}=\fr{su}(\ell)\oplus\fr{su}(m)\oplus\fr{su}(n) =\Big\{  
\begin{pmatrix}
A_1 & 0 & 0\\
0 & A_2 & 0\\
0 & 0 & A_3
\end{pmatrix}
: A_1\in\fr{su}(\ell), A_2\in\fr{su}(m), A_3\in\fr{su}(m)
\Big\}\subset \fr{su}(N).
\]
 The coset $M=G/H$  is a  torus bundle over  the flag manifold  $F=G/K=\SU(\ell+m+n)/\Ss(\U(\ell)\times\U(m)\times\U(n))$, 
\[
{\rm T}^{2}\simeq K/H \longrightarrow G/H=\SU(N)/\SU(\ell)\times\SU(m)\times\SU(n)\overset{\pi}{\longrightarrow}G/K=\SU(N)/\Ss(\U(\ell)\times\U(m)\times\U(n))\,.
\]
We have $K={\rm T}^{2}\cdot H$, where  the stabilizer $H$ is semisimple and coincides with the semisimple part  $K'$ of $K$, i.e. $
K'=H=\SU(\ell)\times\SU(m)\times\SU(n)$. 
 Note that   $F=G/K$ has $b_{2}(F)=2$. 

\subsection{A reductive decomposition and equivalent summands}
Let us describe a $B$-orthogonal reductive decomposition of $M$, where from now on we set $B=-B_{\fr{g}}$, where
\[
B_{\fr{g}}(X, Y)=2N\tr(XY)=2(\ell+m+n)\tr(XY)
\] 
is the Killing form of $\fr{g}=\fr{su}(N)$.
  We will denote by $\fr{h}\subset\fr{k}\subset\fr{g}$ the Lie algebras of $H\subset K\subset G$, respectively, and by $\fr{g}=\fr{k}\oplus\fr{f}$ a $B$-orthogonal reductive decomposition of the pair $(\fr{g}, \fr{k})$, associated to the base space $F=G/K$.  This   flag manifold corresponds to the  painted Dynkin diagram 
  \[
  \begin{picture}(400, 0)(-8, 10)
  \put(83, 0.5){\circle{5}}
\put(83, 5){\line(0,3){10}}
\put(83, 15){\line(1,0){11}}
\put(107, 15){\makebox(0,0){{\tiny{$\ell-1$}}}}
\put(119, 15){\line(1,0){12}}
\put(131, 15){\line(0,-3){10}}
 \put(85, 0.5){\line(1,0){14}}
\put(97, 0){ $\ldots$}
\put(115, 0.5){\line(1,0){14}}
\put(131, 0.5){\circle{5}}
\put(133, 0.5){\line(1,0){14}}
\put(149, 0.5){\circle*{5}}
 \put(149,10){\makebox(0,0){$\al_{\ell}$}}
  \put(150.9, 0.5){\line(1,0){13}}
  \put(166, 0.5){\circle{5}}
  \put(166, 5){\line(0,3){10}}
\put(166, 15){\line(1,0){11}}
\put(190, 15){\makebox(0,0){{\tiny{$m-1$}}}}
\put(203, 15){\line(1,0){13}}
\put(216, 15){\line(0,-3){10}}
    \put(168, 0.5){\line(1,0){14}}
     \put(180, 0){ $\ldots$}
  \put(200, 0.5){\line(1,0){13.6}}
   \put(216, 0.5){\circle{5}}
    \put(218, 0.5){\line(1,0){13.6}}
\put(232, 0.5){\circle*{5}}
 \put(232,10){\makebox(0,0){$\al_{\ell+m}$}}
  \put(233.9, 0.5){\line(1,0){13}}
  \put(249, 0.5){\circle{5}}
   \put(249, 5){\line(0,3){10}}
\put(249, 15){\line(1,0){11}}
\put(273.6, 15){\makebox(0,0){{\tiny{$n-1$}}}}
\put(287, 15){\line(1,0){12}}
\put(299, 15){\line(0,-3){10}}
    \put(251, 0.5){\line(1,0){13.6}}
  \put(262.5, 0){ $\ldots$}
  \put(282.5, 0.5){\line(1,0){13.6}}
   \put(299, 0.5){\circle{5}}
           \end{picture} 
 \]
 \vskip 0.5cm
\noindent and  its T-root system  is given by $\{\overline{\al}_{\ell}, \overline{\al}_{\ell+m}, \overline{\al}_{\ell}+\overline{\al}_{\ell+m}\}$, where here we denote by $\overline{\al}$ the projection of a complementary root $\al\in R_{F}=R\backslash R_{K}$  of $F$ to the 2-dimensional center of $\fr{k}$ (see \cite{Chry1, Graev}  for details on $T$-roots and painted Dynkin diagrams).  Hence, the $\Ad(K)$-module $\fr{f}\simeq T_{eK}F$ decomposes into three inequivalent and irreducible $\Ad(K)$-submodules $\fr{f}_{\ell m}, \fr{f}_{\ell n}$ and  $\fr{f}_{mn}$, of real dimensions $2\ell m$, $2\ell n$ and $2mn$,  respectively.  So we get a  $B$-orthogonal decomposition 
\[
\fr{f}=\fr{f}_{\ell m}\oplus\fr{f}_{\ell n}\oplus\fr{f}_{mn}
\]
 where (see also \cite{mori})
 \begin{eqnarray*}
\fr{f}_{\ell m} &=&
\Big\{  
\begin{pmatrix}
0 & A & 0\\
-\bar{A^t} & 0 & 0\\
0 & 0 & 0
\end{pmatrix}
: A\in M_{\ell , m}(\mathbb{C})
\Big\}\,,\quad
\fr{f}_{\ell n} =
\Big\{  
\begin{pmatrix}
0 & 0 & B\\
 0 & 0 & 0\\
 -\bar{B^t} & 0 & 0
\end{pmatrix}
: B\in M_{\ell, n}(\mathbb{C})
\Big\}\,,\\
\fr{f}_{mn} &=&
\Big\{  
\begin{pmatrix}
0 & 0 & 0\\
0 & 0 & C\\
0 & -\bar{C^t} & 0
\end{pmatrix}
: C\in M_{m, n}(\mathbb{C})
\Big\}\,.
\end{eqnarray*}
For simplicity we shall write
\[
\fr{g}=\fr{k}\oplus\fr{f}\,,\quad\fr{f}=\fr{f}_1\oplus\fr{f}_2\oplus\fr{f}_3\,,\quad \fr{f}_1:=\fr{f}_{\ell m}\,, \  \fr{f}_2:=\fr{f}_{\ell n}\,,\ \fr{f}_3:=\fr{f}_{mn}\,.
\]
We also denote  by $\fr{f}_{0}$   the tangent space of the fibre $\bb{T}^2$, 
\[
T_{e}\bb{T}^{2}\simeq \fr{f}_0=\Big\{
\begin{pmatrix}
\frac{\sqrt{-1}a_1}{\ell}I_\ell & 0 & 0\\
0 & \frac{\sqrt{-1}a_2}{m}I_m & 0\\
0 & 0 & \frac{\sqrt{-1}a_3}{n}I_n
\end{pmatrix}
: a_1+a_2+a_3=0, \ a_1, a_2, a_3 \in\mathbb{R}
\Big\}\,.
\]
Obviously, $\fr{f}_{0}$ coincides with the central part of $\fr{k}$ and we have $\fr{k}=\fr{h}\oplus\fr{f}_{0}$.   For the pair   $(\fr{g}, \fr{h})$ related with the C-space $M=G/H$, we get the reductive decomposition 
\[
\fr{g}=\fr{h}\oplus\fr{m}\,, \quad \fr{m}=\fr{f}\oplus\fr{f}_{0}\,, \quad [\fr{h}, \fr{m}]\subset\fr{m}\,,
\]
where   we identify   the $\Ad(H)$-module $\fr{m}$ with the tangent space $T_{eH}M$.  Note that 
\[
\dim_{\R}\fr{m}=d_1+d_2+d_3+2=2(\ell m+\ell n+mn+1)\,, \quad d_1:=2\ell m\,,\quad d_2:=2\ell n\,,\quad d_3:=2mn\,.
\]

Observe now that  the  $\Ad(H)$-submodule $\fr{f}_{0}$ is not irreducible;  Set
\begin{eqnarray*}
Z_4:=\sqrt{-1}
\begin{pmatrix}
\frac{1}{\ell+m}I_\ell & 0 & 0\\
0 & \frac{1}{\ell+m}I_m & 0\\
0 & 0 & -\frac{1}{n}I_n
\end{pmatrix},\ \ \mbox{and}\ \ 
Z_5:= \sqrt{-1}
\begin{pmatrix}
\frac{1}{\ell}I_\ell & 0 & 0\\
0 & -\frac{1}{m}I_m & 0\\
0 & 0 & 0
\end{pmatrix}.
\end{eqnarray*} 
Then, we obtain the   $B$-orthogonal decomposition   $\fr{f}_0 = \fr{f}_4\oplus\fr{f}_5$, 
where $ \fr{f}_4= \Span\{Z_4\}$ and $ \fr{f}_5 = \Span\{Z_5\}$ are the 1-dimensional $\Ad(K)$-invariant subspaces of   $\fr{f}_{0}$,  generated by $Z_4, Z_5$, respectively.  Therefore,  the reductive complement $\fr{m}=T_{eH}M$ splits as follows
\[
\fr{m}=\fr{f}\oplus\fr{f}_{0}=\fr{f}_1\oplus\fr{f}_2\oplus\fr{f}_3\oplus \fr{f}_4\oplus\fr{f}_5\,.
\]
\begin{lemma}\label{thetalem} \textnormal{($\theta$-Lemma)}
The $B$-orthogonal $\Ad(H)$-invariant  decomposition  $\fr{m}=\fr{f}_1\oplus\fr{f}_2\oplus\fr{f}_3\oplus \fr{f}_4\oplus\fr{f}_5$ is not   unique.
\end{lemma}
\begin{proof}
The isotropy representation acts trivially on the tangent space of the fiber, $\Ad(H)|_{\fr{f}_{0}}=\Id$. Thus, the modules $\fr{f}_4, \fr{f}_5$ can be replaced by any pair of orthogonal 1-dimensional submodules of  $\fr{f}_{0}$. For example, set
\[
Z^{\theta}_{4}=\cos(\theta)Z_{4}+\sin(\theta)Z_{5}\,,\quad Z^{\theta}_{5}=-\sin(\theta)Z_{4}-\cos(\theta)Z_{5}\,,
\]
for some  $\theta\in[0, 2\pi]$ and moreover $\fr{f}_4^{\theta}=\Span\{Z_{4}^{\theta}\}$ and $\fr{f}_5^{\theta}=\Span\{Z_{5}^{\theta}\}$, respectively. Then, it is easy to see that $\fr{f}_{0}^{\theta}:=\fr{f}_4^{\theta}\oplus\fr{f}_5^{\theta}\simeq\fr{f}_{0}=\fr{f}_4\oplus\fr{f}_5$ are equivalent as $\Ad(H)$-modules. 
\end{proof}
\begin{remark}
\textnormal{This lemma is helpful  for realizing  the equivalence of the modules  $\fr{f}_4$ and $\fr{f}_5$. However,  for solving explicitly the homogeneous Einstein equation  we find more useful to  parametrize the scalar products on $\fr{f}_{0}$ in terms of nilpotent matrices, instead of some rotation defined by $\theta$ (see Proposition \ref{parametrize}).}
\end{remark}


\subsection{The $B$-structure constants}
Let us now compute the $B$-structures constants of $M=G/H$.  Consider the decomposition 
 \begin{equation}\label{ciso3}
 \fr{g}=\fr{h}\oplus\fr{m}=\fr{h}\oplus\fr{f}_1\oplus\fr{f}_2\oplus\fr{f}_3\oplus\fr{f}_0=\fr{h}\oplus \fr{f}_1\oplus\fr{f}_2\oplus\fr{f}_3\oplus\fr{f}_4\oplus\fr{f}_5\,,
 \end{equation}
and as before,  set $\fr{k}=\fr{h}\oplus\fr{f}_0$. Based on the corresponding $T$-roots of $F=G/K$, we see that

\begin{lemma}\label{brackets}  
The $\Ad(H)$-invariant submodules $\fr{f}_{i}$ $(i=1,\ldots, 5)$ in the decomposition  $(\ref{ciso3})$ are such that 
\[
[\fr{f}_i, \fr{f}_{i}]\subset\fr{k}\,,\quad\forall \ i=1,2, 3,\quad  [\fr{f}_{j}, \fr{f}_{k}]\subset\fr{f}_{k}\,,\quad\forall \ j=4, 5, \  k=1, 2, 3,
\]
and $[\fr{f}_1, \fr{f}_2]\subset\fr{f}_3$, $[\fr{f}_2, \fr{f}_3]\subset\fr{f}_1$,    $[\fr{f}_1, \fr{f}_3]\subset\fr{f}_2$, $[\fr{f}_4, \fr{f}_5]=0$.
   \end{lemma}

As an immediate consequence of  Lemma \ref{brackets} we get that
\begin{corol}\label{cor1}
The non-zero $B$-structure constants of $M_{\ell, m, n}=G/H$ with respect to {\rm (\ref{ciso3})} are listed as follows
\[
\displaystyle{3 \brack {1 \ 2}}\,,\quad \displaystyle{4 \brack {1 \ 1}}\,,\quad \displaystyle{4 \brack {2 \ 2}}\,,\quad \displaystyle{4 \brack {3 \ 3}}\,,\quad \displaystyle{5 \brack {1 \ 1}}\,,\quad \displaystyle{ 5 \brack {2 \ 2}}\,,\quad \displaystyle{5 \brack {3 \ 3}}\,.
\]
\end{corol}
 Let us now recall a   result  which can be useful for the computation of these unknowns.  
\begin{lemma}\label{lemma5.2aa}  
Consider  a simple subalgebra   $\fr{q}\subset \fr{g}=\fr{su}(N)=\fr{su}(\ell+m+n)$ and  let us denote by $\{\e_i\}_{1\leq i\leq \dim \fr{q}}$   a $B$-orthonormal basis of $\fr{q}$, where $B=-B_{\fr{g}}$.  Then,  
\[
\sum_{j, k = 1}^{\dim \fr{q}} \big(B([\e_i, \e_j], \e_k )\big)^2  = {\al}^{\fr{q}}_{\fr{g}}\,,\quad\quad \sum_{i, j, k = 1}^{\dim \fr{q}} \big(B([\e_i, \e_j], \e_k )\big)^2  = \dim\fr{q} \cdot {\al}^{\fr{q}}_{\fr{g}}\,,
\]
where ${\al}^{\fr{q}}_{\fr{g}}$ is the constant determined  by $B_{\fr{q}} = {\al}^{\fr{q}}_{\fr{g}}\cdot B|_{\fr{q}}$, with   
 $B_{\fr{q}}$ being the Killing form of $\fr{q}$. 
\end{lemma} 
\noindent For a short proof of this simple formula (for  the Lie algebra $\fr{g}$ of a more general compact simple Lie group $G$) we refer to \cite{ADN1}.  

\begin{prop}\label{ijk}
The non-zero $B$-structure constants of  the C-space $M_{\ell, m, n}=\SU(N)/\SU(\ell)\times\SU(m)\times\SU(n)$ ($N=\ell+m+n$),  attain the following values
{\small
\begin{equation*}\label{eq14}
\begin{array}{llll} 
  \vspace{0.3cm}
\displaystyle{{3 \brack {1 \ 2}} =  \frac{\ell m n}{N} }\,,   
&  \displaystyle{{4 \brack {1 \ 1}} = 0}\,,
  &  \displaystyle{{ 4 \brack {2 \ 2}} = \frac{\ell}{\ell+m}  }\,,
  &  \displaystyle{{4 \brack {3 \ 3}} = \frac{m}{\ell+m}   }\,, \\ 
  \vspace{0.3cm}
    \displaystyle{{5 \brack {1 \ 1}} =  \frac{\ell+m}{N} }\,, 
  & \displaystyle{{5 \brack {2 \ 2}} = \frac{mn}{N(\ell+m)}}\,, 
  &  \displaystyle{{5 \brack {3 \ 3}} =  \frac{\ell n}{N (\ell+m)} }\,.  
\end{array} 
\end{equation*}
}
\end{prop} 
\begin{proof}
Our method relies  on a combination of  Lemma \ref{lemma5.2aa} with other dimensional identities,  which we obtain by using  the structure constants of the whole isometry algebra $\fr{g}=\fr{su}(\ell+m+n)$. Hence  is useful  to split the stability algebra $\fr{h}$ into its three simple ideals  
\[
\fr{h}=\fr{h}_{\upa}\oplus\fr{h}_{\upb}\oplus\fr{h}_{\upg}\,,\quad \fr{h}_{\upa}=\fr{su}(\ell)\,,\quad  \fr{h}_{\upb}=\fr{su}(m)\,,\quad  \fr{h}_{\upg}=\fr{su}(n)\,,
\]
and rewrite the reductive decomposition  (\ref{ciso3}) as
\begin{equation}\label{giso}
\fr{g}=\fr{h}\oplus\fr{m}=\fr{h}_{\upa}\oplus\fr{h}_{\upb}\oplus\fr{h}_{\upg}\oplus\fr{f}_1\oplus\fr{f}_2\oplus\fr{f}_3\oplus\fr{f}_4\oplus\fr{f}_5\,.
\end{equation}
Then,   the previous definition  of the  $B$-structure constants $\displaystyle{k \brack {i \ j}}$ on $\fr{m}$ with respect to (\ref{iso}), naturally extends to the whole Lie algebra $\fr{g}$  with respect to (\ref{giso}), see for example \cite[pp.~154-155]{CS2}. Moreover, this induces the following relations
\begin{eqnarray*}
d_{\updelta}:=\dim_{\R}\fr{h}_{\updelta}=\sum_{j, k}\displaystyle{k \brack {\updelta \ j}}\,,\quad \text{with} \ \ j, k\in\{\upa, \upb, \upg, 1, \ldots, 5\}, \ \ \text{for  any} \ \ \delta\in\{\upa, \upb, \upg\}\,,\\
d_{i}:=\dim_{\R}\fr{f}_{i}=\sum_{j, k}\displaystyle{k \brack {i \  j}}\,,\quad \text{with} \ \ j, k\in\{\upa, \upb, \upg, 1, \ldots, 5\}, \ \ \text{for  any} \ \ i\in\{1,\ldots, 5\}\,.
\end{eqnarray*}
It is also useful to indicate the  Lie bracket relations  related with the isotropy action of the ideals $\fr{h}_{\upa}, \fr{h}_{\upb}$ and $\fr{h}_{\upg}$. These are  listed as follows:
\[
\begin{tabular}{ l | l | l || l  l }
\hline
$[\fr{h}_{\upa}, \fr{f}_1]\subset\fr{f}_1\,,$ &  $[\fr{h}_{\upa}, \fr{f}_2]\subset\fr{f}_2\,,$ & $[\fr{h}_{\upa}, \fr{f}_3]=0\,,$ & $[\fr{h}_{\updelta}, \fr{h}_{\updelta}]\subset\fr{h}_{\updelta}$\,, &   $\forall \updelta\in\{\upa, \upb, \upg\}$\,,\\
$[\fr{h}_{\upb}, \fr{f}_1]\subset\fr{f}_1\,,$ &  $[\fr{h}_{\upb}, \fr{f}_2]=0\,,$ & $[\fr{h}_{\upb}, \fr{f}_3]\subset\fr{f}_3\,,$ &  $[\fr{h}_{\updelta}, \fr{h}_{\upepsilon}]=0$\,, &  $\forall  \updelta\neq\upepsilon\in\{\upa, \upb, \upg\}$\,,\\
$[\fr{h}_{\upg}, \fr{f}_1]=0\,,$ &  $[\fr{h}_{\upg}, \fr{f}_2]\subset\fr{f}_2\,,$ & $[\fr{h}_{\upg}, \fr{f}_3]\subset\fr{f}_3\,,$ &  $[\fr{h}_{\updelta}, \fr{f}_{i}]=0$\,, & $\forall \updelta\in\{\upa, \upb, \upg\},\ i\in\{4, 5\}$\,.\\
\hline
\end{tabular}
\]
 Hence,   the non-zero structure constants of the whole Lie group $G=\SU(\ell+m+n)$, with respect to (\ref{giso}), are exhausted by the triples in Corollary \ref{cor1}, together with the following one
 \[
 \displaystyle{\upa \brack {\upa \  \upa}},\quad  \displaystyle{\upb \brack {\upb \  \upb}},\quad  \displaystyle{\upg \brack {\upg \ \upg}},\quad   \displaystyle{\upa \brack {1 \ 1}}\,,\quad  \displaystyle{\upa \brack {2 \ 2}},\quad  \displaystyle{\upb \brack {1 \ 1}},\quad  \displaystyle{\upb \brack {3 \ 3}},\quad  \displaystyle{\upg \brack {2 \ 2}},\quad  \displaystyle{\upg \brack {3 \ 3}}\,. 
  \]
  The computation of the first three is an easy task;  Set   $\fr{q}:=\fr{h}_{\upa}\equiv\fr{su}(\ell)$ and recall that  $B_{\fr g}(X, Y)=2(\ell+m+n)\tr (XY)$ and $ B_{\fr q}=2\ell\tr (XY)$.  Let $\{\e_j\}$ be an orthonormal basis of $\fr q$ with respect to $-B_{\fr g}$ with $1\le j\le \ell^2-1$.
Then 
\[
{\al}^{\fr{h}_{\upa}}_{\fr{g}}\equiv {\al}^{\fr{su}(\ell)}_{\fr{g}}=\frac{B_{\fr{q}}}{\left.B_{\fr g}\right|_{\fr{q}}}=\frac{\ell}{\ell+m+n}=\frac{\ell}{N}
\]
and similarly   ${\al}^{\fr{h}_{\upb}}_{\fr{g}}\equiv{\al}^{\fr{su}(m)}_{\fr{g}}=\frac{m}{N}$ and  ${\al}^{\fr{h}_{\upg}}_{\fr{g}}\equiv{\al}^{\fr{su}(n)}_{\fr{g}}=\frac{n}{N}$. 
 Therefore,  
\begin{equation}\label{aaabbbccc}
\displaystyle{{\upa \brack {\upa \upa}}}=\sum _{i=1}^{\ell^2-1}\sum _{j, k=1}^{\dim\fr{h}_{\upa}} B_\fr{g}([\e_i, \e_j], \e_k)^2
=\dim_{\R}\fr{h}_{\upa}\cdot {\al}^{\fr{h}_{\upa}}_{\fr{g}}=\frac{\ell(\ell^2-1)}{N}\,,
\end{equation}
and by  repeating the same method we also obtain
\begin{equation}\label{bg}
 \displaystyle{{\upb \brack {\upb \upb}}}=\frac{m(m^2-1)}{N}\,,\quad  \displaystyle{{\upg \brack {\upg \upg}}}=\frac{n(n^2-1)}{N}\,.
\end{equation}
Therefore, we result with the following identities:
\begin{equation}\label{rel1}
\left\{\begin{tabular}{l}
$d_{\upa}=d_{\fr{su}(\ell)}=\ell^{2}-1=\sum_{j, k}\displaystyle{k \brack {\upa \ j}}=\displaystyle{{\upa \brack {\upa \ \upa}}}+\displaystyle{{1 \brack {\upa \ 1}}}+\displaystyle{{2 \brack {\upa \ 2}}}=\frac{\ell(\ell^2-1)}{N}+\displaystyle{{1 \brack {\upa \ 1}}}+\displaystyle{{2 \brack {\upa \ 2}}}\,,$\\\\
$d_{\upb}=d_{\fr{su}(m)}=m^{2}-1=\sum_{j, k}\displaystyle{k \brack {\upb \ j}}=\displaystyle{{\upb \brack {\upb \upb}}}+\displaystyle{{1 \brack {\upb \ 1}}}+\displaystyle{{3 \brack {\upb \ 3}}}=\frac{m(m^2-1)}{N}+\displaystyle{{1 \brack {\upb \ 1}}}+\displaystyle{{3 \brack {\upb \ 3}}}\,,$\\\\
$d_{\upg}=d_{\fr{su}(n)}=n^{2}-1=\sum_{j, k}\displaystyle{k \brack {\upg \  j}}=\displaystyle{{\upg \brack {\upg \ \upg}}}+\displaystyle{{2 \brack {\upg \ 2}}}+\displaystyle{{3 \brack {\upg \ 3}}}=\frac{n(n^2-1)}{N}+\displaystyle{{2 \brack {\upg \ 2}}}+\displaystyle{{3 \brack {\upg \ 3}}}\,.$
\end{tabular}\right.
\end{equation}
Now, based on the expression of the generator $Z_4$ of $\fr{f}_4$ and the matrices spanning $\fr{f}_1$, a direct computation  yields that $ \displaystyle{1 \brack {4 \ 1}}=0$.
Set  $\fr q=\fr{su}(\ell+m)$ and let  $\{\e_j\}$ be an orthonormal basis of $\fr q$ with respect to $-B_{\fr g}$.
We may split  $\fr{su}(\ell+m)$ as 
\begin{equation}\label{sulm}
\fr{su}(\ell+m)=\fr{su}(\ell)\oplus\fr{f}_{1}\oplus\fr{su}(m)\oplus\fr{f}_5=\fr{h}_{\upa}\oplus\fr{f}_1\oplus\fr{h}_{\upb}\oplus\fr{f}_5\,,
\end{equation}
 and   adapt the basis $\{\e_j\}$ to this decomposition as follows:
$\{\e_1, \dots , \e_{\ell^2-1}\}\in\fr{su}(\ell)$, $\{\e_{\ell^2}, \dots , \e_{(\ell^2-1)+2\ell m}\}\in\fr{f}_{1}$,
$\{\e_{\ell^2+2\ell m}, \dots , \e_{\ell^2+2\ell m+m^2-2}\}\in\fr{su}(m)$, $\e_{(\ell+m)^2-1}\in\fr{f}_5$.
Then
$$
\sum _{j, k=1}^{(\ell+m)^2-1}B_\fr{g}([\e_i, \e_j], \e_k)^2
={\al}^{\fr{su}(\ell+m)}_{\fr{g}}=\frac{\ell+m}{N}\,,
$$
and thus $\displaystyle{{1 \brack {5 \ 1}}=\frac{\ell+m}{N}}$.  Moreover, for any $\{\e_i: i=1, \dots , \ell^2-1\}\in\fr{su}(\ell)$ we have
\begin{equation}\label{rel2}
\sum _{i=1}^{\ell^2-1}\left(
\sum _{j, k=1}^{(\ell+m)^2-1}B_\fr{g}([\e_i, \e_j], \e_k)^2\right)=(\ell^2-1)\frac{(\ell+m)}{N}\,.
\end{equation}
Recall now that 
\[
[\e_i, \e_j]\in\left\{
\begin{tabular}{l l}
$\fr{h}_{\upa}\equiv\fr{su}(\ell)\,, \ \ $ & $\text{if}\ \ \e_j\in\fr {su}(\ell)\,,$\\
$\fr{f}_1\,, \ \ $ & $\text{if} \ \ \e_j\in\fr{f}_1\,,$\\
$0\,, \ \ $ & $\text{if}\  \ \e_j\in\fr{h}_{\upb}\equiv\fr{su}(m)\,.$
\end{tabular}\right.
\]
where   $\{\e_i: i=1, \dots , \ell^2-1\}\in\fr{su}(\ell)$ and $\{\e_j: j=1, \dots , m^2-1\}\in\fr{su}(m)$, respectively. 
Therefore,  the result in (\ref{rel2}) can be rephrased as
\[
\displaystyle{{\upa \brack {\upa \  \upa}}}+\displaystyle{{\upa \brack {1 \ 1}}}+0+0=\frac{(\ell+m)}{N}(\ell^2-1)\,,
\]
and by replacing the value of $\displaystyle{{\upa \brack {\upa \ \upa}}}$  given in (\ref{aaabbbccc})  we obtain $\displaystyle{{\upa \brack {1 \ 1}}}=\frac{m(\ell^2-1)}{N}$. 
Returning back to the first relation of (\ref{rel1}) and by substituting the   values of $\displaystyle{{\upa \brack {\upa \ \upa}}}$ and $\displaystyle{{\upa \brack {1 \ 1}}}$, we also get $
\displaystyle{{\upa \brack {2 \ 2}}}=\frac{n(\ell^2-1)}{N}$.   By reversing the role of $\fr{su}(\ell)$ and $\fr{su}(m)$ in (\ref{sulm}) and considering again an adapted basis starting this time by $\fr{h}_{\upb}=\fr{su}(m)$, we further obtain
\[
\displaystyle{{\upb \brack {\upb \  \upb}}}+\displaystyle{{\upb  \brack {1 \ 1}}}+0+0=\frac{(\ell+m)}{N}(m^2-1)\,,
\]
and in combination with  (\ref{bg}) we deduce that $\displaystyle{{\upb \brack {1 \ 1}}}=\frac{\ell(m^2-1)}{N}$.  After that, the  second relation in (\ref{rel1}) yields   the value  $\displaystyle{{\upb \brack {3 \ 3}}}=\frac{n(m^2-1)}{N}$.
 For the computation of  the triple $\displaystyle{{2 \brack {\upg \ 2}}}=\displaystyle{{\upg \brack {2 \ 2}}}$,  fix the subalgebra $\fr{q}=\fr{su}(m+n)$
and the decomposition
\[
\fr{su}(m+n)=\fr{su}(m)\oplus\fr{f}_3\oplus\fr{su}(n)\oplus\fr{f}_5=\fr{h}_{\upb}\oplus\fr{f}_3\oplus\fr{h}_{\upg}\oplus\fr{d} \,,
\]
where $\fr{d}\subset\fr{f}_{0}$ is some 1-dimensional space. Consider   an orthonormal basis  $\{\e_j\}$ of $\fr q=\fr{su}(m+n)$ with respect to $-B_{\fr g}$.  Then, by   adapting  $\{\e_j\}$ to  this decomposition as  
$\{\e_1, \dots , \e_{m^2-1}\}\in\fr{su}(m)$, $\{\e_{m^2}, \dots , \e_{(m^2-1)+2mn}\}\in\fr{f}_{3}$,
$\{\e_{m^2+2mn}, \dots , \e_{m^2+2mn+n^2-2}\}\in\fr{su}(n)$, $\e_{(m+n)^2-1}\in\fr{d}$, and by repeating the above technique we compute
\[
\displaystyle{{\upg \brack {2 \ 2}}}=  \frac{\ell(n^2-1)}{N}\,,\quad \displaystyle{{\upg \brack {3 \ 3}}}=\frac{m(n^2-1)}{N}\,.
\]
Now we can consider the relations
\begin{eqnarray*}
d_1&=&2\ell  m=\sum_{j, k}\displaystyle{k \brack {1 \ j}} = 2\Big(\displaystyle{\upa \brack {1 \ 1}}+\displaystyle{\upb \brack {1 \ 1}}+\displaystyle{3 \brack {1 \ 2}}+\displaystyle{4 \brack {1 \ 1}}+\displaystyle{5 \brack {1 \ 1}}\Big)\,,\\
d_2&=&2\ell  n=\sum_{j, k}\displaystyle{k \brack {2 \ j}} = 2\Big(\displaystyle{\upa \brack {2 \ 2}}+\displaystyle{\upg \brack {2 \ 2}}+\displaystyle{3 \brack {2 \ 1}}+\displaystyle{4 \brack {2 \ 2}}+\displaystyle{5 \brack {2 \ 2}}\Big)\,,\\
d_3&=&2m n=\sum_{j, k}\displaystyle{k \brack {3 \ j}} = 2\Big(\displaystyle{\upb \brack {3 \ 3}}+\displaystyle{\upg \brack {3 \ 3}}+\displaystyle{2 \brack {3 \ 1}}+\displaystyle{4 \brack {3 \ 3}}+\displaystyle{5 \brack {3 \ 3}}\Big)\,,\\
d_4&=&1=\sum_{j, k}\displaystyle{k \brack {4 \ j}} =  \displaystyle{1 \brack {4 \ 1}}+\displaystyle{2 \brack {4 \ 2}}+\displaystyle{3 \brack {4\ 3}}\,,\\
d_5&=&1=\sum_{j, k}\displaystyle{k \brack {5 \ j}} =  \displaystyle{1 \brack {5 \ 1}}+\displaystyle{2 \brack {5 \ 2}}+\displaystyle{3 \brack {5 \ 3}}\,.
\end{eqnarray*}
After replacing the structures constants computed above, the first relation yields  $\displaystyle{{3 \brack {1 \ 2}}= \frac{\ell m n}{N}}$; This is the unique non-zero structure constant of the base space $F=G/K$ (see for example \cite{mori}).
The rest 4 equations reduce to the following system of equations
\[
\displaystyle{2 \brack {5 \ 2}}=\frac{\ell+n}{N}-\displaystyle{2 \brack {4 \ 2}}\,, \quad \displaystyle{3 \brack {4 \ 3}}=1-\displaystyle{2 \brack {4 \ 2}}\,, \quad \displaystyle{3 \brack {5 \ 3}}=\displaystyle{2 \brack {4 \ 2}}-\frac{\ell}{N}\,. 
\]
Hence it is sufficient to compute one of these triples. Consider  the generators $Z_4, Z_5$ of $\fr{f}_4, \fr{f}_5$, respectively,  and let us denote by $\tilde{Z}_4, \tilde{Z}_{5}$ the corresponding $B$-orthonormal vectors. Then   $\tilde{Z}_{i}=c_{i}Z_{i}$ for $i=4, 5$, with
 \begin{equation}\label{ci}
c_4 = \displaystyle{\frac{\sqrt{(\ell+m)n}}{(\ell+m+n)\sqrt{2}}}\,, 
\quad 
c_5 = \displaystyle{\frac{\sqrt{\ell m}}{\sqrt{2(\ell+m+n)}\sqrt{\ell+m}}}\,,
\end{equation}
respectively. Given now some $B$-orthonormal vector $X_{j}\in\fr{f}_2$, for example, we compute $[\tilde{Z}_5, X_{j}]=\frac{c_{5}\sqrt{-1}}{\ell}X_{j}$.  Hence, by applying the definition of structure constants it follows that 
\[
\displaystyle{2 \brack {5 \ 2}}=\frac{(c_5)^2}{\ell^{2}}\dim_{\R}\fr{f}_2=\frac{d_{2}(c_5)^2}{\ell^{2}}=\frac{mn}{(\ell+m)N}\,,
\]
and in  a similar way one may compute $\displaystyle{2 \brack {4 \ 2}}$. This completes the proof. 
\end{proof}


\subsection{$\SU(\ell+m+n)$-invariant metrics}\label{invoi} 
Consider the reductive decomposition  (\ref{giso}) and let $\rho_{0}\equiv( \ , \ )_{0}$ be an arbitrary scalar product on $\fr{f}_0$. Since $\rho_{0}$ is an inner product, its matrix  $[\rho_{0}]$   with respect to the $B$-orthonormal  basis $\widetilde{\mathscr{B}}_{0}=\{\tilde{Z}_4, \tilde{Z}_5\}$  is  diagonalizable. In particular,   by $\theta$-Lemma (see Lemma \ref{thetalem}),   $\fr{f}_{0}$ decomposes into two equivalent 1-dimensional modules and given any other scalar product  $\beta\equiv ( \ , \ )$  in $\fr{f}_0$, there is some  $\beta$-orthonormal basis $\hat{\mathscr{B}}_{0}=\{V_4, V_5\}$ of $\fr{f}_{0}$ such that
\begin{equation}\label{prod}
( \ , \ )_{0} = v_4 ( \ , \ )|_{\hat{\frak f}_4} +v_5 ( \ , \ )|_{\hat{\frak f}_5}\,.
\end{equation}  
Here,   $\hat{\fr{f}}_4:= \Span\{V_4\}$, $\hat{\fr{f}}_5:= \Span\{V_5\}$ with $\fr{f}_{0}\simeq\hat{\fr{f}}_4\oplus\hat{\fr{f}}_{5}$, and the positive numbers $v_4, v_5 >0$, are the eigenvalues of $[\rho_{0}]$. 
Note that the basis $\hat{\mathscr{B}}_{0}$ is related to the basis $\widetilde{\mathscr{B}}_{0}$ via a matrix $Q=\begin{pmatrix}
p & q\\
r & s
\end{pmatrix}\in\Gl_2^{+}(\R)$ (see also below the relation (\ref{changeofB})), such that 
\[
[\rho_{0}]=(Q^{-1})^{T}\diag\{v _4, v_5\}Q^{-1}\,,
\]
and it turns out that any scalar product 
  on $\frak f_0$  is given by an expression of the  form
 (\ref{prod}).

From now on we consider   the reductive decomposition of $M_{\ell, m, n}$ given by
\begin{equation}\label{dec2}
\fr{g}=\fr{h}\oplus\fr{m}=(\fr{h}_{\upa}\oplus\fr{h}_{\upb}\oplus\fr{h}_{\upg})\oplus(\fr{f}_1\oplus\fr{f}_2\oplus\fr{f}_3\oplus\hat{\fr{f}}_{4}\oplus\hat{\fr{f}}_{5})\,.
\end{equation}
Note that  except $\ell m =1$, $\ell n =1$ or $m n =1$,  the $\Ad(K)$-decomposition of $\fr{m}$ coincides with the $\Ad(H)$-decomposition of $\fr{m}$.
Hence, we can consider $G$-invariant metrics $g$ on   
   $M_{\ell, m, n}=G/H$ which are given by an $\Ad(K)$-invariant inner product  $\langle \ , \ \rangle$ on $\fr{m}=\fr{f}_1\oplus\fr{f}_2\oplus\fr{f}_3\oplus\hat{\fr{f}}_{4}\oplus\hat{\fr{f}}_{5}$  of  the form
 \begin{equation}\label{ggg}
g\equiv\langle \ , \ \rangle=x_{1}B|_{\fr{f}_1}+x_2B|_{\fr{f}_2}+x_3B|_{\fr{f}_3}+( \ , \ )_{0}=x_{1}B|_{\fr{f}_1}+x_2B|_{\fr{f}_2}+x_3B|_{\fr{f}_3}+v_4 ( \ , \ )|_{\hat{\frak f}_4} +v_5 ( \ , \ )|_{\hat{\frak f}_5},
\end{equation}
with $x_{i}>0, v_{j}>0$ for any $i=1, 2, 3$ and $j=4, 5$. 
   
\begin{prop}\label{parametrize}
Consider the scalar product $( \ , \ )_{0}$ on $\fr{f}_{0}$ given by  {\rm (\ref{prod})}. Then, there exists a basis $\{V_4', V_5'\}$ of $\fr{f}_0$ such that the matrix $[\rho_{0}]$ of   $( \ , \ )_{0}$ with respect to $\widetilde{\mathscr{B}}_{0}=\{\tilde{Z}_4, \tilde{Z}_5\}$ has the expression
\begin{equation}\label{form}
[\rho_{0}]=\begin{pmatrix}
1 & 0\\
\gamma & 1
\end{pmatrix}^{T}
\begin{pmatrix}
v_4' & 0\\
0 &  v_5'
\end{pmatrix}
\begin{pmatrix}
1 & 0\\
\gamma & 1
\end{pmatrix},
\end{equation}
for some real number $\gamma$  and $v_4', v_5'>0$. 
\end{prop}
\begin{proof}
Let us assume that $Q^{-1}=\begin{pmatrix}
a & b\\
c & d
\end{pmatrix}$. 
Based on the QR-decomposition  we obtain that
$$
\begin{pmatrix}
a & b\\
c & d
\end{pmatrix}=\begin{pmatrix}
\cos t & -\sin t\\
\sin t & \cos t
\end{pmatrix}\begin{pmatrix}
x & 0\\
0 & y
\end{pmatrix}\begin{pmatrix}
1 & 0\\
\gamma & 1
\end{pmatrix},
$$
for some non-zero   $x, y$ and some real number $\gamma$.
Thus, the matrix $[\rho_{0}]$    takes the form
$$
\begin{pmatrix}
1 & 0\\
\gamma & 1
\end{pmatrix}^{T}
\begin{pmatrix}
x & 0\\
0 &  y
\end{pmatrix}
\begin{pmatrix}
\cos t & \sin t\\
-\sin t & \cos t
\end{pmatrix}
\begin{pmatrix}
v_4 & 0\\
0 &  v_5
\end{pmatrix}
\begin{pmatrix}
\cos t & -\sin t\\
\sin t & \cos t
\end{pmatrix}
\begin{pmatrix}
x & 0\\
0 &  y
\end{pmatrix}
\begin{pmatrix}
1 & 0\\
\gamma & 1
\end{pmatrix}.
$$
By changing the orthonormal basis $\hat{\mathscr{B}}_{0}=\{V_4, V_5\}$ into some orthonormal basis
$\{V_4', V_5'\}$, we deduce that 
$$
[\rho_{0}]=\begin{pmatrix}
1 & 0\\
\gamma & 1
\end{pmatrix}^{T}
\begin{pmatrix}
x & 0\\
0 &  y
\end{pmatrix}
\begin{pmatrix}
v_4 & 0\\
0 &  v_5
\end{pmatrix}
\begin{pmatrix}
x & 0\\
0 &  y
\end{pmatrix}
\begin{pmatrix}
1 & 0\\
\gamma & 1
\end{pmatrix},
$$ which yields the expression (\ref{form}). \end{proof}
Thus,   without loss of generality we may set
$a=d=1$, $b=0$ and $c=\gamma\in\R$, which means that  the parametrization of  inner products on $\fr{f}_{0}$ depends on two positive real numbers  plus  a $2\times 2$-nilpotent matrix  $N=\begin{pmatrix}
1 & 0\\
\gamma & 1
\end{pmatrix}$, with $\gamma\in\R$.   This conclusion  is in agreement with the $\theta$-Lemma,  in particular  from  the entries in the  matrix $Q^{-1}$ (or $Q$)  survives only a real parameter $c=\gamma$,    and this establishes  an alternative  parametrization of the inner products on $\fr{f}_{0}$. 
As a summary, we  obtain
\begin{corol}\label{invRIEM}
For  $\ell m\neq 1$, $\ell n\neq 1$ and  $m n\neq 1$, the space of $G$-invariant metrics on the C-space   $M_{\ell, m, n}=G/H=\SU(\ell+m+n)/\SU(\ell)\times\SU(n)\times\SU(m)$ is 6-dimensional, and any  invariant Riemannian metric on $M$ is given by {\rm (\ref{ggg})}.
\end{corol}
\begin{remark}\label{remarkinv}
  By Corollary \ref{invRIEM}, the metric defined by (\ref{ggg}) will represent a general $G$-invariant metric on $M_{\ell, m, n}=G/H$ only under the assumption  $\ell m\neq 1$, $\ell n\neq 1$ and  $m n\neq 1$. This assumption certifies  that none of  the components $\fr{f}_i$ in the isotropy representation of $M_{\ell, m, n}$  can be two-dimensional, for $i=1, 2, 3$. For example, for  the particular case $\ell=m=n=1$ which corresponds to the  Lie group $\SU(3)$, the  invariant metric $g$ defined by (\ref{ggg}) is {\it not} the most general invariant metric and similar for other cases induced by values such  that $\ell m= 1$, $\ell n= 1$ or $m n= 1$. However, as we will see below by using (\ref{ggg})  there are cases with $\ell m= 1$, $\ell n= 1$ or  $m n= 1$, for which we can still provide new invariant  Einstein metrics (see  Table \ref{Table31}). We should also mention  that unfortunately, this is not the case for $M_{1, 1, 1}=\SU(3)$ (see  Section \ref{casecccc}).
 \end{remark}

Now, given the $B$-orthonormal basis $\mathscr{B}_{0}=\{\tilde{Z}_i=c_iZ_{i} : i=4, 5\}$ of $\fr{f}_{0}$,  one has to observe that $(\tilde{Z}_4, \tilde{Z}_5)_{0}$ maybe {\it not} be zero.    Let    $\hat{\mathscr{B}}_{0}^{g}=\{U_4, U_5\}$  be an orthonormal basis of $\fr{f}_0\simeq\hat{\fr{f}}_4\oplus\hat{\fr{f}}_5$ with respect to $g=\langle \ , \ \rangle$, that is $U_{k} = \frac{1}{\sqrt{v_{k}}}V_{k}$ for any  $ k = 4, 5$, 
 and let $\widetilde{\mathscr{B}}_{\z}=\{\tilde{X}_{j}^{\z}: j = 1, \dots , \dim\fr{f}_{\z}\}$  be a  $B$-orthonormal basis of $\fr{f}_{\z}$, for any $\z=1, 2, 3$.  Set  $\mathscr{B}_{\z}^{g}=\frac{1}{\sqrt{x_{\z}}}\widetilde{\mathscr{B}}_{\z}$, i.e.
{\small \[
 \mathscr{B}_1^{g}=\{X_{h}^{1}=\displaystyle{\frac{\tilde{X}_{h}^{1}}{\sqrt{x_{1}}}} : 1\leq h\leq 2\ell m\}\,,  \ \mathscr{B}_2^{g}=\{X_{i}^{2}=\displaystyle{\frac{\tilde{X}_{i}^{2}}{\sqrt{x_{2}}}}: 1\leq i\leq 2\ell n\}\,, \ \mathscr{B}_3^{g}=\{X_{j}^{3}=\displaystyle{\frac{\tilde{X}_{j}^{3}}{\sqrt{x_{3}}}} : 1\leq j\leq 2mn\}\,.
 \]}
Then, the following (disjoint)  union of sets
\[
\mathscr{B}=\mathscr{B}_1^{g}\cup\mathscr{B}_2^{g}\cup\mathscr{B}_3^{g}\cup\hat{\mathscr{B}}_{0}^{g}=
\mathscr{B}_1^{g}\cup\mathscr{B}_2^{g}\cup\mathscr{B}_3^{g}\cup\{U_4\}\cup\{U_5\}
\] 
forms a $g$-orthonormal basis of $\fr{m}$, adapted to the decomposition $\fr{m}=\fr{f}_1\oplus\fr{f}_2\oplus\fr{f}_3\oplus\hat{\fr{f}}_{4}\oplus\hat{\fr{f}}_5$. Now, because we may have
\begin{equation}\label{nonnatred}
g([\tilde{X}_{i}^{\z}, \tilde{X}_{j}^{\z}]\, ,\, U_{k})\neq g(\tilde{X}_{i}^{\z}\, ,\, [\tilde{X}_{j}^{\z}, U_{k}])\,, \ \ \mbox{for}\ \ k = 4, 5\,,
\end{equation}
it is useful to express the inner product $( \  , \ )_{0}$ defined by (\ref{prod}), in terms of $B$.

\begin{prop}\label{prop1}
Any two $g$-orthonormal vectors  $X_{i}^{\z}, X_{j}^{\z}\in\mathscr{B}_{\z}^{g}$, $(\z = 1,2,3)$ satisfy the relations
\begin{eqnarray*}
([X_{i}^{\z}, X_{j}^{\z}], U_4)_{0} &=& \sqrt{v_4}\Big(aB([X_{i}^{\z}, X_{j}^{\z}], \tilde{Z}_{4}) + bB([X_{i}^{\z}, X_{j}^{\z}], \tilde{Z}_5)\Big)\,,\\
([X_{i}^{\z}, X_{j}^{\z}], U_{5})_{0} &=&  \sqrt{v_5}\Big(cB([X_{i}^{\z}, X_{j}^{\z}], \tilde{Z}_{4}) + dB([X_{i}^{\z}, X_{j}^{\z}], \tilde{Z}_5)\Big)\,.
\end{eqnarray*}
\end{prop}
\begin{proof}
Let us provide a proof for general $a, b, c, d$, although by Proposition \ref{parametrize} one can work with $b=0$, $a=d=1$ and $c\in\R$.   We use the notation  
$Q^{-1}=\begin{pmatrix}
a & b\\
c & d
\end{pmatrix}$, such that
\begin{equation}\label{changeofB}
\widetilde{\mathscr{B}}_{0}=\hat{\mathscr{B}}_{0}Q^{-1}\,,
\end{equation}
where $\widetilde{\mathscr{B}}_{0}=\{\tilde{Z}_4, \tilde{Z}_5\}$ and $\hat{\mathscr{B}}_{0}=\{V_4, V_5\}$, respectively.
Thus 
\[
\tilde{Z}_4=aV_4+cV_5,\quad \tilde{Z}_5=bV_4+dV_5\,,
\]
and due to   (\ref{prod}), by applying  a $B$-orthonormal  expansion we obtain  
\begin{eqnarray*}
([X_{i}^{\z}, X_{j}^{\z}], U_4)_{0} &=& v_4 ([X_{i}^{\z}, X_{j}^{\z}], U_4) = \sqrt{v_4}([X_{i}^{\z}, X_{j}^{\z}], V_4)\\
&=&\sqrt{v_4}\Big(B([X_{i}^{\z}, X_{j}^{\z}], \tilde{Z}_4)\tilde{Z}_4 + B([X_{i}^{\z}, X_{j}^{\z}],  \tilde{Z}_5)\tilde{Z}_5\,,  V_4\Big)\\
&= &\sqrt{v_4}\Big[ B([X_{i}^{\z}, X_{j}^{\z}], \tilde{Z}_4)\big(\tilde{Z}_4\,, V_4\big) +B([X_{i}^{\z}, X_{j}^{\z}],  \tilde{Z}_5)\big(\tilde{Z}_5\,, V_4\big)\Big]\\
&= &\sqrt{v_4}\Big[ B([X_{i}^{\z}, X_{j}^{\z}], \tilde{Z}_4)a +B([X_{i}^{\z}, X_{j}^{\z}],  \tilde{Z}_5)b\Big]\,.
\end{eqnarray*}
Similarly is treated the second relation.
\end{proof}

For our computations  it will be useful to consider also orthonormal {\it Weyl bases} of the irreducible modules $\fr{f}_{1}, \fr{f}_{2}$ and $\fr{f}_{3}$.    So, let $E_{ij}$ denote the $N\times N$ matrix with $1$ in the $(i, j)$-entry and $0$ elsewhere, where as usual $N=\ell+m+n>0$, and set
\[
A_{ij}:=E_{ij}-E_{ji}\,,\quad B_{ij}:=\sqrt{-1}(E_{ij}+E_{ji})\,.
\] 
In these terms,  a direct computation shows that

\begin{lemma}\label{lemma1}
1) Consider the $B$-orthonormal vectors  $\tilde{A}_{ij}=\mu A_{ij}$ and $\tilde{B}_{ij}=\mu B_{ij}$ of $\SU(N)$, where $\mu=2\sqrt{\ell+m+n}=2\sqrt{N}$. Then, the sets  
 \begin{eqnarray*}
\mathscr{D}_{1}&=&\{\tilde{A}_{ij}, \tilde{B}_{ij}: i=\ell+1, \dots ,\ell +m; j=1,\dots ,\ell\}\,,\\
\mathscr{D}_{2}&=&\{\tilde{A}_{ij}, \tilde{B}_{ij}: i=\ell+m+1, \dots ,N; j=1,\dots ,\ell\}\,,\\
\mathscr{D}_{3}&=&\{\tilde{A}_{ij}, \tilde{B}_{ij}: i=\ell+m+1, \dots ,N; j=\ell+1,\dots ,\ell+m\}\,,
\end{eqnarray*}
 form $B$-orthonormal bases of the modules $\fr{f}_1, \fr{f}_2$ and $\fr{f}_3$, respectively. \\ 
 2) The $B$-orthonormal vectors belonging in $\mathscr{D}_{1}, \mathscr{D}_{2}$ and $\mathscr{D}_{3}$ satisfy the following Lie bracket relations:
\[
\begin{tabular} { l | l | l }
$A_{ij}, B_{ij}\in \fr{f}_1$                                                                                  &   $A_{ij}, B_{ij}\in \fr{f}_2$                                                                                &  $A_{ij}, B_{ij}\in \fr{f}_3$ \\
\thickline 
$[\tilde{Z}_4, A_{ij}]=0$                     								  &   $[\tilde{Z}_4, A_{ij}] = \displaystyle{c_4\frac{N}{n(\ell+m)}B_{ij}}$                 &  $ [\tilde{Z}_4, A_{ij}] = \displaystyle{c_4\frac{N}{n(\ell+m)}B_{ij}}$ \\
$[\tilde{Z}_4, B_{ij}]=0$ 									          &   $[\tilde{Z}_4, B_{ij}] = \displaystyle{-c_4\frac{N}{n(\ell+m)}A_{ij}}$		&  $ [\tilde{Z}_4, B_{ij}] = \displaystyle{-c_4\frac{N}{n(\ell+m)}A_{ij}}$ \\
$[\tilde{Z}_{5}, A_{ij}] = \displaystyle{c_5\frac{(\ell+m)}{\ell m} B_{ij}}$           &   $[\tilde{Z}_{5}, A_{ij}] = \displaystyle{c_5\frac{1}{\ell}B_{ij}}$ 			& $[\tilde{Z}_{5}, A_{ij}] = \displaystyle{-c_5\frac{1}{m}B_{ij}}$ \\
$ [\tilde{Z}_{5}, B_{ij}] = \displaystyle{-c_5\frac{(\ell+m)}{\ell m} A_{ij}}$  	  &   $[\tilde{Z}_{5}, B_{ij}] = \displaystyle{-c_5\frac{1}{\ell}A_{ij}}$       		& $[\tilde{Z}_{5}, B_{ij}] = \displaystyle{c_5\frac{1}{m}A_{ij}},$
\end{tabular}
\]
where $c_4, c_5$ are given in {\rm (\ref{ci})}.
\end{lemma}


\section{The Ricci tensor on $M_{\ell, n, m}=G/H=\SU(\ell+m+n)/\SU(\ell)\times\SU(m)\times\SU(n)$}\label{section4}
In this section we  describe the Ricci tensor on the   torus bundle   \[M_{\ell, n, m}=G/H=\SU(\ell+m+n)/\SU(\ell)\times\SU(m)\times\SU(n)\,,
\]  for the    $G$-invariant metric   $g$ induced by the inner product $\langle \ , \ \rangle$ on $\fr{m}=\fr{f}_1\oplus\fr{f}_2\oplus\fr{f}_3\oplus\fr{f}_{0}$, given by (\ref{ggg}) (and from now on  we shall identify them).  As  we explained  in Section \ref{invoi}, except of the cases $\ell m=1$, $\ell n=1$, or $mn=1$, $g$ is  a general invariant metric on $M_{\ell, m, n}$.  Now, because $\fr{f}_{0}$   decomposes into two equivalent 1-dimensional submodules, we should  divide   the study into two parts, related with the abelian part  $\fr{f}_0$ and the diagonal part $\fr{f}_{1}\oplus\fr{f}_2\oplus\fr{f}_3$ of $\fr{m}$, respectively.


\subsection{The Ricci tensor for the abelian part}
For the computation of  the Ricci tensor $\Ric^{g}$ on the abelian part $\fr{f}_{0}$,  we shall apply directly the formula (\ref{genric}). However, for such a procedure first we need  to do some preparatory work.    We begin with the following result, whose proof  relies on   the  $g$-orthonormal bases $\mathscr{B}_{\z}^{g}=\{X_{j}^{\z}: j =1, \dots ,\dim\frak f_{\z}\}$ of the modules $\fr{f}_{\z}$  ($\z=1, 2, 3$), given in Section \ref{invoi} and the expression (\ref{ggg}) of $g=\langle \ , \ \rangle$.
\begin{prop}\label{prop2}  
 The  invariant metric on $M_{\ell, n, m}=G/H$ defined by  {\rm (\ref{ggg})}, satisfies the relation   
\begin{eqnarray}
\sum_{\z = 1, 2, 3}\sum_{j=1}^{\dim\fr{f}_{\z}}g\big([Z, X_{j}^{\z}],    [W, X_{j}^{\z}]\big) = B(Z, W)\,,
\end{eqnarray}
for any  $Z, W\in\fr{f}_0$. 
\end{prop}

The non-zero $B$-structure constants of $M_{\ell, n, m}$ computed in Proposition \ref{ijk} do  not   provide all the  sufficient data for the computation of   $\Ric^{g}$.  Indeed, the invariant metric $g$ except of $B$, depends also on the inner product $( \ , \ )_{0}$ given in (\ref{prod}). Thus, we need     exactly the structure constants related with $g=\langle \ , \ \rangle$. For their computation  is useful to denote the  decomposition of $\fr{m}$ by
 \[
\fr{m}=\fr{f}_{1}\oplus\fr{f}_2\oplus\fr{f}_3\oplus\hat{\fr{f}}_4\oplus\hat{\fr{f}}_5=\fr{m}_1\oplus\fr{m}_2\oplus\fr{m}_3\oplus\fr{m}_4\oplus\fr{m}_5
\]
with $\fr{f}_{1}=\fr{m}_{1}$, $\fr{f}_{2}=\fr{m}_{2}$, $\fr{f}_{3}=\fr{m}_{3}$, $\hat{\fr{f}}_{4}=\fr{m}_{4}$, and $\hat{\fr{f}}_{5}=\fr{m}_{5}$. Moreover, with aim to maintain   the previous notation of the $g$-orthonormal bases  of $\fr{f}_1, \fr{f}_2, \fr{f}_3$, let us   denote by 
$\{X_{\al}^{\ii}\}$ the  $g$-orthonormal bases of $\fr{m}_{\ii}$,  with $\al=1, \ldots, \dim_{\R}\fr{m}_{\ii}$, for any $\ii=1,\ldots, 5$, i.e.
\[
X_{\al}^{\ii}=\frac{1}{\sqrt{x_{i}}}\tilde{X}_{\al}^{\ii},\quad \ii=1, 2, 3,\quad X_{1}^{4}=U_4=\frac{1}{\sqrt{v_4}}V_4, \quad X_{1}^{5}=U_5=\frac{1}{\sqrt{v_5}}V_5\,.
\]
  In these terms, the    {\it $\langle \ , \ \rangle$-structure constants} are defined by 
 \[
\bigg\lbrace
\begin{matrix}
\kk\\
\ii \ \jj
\end{matrix}
\bigg\rbrace := \sum_{\al, \beta, \gamma} \langle [\tilde{X}_{\al}^{\ii}, \tilde{X}_{\beta}^{\jj}]\, ,\, 
\tilde{X}_{\gamma}^{\kk}\rangle^2  \,,
\]
with $1\leq \ii, \jj, \kk\leq 5$.  
We mention that 
$
\bigg\lbrace
\begin{matrix}
\kk\\
\ii\ \jj
\end{matrix}
\bigg\rbrace = 
\bigg\lbrace
\begin{matrix}
\kk\\
\jj\ \ii
\end{matrix}
\bigg\rbrace
$, but due to the relation (\ref{nonnatred}) it follows that
$ 
\bigg\lbrace
\begin{matrix}
\kk\\
\ii\ \jj
\end{matrix}
\bigg\rbrace$ is not always equal to
$ 
\bigg\lbrace
\begin{matrix}
\jj\\
\ii\ \kk
\end{matrix}
\bigg\rbrace.
$
However, in view of (\ref{dec2}) we have that $
\bigg\lbrace
\begin{matrix}
3\\
1 \ 2
\end{matrix}
\bigg\rbrace = 
\displaystyle{3\brack 1 \ 2}$. 
 On the other hand, the triples
\begin{equation}\label{numbc}
\bigg\lbrace
\begin{matrix}
4\\
1 \ 1
\end{matrix}
\bigg\rbrace, \ \ 
\bigg\lbrace
\begin{matrix}
4\\
2 \   2
\end{matrix}
\bigg\rbrace, \ \ 
\bigg\lbrace
\begin{matrix}
4\\
3 \ 3
\end{matrix}
\bigg\rbrace, \ \ 
\bigg\lbrace
\begin{matrix}
5\\
1 \ 1
\end{matrix}
\bigg\rbrace, \ \ 
\bigg\lbrace
\begin{matrix}
5\\
2 \ 2
\end{matrix}
\bigg\rbrace, \ \ 
\bigg\lbrace
\begin{matrix}
5\\
3 \ 3
\end{matrix}
\bigg\rbrace
\end{equation}
need to be computed.  Now,  due to Proposition \ref{parametrize} and with aim to simplify the exposition, we may adopt the values  $a=d=1$, $b=0$ and $c=\gamma\in\R$, and  replace them in our formulas.  
\begin{prop}\label{mainijk} 
  The quantities appearing in  $(\ref{numbc})$ attain the following values:
\begin{equation}
\begin{tabular}{l l}
$\displaystyle\bigg\lbrace
\begin{matrix}
4\\
1 \ 1
\end{matrix}
\bigg\rbrace = 0$,  
\quad & $\displaystyle\bigg\lbrace
\begin{matrix}
5\\
1  \ 1
\end{matrix}
\bigg\rbrace 
=\frac{(\ell + m)}{N}\,, $ \\
$\displaystyle\bigg\lbrace
\begin{matrix}
4\\
2 \ 2
\end{matrix}
\bigg\rbrace = \frac{\ell}{\ell + m} $, 
\quad & $\displaystyle\bigg\lbrace
\begin{matrix}
5\\
2 \  2
\end{matrix}
\bigg\rbrace = \frac{c^{2}\ell}{\ell + m} + \frac{mn}{N(\ell + m)} + \frac{2c\sqrt{\ell m n}}{(\ell + m)\sqrt{N}}\,,$ \\
$\displaystyle\bigg\lbrace
\begin{matrix}
4\\
 3 \ 3 
\end{matrix}
\bigg\rbrace = \frac{m}{\ell + m} $, 
\quad & $\displaystyle\bigg\lbrace 
\begin{matrix}
5\\
3 \  3
\end{matrix}
\bigg\rbrace = \frac{c^{2}m}{\ell + m} + \frac{\ell n}{N(\ell + m)} - \frac{2c\sqrt{\ell m n}}{(\ell + m)\sqrt{N}}$\,.
\end{tabular}
\end{equation}
\end{prop}
\begin{proof}
Since $a=1$ and $b=0$, by Proposition \ref{prop1} it is easy to see that the structure constants $\bigg\lbrace
\begin{matrix}
4\\
\ii \ \ii
\end{matrix}
\bigg\rbrace$ are equal to $\displaystyle{4\brack i \ i}$ for any $\ii=i=1, 2, 3$.  Indeed, based on the  the relation $( \ , \ )_{0}|_{\hat{\fr{f}}_4}=v_4 ( \ , \ )|_{\hat{\frak f}_4}$, we  compute
\begin{eqnarray*}
\bigg\lbrace
\begin{matrix}
4\\
\ii \ \ii
\end{matrix}
\bigg\rbrace &=& \sum_{j, l}([\tilde{X}_{j}^{\ii}, \tilde{X}_{l}^{\ii}], V_4)^2 =\sum_{j, l}\frac{1}{v_4^2}([\tilde{X}_{j}^{\ii}, \tilde{X}_{l}^{\ii}], V_4)^2_{0}\\
&=&\sum_{j, l}\frac{1}{v_4^2}([\tilde{X}_{j}^{\ii}, \tilde{X}_{l}^{\ii}], \sqrt{v_4}U_4)^2_{0}=\sum_{j, l}\frac{1}{v_4}([\tilde{X}_{j}^{\ii}, \tilde{X}_{l}^{\ii}], U_4)^2_{0}\\
&\overset{\text{Prop.} \ \ref{prop1}}{=}& \sum_{j, l} \left(aB([\tilde{X}_{j}^{\ii}, \tilde{X}_{l}^{\ii}]\, ,\, \tilde{Z}_4)\right)^{2}  \overset{a=1}{=} \sum_{j, l} B([\tilde{X}_{j}^{\ii}, \tilde{X}_{l}^{\ii}], \tilde{Z}_4)^{2}=\displaystyle{4\brack \ii \  \ii}\,. \\
\end{eqnarray*}
Let us now pass to the triples related with the component $\hat{\fr{f}}_{5}$ and  compute for example the unknown $\displaystyle\bigg\lbrace
\begin{matrix}
5\\
2  \ 2
\end{matrix}
\bigg\rbrace$ (and similarly are treated the other two cases).
By definition,  
\begin{eqnarray}
\bigg\lbrace
\begin{matrix}
5\\
2 \ 2
\end{matrix}\label{inijk}
\bigg\rbrace &=& \sum _{j,l} ([\tilde{X}_{j}^{2}, \tilde{X}_{l}^{2}],  V_{5})^{2} =\sum_{j, l}\frac{1}{v_5}([\tilde{X}_{j}^{2}, X_{l}^2], U_5)^2_{0}\nonumber\\
&\overset{\text{Prop.} \ \ref{prop1}}{=}& \sum \left(cB([\tilde{X}_{j}^{2}, \tilde{X}_{l}^{2}], \tilde{Z}_4)  + dB([\tilde{X}_{j}^{2}, \tilde{X}_{l}^{2}],  \tilde{Z}_5)\right)^{2} \nonumber\\
&=& \sum \left(c^{2}B([\tilde{X}_{j}^{2}, \tilde{X}_{l}^{2}], \tilde{Z}_4)^{2} + d^{2}B([\tilde{X}_{j}^{2}, \tilde{X}_{l}^{2}], \tilde{Z}_5)^{2}  
+ 2cdB([\tilde{X}_{j}^{2}, \tilde{X}_{l}^{2}], \tilde{Z}_4)B([\tilde{X}_{j}^{2}, \tilde{X}_{l}^{2}], \tilde{Z}_5) \right)\nonumber \\
&=&c^{2}\displaystyle{4\brack 2 \  2}+d^2\displaystyle{5\brack 2 \  2}+2cd A\,,
\end{eqnarray}
where by using   the bi-invariance of the Killing form we obtain
\[
A:=\sum_{j, l}B([\tilde{X}_{j}^{2}, \tilde{X}_{l}^{2}], \tilde{Z}_4)B([\tilde{X}_{j}^{2}, \tilde{X}_{l}^{2}], \tilde{Z}_5)\\
=\sum_{j, l}B(\tilde{X}_{j}^{2}, [\tilde{Z}_{4}, \tilde{X}_{l}^{2}])B(\tilde{X}_{j}^{2},  [\tilde{Z}_{5}, \tilde{X}_{l}^{2}])\,.
\]
  To compute A, and since the structure constants are independent of the adapted orthonormal bases, it is convenient to  change $\{\tilde{X}_{l}^{2} : 1\leq l\leq 2\ell n \}$ with the basis $\{\tilde{A}_{ij}, \tilde{B}_{ij} : i=\ell+m+1, \dots ,N; j=1,\dots ,\ell\}$ of Lemma \ref{lemma1}, and moreover set $N_1:=\ell+m+1$. Then,  
\begin{eqnarray*}
A&=&\sum_{\substack{N_1\le i,k\le N\\
     1\le j,l\le\ell}}\Big\{B(\tilde{A}_{ij}, [\tilde{Z}_4, \tilde{A}_{kl}])B(\tilde{A}_{ij}, [\tilde{Z}_5, \tilde{A}_{kl}])+B(\tilde{B}_{ij}, [\tilde{Z}_4, \tilde{B}_{kl}])B(\tilde{B}_{ij}, [\tilde{Z}_5, \tilde{B}_{kl}])\Big\}\\
     &&+\sum_{\substack{N_1\le i,k\le N\\
     1\le j,l\le\ell}}\Big\{B(\tilde{A}_{ij}, [\tilde{Z}_4, \tilde{B}_{kl}])B(\tilde{A}_{ij}, [\tilde{Z}_5, \tilde{B}_{kl}])+B(\tilde{B}_{ij}, [\tilde{Z}_4, \tilde{A}_{kl}])B(\tilde{B}_{ij}, [\tilde{Z}_5, \tilde{A}_{kl}])\Big\}
\end{eqnarray*}
Due to the stated  Lie brackets in  Lemma \ref{lemma1} and in combination with the relation $B(\tilde{A}_{ij}, \tilde{B}_{kl})=0$ we see that the first line vanishes and hence
\begin{eqnarray*}
A&=&\sum_{\substack{N_1\le i,k\le N\\
     1\le j,l\le\ell}}\Big\{B(\tilde{A}_{ij}, [\tilde{Z}_4, \tilde{B}_{kl}])B(\tilde{A}_{ij}, [\tilde{Z}_5, \tilde{B}_{kl}])+B(\tilde{B}_{ij}, [\tilde{Z}_4, \tilde{A}_{kl}])B(\tilde{B}_{ij}, [\tilde{Z}_5, \tilde{A}_{kl}])\Big\}\\
     &=&\frac{c_4c_5N}{\ell n(\ell+m)}\sum_{\substack{N_1\le i,k\le N\\
     1\le j,l\le\ell}}\Big(B(\tilde{A}_{ij}, \tilde{A}_{kl})B(\tilde{A}_{ij}, \tilde{A}_{kl})+B(\tilde{B}_{ij}, \tilde{B}_{kl})B(\tilde{B}_{ij}, \tilde{B}_{kl})\Big)\\
     &=&\frac{\sqrt{\ell m n}}{2 n\ell(\ell+m)\sqrt{N}}\cdot 2n\ell=\frac{\sqrt{\ell m n}}{(\ell+m)\sqrt{N}}\,,
     \end{eqnarray*}
where we  substitute the values $c_4, c_5$ by (\ref{ci}). Our assertion now follows by relation (\ref{inijk}), in combination with the results of Proposition \ref{ijk}. \end{proof}

\begin{prop}\label{ricci1}  
Consider the C-space  $(M_{\ell, m, m}=G/H, g)$,  where    $g$ is the $G$-invariant metric defined by $(\ref{ggg})$. Then, 
	the components of the Ricci tensor $\Ric^{g}$  on $\fr{f}_0$ are given by
\begin{eqnarray}
\Ric_4^g &=& \frac{v_4}{4}\left(\frac{1}{{x_{2}}^{2}}\bigg\lbrace
\begin{matrix}
4\\
2  \ 2
\end{matrix}
\bigg\rbrace + \frac{1}{{x_{3}}^{2}}\bigg\lbrace
\begin{matrix}
4\\
3 \ 3
\end{matrix}
\bigg\rbrace \right)\,,\label{r4}\\ 
\Ric_5^g &=& \frac{v_5}{4}\left(\frac{1}{{x_{1}}^{2}}\bigg\lbrace
\begin{matrix}
5\\
1 \ 1
\end{matrix}
\bigg\rbrace + \frac{1}{{x_{2}}^{2}}\bigg\lbrace
\begin{matrix}
5\\
2 \ 2
\end{matrix}
\bigg\rbrace + \frac{1}{{x_{3}}^{2}}\bigg\lbrace
\begin{matrix}
5\\
3 \ 3
\end{matrix}
\bigg\rbrace \right)\,.\label{r5}
\end{eqnarray}
Moreover, the mixed term $\Ric_0^g:=\Ric^{g}(U_4, U_5)$ is expressed by
\begin{eqnarray}
 \Ric_0^g &=& \frac{\sqrt{v_4 v_5}}{4} \Bigg\lbrace\frac{1}{{x_{2}}^{2} (\ell + m)}\left(\ell c + \frac{\sqrt{\ell m n}}{\sqrt{N}}\right)
+ \frac{1}{{x_{3}}^{2} (\ell + m)} \left( mc - \frac{\sqrt{\ell m n}}{\sqrt{N}} \right) \Bigg\rbrace\,.  \label{r0}
\end{eqnarray}
\end{prop}
\begin{proof}
Let us consider  the $g$-orthonormal basis $\hat{\mathscr{B}}_{0}^{g}=\{U_i = \displaystyle{\frac{1}{\sqrt{v_i}}}V_i : i=4, 5\}$ of $\fr{f}_{0}$ introduced in Section \ref{invoi}.  We will apply   the formula (\ref{genric}) in combination with Proposition \ref{prop2}. This yields  
  \begin{eqnarray*}
  \Ric^{g}_4& =& \Ric^{g}(U_4, U_{4}) = -\frac{1}{2}\sum_{\ii=1, 2, 3} 
\sum_{j=1}^{\dim\fr{m}_{\ii}}g([U_4, X_{j}^{\ii}], [U_4, X_{j}^{\ii}]) - \frac{1}{2}\, g([U_4, U_5],  [U_4, U_5]) \\
&& +\frac{1}{2}\, B(U_4, U_4) 
 +\frac14 \sum_{\ii, \kk } 
\sum_{j=1}^{\dim\fr{m}_{\ii}} \sum_{l=1}^{\dim\fr{m}_{\kk}} g([X_{j}^{\ii}, X_{l}^{\kk}], U_4)\, g([X_{j}^{\ii}, X_{l}^{\kk}], U_{4})\\
&=& -\frac12\, B(U_4, U_4) + \frac12\, B(U_4, U_4) 
 + \frac{1}{4}\sum_{\ii, \kk }
\sum_{j=1}^{\dim\fr{m}_{\ii}}\sum_{l=1}^{\dim\fr{m}_{\kk}} g([X_{j}^{\ii}, X_{l}^{\kk}], U_{4})^{2}\\
&=& \frac{1}{4}\sum_{\ii, \kk } 
\sum_{j=1}^{\dim\fr{m}_{\ii}}\sum_{l=1}^{\dim\fr{m}_{\kk}}  g([X_{j}^{\ii}, X_{l}^{\kk}], U_{4})^{2}\\
& =& \frac{1}{4}\sum_{\ii} \sum_{j, l} g([X_{j}^{\ii}, X_{l}^{\ii}], U_4)^{2} + \frac{1}{4} \sum_{\substack{j, l \\ \ii\neq\kk}}g([X_{j}^{\ii}, X_{l}^{\kk}], U_{4})^{2}\\
&& +\frac{1}{2} \sum_{\ii}\sum_{l} g([U_4, X_{l}^{\ii}], U_4)^{2} + \frac{1}{2} \sum_{\ii}\sum_{l} g([U_5,  X_{l}^{\ii}]\, ,\, U_4)^{2}\,.
\end{eqnarray*}  
Since  $\fr{f}_{1}=\fr{m}_{1}$, $\fr{f}_{2}=\fr{m}_{2}$, $\fr{f}_{3}=\fr{m}_{3}$, $\hat{\fr{f}}_{4}=\fr{m}_{4}$, and $\hat{\fr{f}}_{5}=\fr{m}_{5}$,  by  Lemma \ref{brackets}
it follows that the  last three terms in the above sum are equal to zero.  
For the first term, and since  for any $\ii=1, 2, 3$ we have $[X_{j}^{\ii}, X_{l}^{\ii}]\subset \fr{h}_{\upa}\oplus\fr{h}_{\upb}\oplus\fr{h}_{\upg}\oplus\fr{m}_4\oplus\fr{m}_5$, we obtain that
\begin{eqnarray*}
\Ric^{g}_{4}&=& \frac14 \sum_{\ii}\sum_{j, l}g([X_{j}^{\ii}, X_{l}^{\ii}], U_4)^{2} = \frac14 \sum_{\ii}\sum_{j, l} ([X_{j}^{\ii}, X_{l}^{\ii}], U_4)^{2}_{0}  = \frac{1}{4v_4}\sum_{\ii=1, 2, 3} \sum_{j, l} ([X_{j}^{\ii}, X_{l}^{\ii}], V_4)^{2}_{0}  \\
  &=& \sum_{\ii=1, 2, 3} \frac{v_4}{4{x_{\ii}}^{2}} \sum_{j, l} ([\tilde{X}_{j}^{\ii}, \tilde{X}_{l}^{\ii}], V_{4})^{2}  = \frac{v_4}{4{x_{1}}^{2}}\bigg\lbrace
\begin{matrix}
4\\
1 \ 1
\end{matrix}
\bigg\rbrace + 
\frac{v_4}{4 {x_{2}}^{2}}\bigg\lbrace
\begin{matrix}
4\\
2 \ 2
\end{matrix}
\bigg\rbrace +
\frac{v_4}{4 {x_{3}}^{2}}\bigg\lbrace
\begin{matrix}
4\\
3 \ 3
\end{matrix}
\bigg\rbrace\,.
\end{eqnarray*}
Then, the given expression (\ref{r4}) follows by Proposition \ref{mainijk}.  Note that  the formula  (\ref{r5}) occurs  by an analogous way. 

\smallskip
Let us now examine  the  mixed term $\Ric^{g}_{0}=\Ric^{g}(U_4, U_5)$.   By applying  again   formula (\ref{genric})  and Proposition \ref{prop2}, we obtain 
\begin{eqnarray*}
\Ric^{g}_0 &=& \Ric^{g}(U_4, U_5)  = -\frac{1}{2}\sum_{\ii =1, 2, 3 }\sum_{j =1}^{\dim\fr{m}_{\ii}} g([U_4, X_{j}^{\ii}],  [U_{5}, X_{j}^{\ii}])   
 - \frac{1}{2} g([U_4, U_5],  [U_5, U_5]) + \frac{1}{2} g([U_4, U_4], [U_4, U_5]) \\
 && + \frac12 B(U_4, U_5) 
 +\frac{1}{4}\sum_{\ii, \kk}\sum_{j=1}^{\dim\fr{m}_{\ii}}\sum_{l=1}^{\dim\fr{m}_{\kk}} g([X_{j}^{\ii}, X_{l}^{\kk}],  U_4) 
 g([X_{j}^{\ii}, X_{l}^{\kk}], U_5) \\
&=& -\frac{1}{2} B(U_4, U_5) +\frac{1}{2} B(U_4, U_5)  
+ \frac{1}{4}\sum_{\ii, \kk}\sum_{j=1}^{\dim\fr{m}_{\ii}}\sum_{l=1}^{\dim\fr{m}_{\kk}}g([X_{j}^{\ii}, X_{l}^{\kk}], U_4) g([X_{j}^{\ii}, X_{l}^{\kk}], U_5)\,. 
\end{eqnarray*}
Consequently,
\begin{eqnarray*}
\Ric^{g}_{0}&=& \frac{1}{4}\sum_{\ii, \kk}\sum_{j=1}^{\dim\fr{m}_{\ii}}\sum_{l=1}^{\dim\fr{m}_{\kk}} g([X_{j}^{\ii}, X_{l}^{\kk}], U_4) g([X_{j}^{\ii}, X_{l}^{\kk}], U_5) \\
  &=& \frac{1}{4}  \sum_{\ii}\sum_{j,l=1}^{\dim\fr{m}_{\ii}} g([X_{j}^{\ii}, X_{l}^{\ii}], U_4) g([X_{j}^{\ii}, X_{l}^{\ii}], U_5)  
 + \frac{1}{4} \sum_{1\leq\ii \neq\kk\leq 3}\sum_{j=1}^{\dim\fr{m}_{\ii}}\sum_{l=1}^{\dim\fr{m}_{\kk}} g([X_{j}^{\ii}, X_{l}^{\kk}], U_4) g([X_{j}^{\ii}, X_{l}^{\kk}], U_5) \\
& & + \frac{1}{2} \sum_{\ii} \sum_{l=1}^{\dim\fr{m}_{\ii}} g([U_4, X_{l}^{\ii}], U_4) g([U_4, X_{l}^{\ii}], U_5)  + \frac{1}{2} \sum_{\ii}\sum_{l=1}^{\dim\fr{m}_{\ii}} g([U_5, X_{l}^{\ii}], U_4) g([U_5, X_{l}^{\ii}], U_5)\,.
\end{eqnarray*}
Based on the Lie brackets given in Lemma \ref{brackets},    one deduces that the last three terms in the previous relation  are equal to zero. For the first term, and  since $[X_{j}^{\ii}, X_{l}^{\ii}]\subset\fr{h}_{\upa}\oplus\cdots\oplus\fr{m}_4\oplus\fr{m}_5$, for any $\ii=1, 2, 3$, in view of Proposition \ref{prop1} we obtain that (we substitute $a=d=1$ and $b=0$)
 \begin{eqnarray*}
\Ric^{g}_{0}&=&  \frac{1}{4} \sum_{\ii=1, 2, 3}\sum_{j,l=1}^{\dim\fr{m}_{\ii}}g([X_{j}^{\ii}, X_{l}^{\ii}], U_4) g([X_{j}^{\ii}, X_{l}^{\ii}],  U_5)  = \frac{1}{4}\sum_{\ii=1, 2, 3}\sum_{j,l=1}^{\dim\fr{m}_{\ii}}( [X_{j}^{\ii}, X_{l}^{\ii}], U_4)_{0} ([X_{j}^{\ii}, X_{l}^{\ii}], U_5)_{0}  \\ 
 &=& \frac{\sqrt{v_4 v_5}}{4} \sum_{\ii=1, 2, 3}\sum_{j,l=1}^{\dim\fr{m}_{\ii}}  B([X_{j}^{\ii}, X_{l}^{\ii}], \tilde{Z}_4) \cdot 
 \Big(cB([X_{j}^{\ii}, X_{l}^{\ii}],  \tilde{Z}_4) + B([X_{j}^{\ii}, X_{l}^{\ii}],  \tilde{Z}_5)\Big) \\
&=&  \frac{\sqrt{v_4 v_5}}{4} \sum_{\ii=1, 2, 3}\sum_{j,l=1}^{\dim\fr{m}_{\ii}}  B(X_{j}^{\ii}, [\tilde{Z}_4, X_{l}^{\ii}]) \cdot 
 \Big(cB(X_{j}^{\ii}, [\tilde{Z}_4, X_{l}^{\ii}]) + B(X_{j}^{\ii}, [\tilde{Z}_5,  X_{l}^{\ii}])\Big)\\
 &=& \frac{\sqrt{v_4 v_5}}{4} \sum_{\ii=1, 2, 3}\sum_{j,l=1}^{\dim\fr{m}_{\ii}}\frac{1}{{x_{\ii}}^{2}} \Gamma(\tilde{X}_{j}^{\ii}, \tilde{X}_{l}^{\ii})\,,
 \end{eqnarray*}
 where, for simplicity,  given any two $B$-orthonormal vectors $X, Y\in\fr{m}$ we set   
 \[
 \Gamma(X, Y):=cB(X, [\tilde{Z}_4, Y])^{2}+B(X, [\tilde{Z}_4, Y])B(X, [\tilde{Z}_5, Y])\,.
 \]
To compute the above sum we shall use the $B$-orthonormal bases presented in Lemma \ref{lemma1}.  Consider the first case $\ii=1$, i.e.  $\fr{f}_1=\fr{m}_1$.  According to the notation of Lemma \ref{lemma1}, we have to compute the sum
\[
\mathscr{M}_1=\sum_{\substack{\ell+1\leq i, k\leq \ell+m,  \\ 1\leq j, l\leq \ell}}\Big(\Gamma(\tilde{A}_{ij}, \tilde{A}_{kl})+\Gamma(\tilde{A}_{ij}, \tilde{B}_{kl})+\Gamma(\tilde{B}_{ij}, \tilde{A}_{kl})+\Gamma(\tilde{B}_{ij}, \tilde{B}_{kl})\Big)\,.
\]
However, for  $\tilde{A}_{ij}, \tilde{A}_{kl}, \tilde{B}_{ij}, \tilde{B}_{kl}\in\fr{m}_1$, the Lie brackets given in Lemma \ref{lemma1} ensure that
\[
\Gamma(\tilde{A}_{ij}, \tilde{A}_{kl})=\Gamma(\tilde{A}_{ij}, \tilde{B}_{kl})=\Gamma(\tilde{B}_{ij}, \tilde{A}_{kl})=\Gamma(\tilde{B}_{ij}, \tilde{B}_{kl})=0\,,
\]
and hence $\mathscr{M}_1$ vanishes. Let us pass to the second module $\fr{f}_2=\fr{m}_2$ and assume that $\tilde{A}_{ij}, \tilde{A}_{kl}, \tilde{B}_{ij}, \tilde{B}_{kl}\in\fr{m}_2$. In this case we need to compute the sum 
\[
\mathscr{M}_2:=\sum_{\substack{N_1\leq i, k\leq N,  \\ 1\leq j, l\leq \ell}}\Big(\Gamma(\tilde{A}_{ij}, \tilde{A}_{kl})+\Gamma(\tilde{A}_{ij}, \tilde{B}_{kl})+\Gamma(\tilde{B}_{ij}, \tilde{A}_{kl})+\Gamma(\tilde{B}_{ij}, \tilde{B}_{kl})\Big)\,,
\]
where as before we set $N_1=\ell+m+1$. A direct computation yields that $\Gamma(\tilde{A}_{ij}, \tilde{A}_{kl})=0=\Gamma(\tilde{B}_{ij}, \tilde{B}_{kl})$, and
\begin{eqnarray*}
\Gamma(\tilde{A}_{ij}, \tilde{B}_{kl})&=&  \Big[cc_4^2\frac{N^2}{n^2(\ell+m)^2}+\frac{c_4c_5}{\ell}\frac{N}{n(\ell+m)}\Big]B(\tilde{A}_{ij}, \tilde{A}_{kl})^{2}\,,\\
\Gamma(\tilde{B}_{ij}, \tilde{A}_{kl})&=&\Big[cc_4^2\frac{N^2}{n^2(\ell+m)^2}+\frac{c_4c_5}{\ell}\frac{N}{n(\ell+m)}\Big]B(\tilde{B}_{ij}, \tilde{B}_{kl})^{2}\,.
\end{eqnarray*}
Thus
\[
\mathscr{M}_2=2 \ell n \Big[cc_4^2\frac{N^2}{n^2(\ell+m)^2}+\frac{c_4c_5}{\ell}\frac{N}{n(\ell+m)}\Big]=\frac{1}{(\ell + m)}\left(\ell c + \frac{\sqrt{\ell m n}}{\sqrt{N}}\right)\,,
\]
 where we replace $c_4, c_5$ by (\ref{ci}). Finally,  we have the sum
 \[
\mathscr{M}_3:=\sum_{\substack{N_1\leq i, k\leq N,  \\ \ell+1\leq j, l\leq \ell+m}}\Big(\Gamma(\tilde{A}_{ij}, \tilde{A}_{kl})+\Gamma(\tilde{A}_{ij}, \tilde{B}_{kl})+\Gamma(\tilde{B}_{ij}, \tilde{A}_{kl})+\Gamma(\tilde{B}_{ij}, \tilde{B}_{kl})\Big)\,,
\]
with $\tilde{A}_{ij}, \tilde{A}_{kl}, \tilde{B}_{ij}, \tilde{B}_{kl}\in\fr{m}_3$.  By Lemma \ref{lemma1} we see that $\Gamma(\tilde{A}_{ij}, \tilde{A}_{kl})=0=\Gamma(\tilde{B}_{ij}, \tilde{B}_{kl})$ and
 \begin{eqnarray*}
\Gamma(\tilde{A}_{ij}, \tilde{B}_{kl})&=&  \Big[cc_4^2\frac{N^2}{n^2(\ell+m)^2}-\frac{c_4c_5}{m}\frac{N}{n(\ell+m)}\Big]B(\tilde{A}_{ij}, \tilde{A}_{kl})^{2}\,, \\
\Gamma(\tilde{B}_{ij}, \tilde{A}_{kl})&=&\Big[cc_4^2\frac{N^2}{n^2(\ell+m)^2}-\frac{c_4c_5}{m}\frac{N}{n(\ell+m)}\Big]B(\tilde{B}_{ij}, \tilde{B}_{kl})^{2}\,.
\end{eqnarray*}
Consequently 
 \[
 \mathscr{M}_3=2 m n \Big[cc_4^2\frac{N^2}{n^2(\ell+m)^2}-\frac{c_4c_5}{m}\frac{N}{n(\ell+m)}\Big]=\frac{1}{(\ell + m)}\left(m c - \frac{\sqrt{\ell m n}}{\sqrt{N}}\right)
  \]
 and  this completes the proof.
  \end{proof}


\subsection{The Ricci tensor for the diagonal part}
The Ricci tensor for the diagonal part of the $G$-invariant metric $g=\langle \ , \ \rangle$  on $M=G/H$ given by (\ref{ggg}), can be computed by applying 
 Lemma \ref{ric2}, and replacing  the $B$-structure constants by the $\langle \ , \ \rangle$-structure constants.  Since $\bigg\lbrace
\begin{matrix}
3\\
1 \ 2 
\end{matrix}
\bigg\rbrace=\displaystyle{3\brack 1 2}$, in view of (\ref{numbc}) we obtain that
\begin{prop}\label{ricci2}  
The components of the Ricci tensor for the diagonal part of  the $G$-invariant metric on the C-space $M_{\ell, m, n}=G/H$ defined by {\rm (\ref{ggg})},    are given as follows:
\begin{eqnarray*}
\Ric^{g}_1 &=& \frac1{2x_{1}}-\frac1{2d_1 {x_{1}}^2}
\left(v_4\bigg\lbrace
\begin{matrix}
4\\
1 \ 1
\end{matrix}
\bigg\rbrace 
+v_5\bigg\lbrace
\begin{matrix}
5\\
1 \ 1
\end{matrix}
\bigg\rbrace\right)
+\frac1{2d_1}{3\brack {1 \ 2}}
\left(
\frac{{x_{1}}^{2}-{x_{2}}^{2}-{x_{3}}^{2}}{x_1x_2x_3}\right)\,,
\\ 
\Ric^{g}_2 &=& \frac1{2x_{2}}-\frac1{2d_2 {x_{2}}^2}
\left(v_4\bigg\lbrace
\begin{matrix}
4\\
2 \ 2
\end{matrix}
\bigg\rbrace 
+v_5\bigg\lbrace
\begin{matrix}
5\\
2 \ 2
\end{matrix}
\bigg\rbrace\right)
+\frac1{2d_2}{3\brack {1 \ 2}}
\left(
\frac{{x_{2}}^{2}-{x_{1}}^{2}-{x_{3}}^{2}}{x_1x_2x_3}\right)\,,
\\ 
\Ric^{g}_3 &=& \frac1{2x_{3}}-\frac1{2d_3 {x_{3}}^2}
\left(v_4\bigg\lbrace
\begin{matrix}
4\\
3 \ 3
\end{matrix}
\bigg\rbrace 
+v_5\bigg\lbrace
\begin{matrix}
5\\
3 \ 3
\end{matrix}
\bigg\rbrace\right)
+\frac1{2d_3}{3\brack {1 \ 2}} \left(
\frac{{x_{3}}^{2}-{x_{1}}^{2}-{x_{2}}^{2}}{x_1x_2x_3}\right)\,.
\end{eqnarray*}
\end{prop}

Based on the results of  Propositions \ref{ijk} and \ref{mainijk} respectively, we finally obtain  the following expressions.
\begin{corol}\label{ricci3} 
 The components of the Ricci tensor for the diagonal part of the  $G$-invariant metric on the C-space $M_{\ell, m , n}=G/H$ defined by {\rm (\ref{ggg})},    are explicitly given by
\begin{eqnarray*}
 \Ric_1^{g} &=& \frac1{2x_{1}} + \frac{n}{4N} \left(
\frac{{x_{1}}^{2}-{x_{2}}^{2}-{x_{3}}^{2}}{x_1x_2x_3}\right)
  - \frac{(\ell+m)}{4\ell mN}\frac{v_5}{{x_{1}}^2}\,,  \\
    \Ric_2^g&=& \frac1{2x_{2}} + \frac{m}{4N}\left(
\frac{{x_{2}}^{2}-{x_{1}}^{2}-{x_{3}}^{2}}{x_1x_2x_3}\right) \\
&&- \frac{1}{4 \ell n N(\ell+m)}\frac1{{x_{2}}^2}   \bigg( \ell Nv_4+(c^2\ell N+mn+2c\sqrt{\ell m n}\sqrt{N})v_5\bigg)\,, \\
            \Ric^g_3 &=& \frac1{2x_{3}} + \frac{\ell}{4N}\left(
\frac{{x_{3}}^{2}-{x_{1}}^{2}-{x_{2}}^{2}}{x_1x_2x_3}\right)
      \\ && - \frac{1}{4 m n N(\ell+m)}\frac1{{x_{3}}^2}\bigg(m N v_4+(c^{2}m N+ \ell n -2c\sqrt{\ell m n}\sqrt{N})v_5\bigg)\,. 
     \end{eqnarray*}
\end{corol}


\section{Homogeneous Einstein metrics}\label{section5}

\subsection{Existence of an invariant  Einstein metric}\label{existence section}
Consider the C-space $M_{\ell, m, n}=G/H=\SU(\ell+m+n)/\SU(\ell)\times\SU(m)\times\SU(n)$ and let $g$ be the  $G$-invariant metric, defined by (\ref{ggg}).  The homogeneous Einstein equation  $\Ric^{g}=\lambda g$, where $\lambda>0$ is  a positive real number, 
 is equivalent to the following system of equations
 \begin{equation}\label{EINSTEIN}
\left\{\Ric^{g}_1=\lambda\,,  \quad 
\Ric^{g}_2=\lambda \,, \quad 
\Ric^{g}_3=\lambda \,,\quad
\Ric^{g}_4=\lambda \,, \quad 
\Ric^{g}_5=\lambda\,, \quad 
\Ric^{g}_{0}=0
\right\}.
\end{equation}
This is a system of  6 equations and 7 unknowns, namely $x_1, x_2, x_3, v_4, v_5, c$ and the Einstein constant $\lambda$.   One may eliminate $\lambda$   and reduce this system to a system   5 equations and 6 unknowns.  However,  here we will apply  another approach.  As a first step, by Proposition \ref{ricci1} and  the last equation in (\ref{EINSTEIN}), we may express the parameter $c$ in terms of $x_2, x_3$:
\begin{equation}\label{newww}
c=\frac{\sqrt{lmn}}{\sqrt{N}}\frac{({x_{2}}^{2}-{x_{3}}^{2})}{(\ell{x_{3}}^{2}+m{x_{2}}^{2})}=\sqrt{{3\brack {1 \ 2}}}\frac{({x_{2}}^{2}-{x_{3}}^{2})}{(\ell{x_{3}}^{2}+m{x_{2}}^{2})}\,.
\end{equation}
Thus, by Proposition \ref{mainijk} we can specify explicitly the $\langle \ , \ \rangle$-structure constants, depending on $c$. In particular
 \begin{corol}\label{mainijk1} 
 The values of $\displaystyle\bigg\lbrace
\begin{matrix}
5\\
2 \  2
\end{matrix}
\bigg\rbrace$ and $\displaystyle\bigg\lbrace
\begin{matrix}
5\\
3 \  3
\end{matrix}
\bigg\rbrace$ are given by 
\begin{equation}
\displaystyle\bigg\lbrace
\begin{matrix}
5\\
2 \  2
\end{matrix}
\bigg\rbrace = \frac{m (\ell+ m) n\,{x_2}^4}{N(\ell \,{x_3}^2 + m\,{x_2}^2)^2}  \quad  \mbox{and } \quad 
 \displaystyle\bigg\lbrace 
\begin{matrix}
5\\
3 \  3
\end{matrix}
\bigg\rbrace = \frac{\ell(\ell+ m) n\,{x_3}^4}{N(\ell \,{x_3}^2 + m\,{x_2}^2)^2}\,,
\end{equation}
respectively. 
\end{corol}
Now we can present   the system corresponding to the homogeneous Einstein equation.
\begin{prop}\label{newricci1}  
The system of the homogeneous Einstein equation on the C-space $(M_{\ell, m, n}, g=\langle \ , \ \rangle)$ is given as follows:
\begin{eqnarray}
\Ric_4^g - \lambda &=& \frac{v_4}{4}\frac{\ell \,{x_3}^2 + m\,{x_2}^2}{(\ell+m){x_2}^2 \,{x_3}^2}- \lambda =0\,, \label{eq_r4}
\\ 
\Ric_5^g -\lambda &=& \frac{v_5}{4}\frac{(\ell+m)(n {x_1}^2+m {x_2}^2+ \ell {x_3}^2)}{N{x_1}^2 (\ell {x_3}^2+m {x_2}^2)}-\lambda =0\,,\label{eq_r5}
\\
 \Ric_1^{g}-\lambda &=& \frac1{2x_{1}} + \frac{n}{4N} \left(
\frac{{x_{1}}^{2}-{x_{2}}^{2}-{x_{3}}^{2}}{x_1x_2x_3}\right) 
  - \frac{(\ell+m)}{4\ell mN}\frac{v_5}{{x_{1}}^2}-\lambda=0\,,  \label{eq_ric1}\\
    \Ric_2^g -\lambda &=& \frac1{2x_{2}} + \frac{m}{4N}\left(
\frac{{x_{2}}^{2}-{x_{1}}^{2}-{x_{3}}^{2}}{x_1x_2x_3}\right) \nonumber \\
& &- \ v_4\frac{1}{4 (\ell+m) n \,{x_{2}}^2} - v_5\frac{m (\ell+ m)\,{x_2}^2}{4 \ell N(\ell \,{x_3}^2 + m\,{x_2}^2)^2}  -\lambda=0\,, \label{eq_ric2}\\
           \Ric^g_3 -\lambda &=& \frac1{2x_{3}} + \frac{\ell}{4N}\left(
\frac{{x_{3}}^{2}-{x_{1}}^{2}-{x_{2}}^{2}}{x_1x_2 x_3}\right) \nonumber\\
& & - \ v_4 \frac{1}{4 (\ell+m) n \,{x_{3}}^2} - v_5\frac{\ell (\ell+m) {x_3}^2}{4 m N(\ell \,{x_3}^2 + m\,{x_2}^2)^2}-\lambda =0\,. \label{eq_ric3}
    \end{eqnarray}
\end{prop}
As an immediate consequence of the expressions of the above equations, we deduce  that
\begin{corol}\label{linearEIN}
The homogeneous Einstein equation on $(M_{\ell, m, n}, g=\langle \ , \ \rangle)$ is linear on the variables $v_4, v_5$.
\end{corol}
  Due to this linearity, we may use  the  equations (\ref{eq_r4})  and (\ref{eq_r5})  and solve them with respect to $v_4$ and $v_5$, respectively.  We obtain
 \begin{eqnarray}
 v_4&=& 4 \lambda \frac{(\ell+m){x_2}^2 \,{x_3}^2}{\ell \,{x_3}^2 + m\,{x_2}^2}\,, \label{eq_v4}\\
 v_5&= & 4 \lambda \frac{N{x_1}^2 (\ell {x_3}^2+m {x_2}^2)}{(\ell+m)(n {x_1}^2+m {x_2}^2+ \ell {x_3}^2)}\,, \label{eq_v5}
 \end{eqnarray}
so knowing $x_1, x_2, x_3$ and the corresponding Einstein constant, we also know the values of $v_4$ and $v_5$. In the following, this observation will be applied several times.

By substituting (\ref{eq_v4}) and (\ref{eq_v5})  into (\ref{eq_ric1}),   (\ref{eq_ric2}) and  (\ref{eq_ric3}), we result with the equations
 \begin{eqnarray*} 
 \Ric_1^{g}-\lambda &=& \frac1{2x_{1}} + \frac{n}{4N} \left(
\frac{{x_{1}}^{2}-{x_{2}}^{2}-{x_{3}}^{2}}{x_1x_2x_3}\right) 
  -  \frac{\lambda \,(\ell {x_3}^2+m {x_2}^2)}{\ell m (n {x_1}^2+m {x_2}^2+ \ell {x_3}^2)}-\lambda=0\,,\\ 
    \Ric_2^g -\lambda &=& \frac1{2x_{2}} + \frac{m}{4N}\left(
\frac{{x_{2}}^{2}-{x_{1}}^{2}-{x_{3}}^{2}}{x_1x_2x_3}\right) \nonumber \\
&& -  \frac{\lambda \, {x_3}^2}{n(\ell {x_3}^2+m {x_2}^2)} - \frac{\lambda \, m {x_1}^2\,{x_2}^2}{\ell (\ell \,{x_3}^2 + m\,{x_2}^2) (n {x_1}^2+m {x_2}^2+ \ell {x_3}^2)}  -\lambda=0\,,\\ 
          \Ric^g_3 -\lambda &=& \frac1{2x_{3}} + \frac{\ell}{4N}\left(
\frac{{x_{3}}^{2}-{x_{1}}^{2}-{x_{2}}^{2}}{x_1x_2x_3}\right) \nonumber\\
&&  -   \frac{\lambda \, {x_3}^2}{n(\ell {x_3}^2+m {x_2}^2)} - \frac{\lambda \, \ell {x_1}^2\,{x_3}^2}{m (\ell \,{x_3}^2 + m\,{x_2}^2) (n {x_1}^2+m {x_2}^2+ \ell {x_3}^2)}-\lambda =0\,. 
 \end{eqnarray*}
We may solve any of these three equations with respect to $\lambda$, a procedure which gives 
{\small \[
\lambda = \frac{\displaystyle\frac1{2x_{1}} + \frac{n}{4N} \left(
\frac{{x_{1}}^{2}-{x_{2}}^{2}-{x_{3}}^{2}}{x_1x_2x_3}\right) }{\displaystyle 
  1+ \frac{m {x_2}^2+\ell {x_3}^2}{\ell m (n {x_1}^2+m {x_2}^2+ \ell {x_3}^2)}}\,,  \ 
\lambda = \frac{\displaystyle \frac1{2x_{2}} + \frac{m}{4N}\left(
\frac{{x_{2}}^{2}-{x_{1}}^{2}-{x_{3}}^{2}}{x_1x_2x_3}\right) }{\displaystyle 
 1+ \frac{n {x_1}^2+\ell {x_3}^2}{\ell n (n {x_1}^2+m {x_2}^2+ \ell {x_3}^2)}}\,, \
 \lambda = \frac{\displaystyle \frac1{2x_{3}} + \frac{\ell}{4N}\left(
\frac{{x_{3}}^{2}-{x_{1}}^{2}-{x_{2}}^{2}}{x_1x_2x_3}\right)}{\displaystyle 1+ \frac{n {x_1}^2+m{x_2}^2}{m n (n {x_1}^2+m {x_2}^2+ \ell {x_3}^2)}}\,,
\]}
respectively.

Now, consider  the flag manifold  $F_{\ell, m, n}=G/K=\SU(\ell+m+n)/\Ss(\U(\ell)\times\U(m)\times\U(n))$. Any $G$-invariant metric $\check g$ on $F_{\ell, m, n}$ 
is  given by an $\Ad(K)$-invariant inner product  $\langle \langle \ , \ \rangle\rangle$ on $\fr{f}=\fr{f}_1\oplus\fr{f}_2\oplus\fr{f}_3$  of  the form (see also \cite{Kim})
 \begin{equation}\label{flg}
\langle \langle \ , \ \rangle\rangle =x_{1}B|_{\fr{f}_1}+x_2B|_{\fr{f}_2}+x_3B|_{\fr{f}_3}\,,\quad (x_1, x_2, x_3)\in\bb{R}^{3}_{+}\,,
\end{equation}
where  ${\mathbb R}_+^{3}:= \{( x_1, x_2, x_3) \in {\mathbb R}^3\,  |\,   x_1 > 0, x_2 > 0, x_3 > 0 \}$ (and from now on we will identify $\check g$ and $\langle \langle \ , \ \rangle\rangle$).
It is easy to see that  the components $\Ric_i^{\check g}$ ($ i =1, 2, 3$) of the Ricci tensor  on $(F_{\ell, m, n}, {\check g})$ are given by 
 \begin{eqnarray*} 
\Ric_1^{\check g}&=& \displaystyle\frac1{2x_{1}} + \frac{n}{4N} \left(
\frac{{x_{1}}^{2}-{x_{2}}^{2}-{x_{3}}^{2}}{x_1x_2x_3}\right)\,,  \\
\Ric_2^{\check g} &=&  \displaystyle \frac1{2x_{2}} + \frac{m}{4N}\left(
\frac{{x_{2}}^{2}-{x_{1}}^{2}-{x_{3}}^{2}}{x_1x_2x_3}\right)\,, \\
\Ric_3^{\check g}&=&  \displaystyle \frac1{2x_{3}} + \frac{\ell}{4N}\left(
\frac{{x_{3}}^{2}-{x_{1}}^{2}-{x_{2}}^{2}}{x_1x_2x_3}\right)\,.
 \end{eqnarray*} 
Thus,  the rational polynomials of $x_1, x_2, x_3$ appearing in the right hand side of the above expressions of $\lambda$ can    be expressed by
\begin{eqnarray}
 \mathsf{t}_1(x_1, x_2, x_3)&=&\frac{\Ric^{\check g}_1(x_1, x_2, x_3)}{1+\mathsf{q}_{1}(x_1, x_2, x_3)}\,, \label{t_1(x_1, x_2, x_3)}\\
 \mathsf{t}_2(x_1, x_2, x_3)&=&\frac{\Ric^{\check g}_{2}(x_1, x_2, x_3)}{1+\mathsf{q}_{1}(x_1, x_2, x_3)}\,,\label{t_2(x_1, x_2, x_3)} \\
  \mathsf{t}_3(x_1, x_2, x_3)&=&\frac{\Ric^{\check g}_{3}(x_1, x_2, x_3)}{1+\mathsf{q}_{1}(x_1, x_2, x_3)}\,,\label{t_3(x_1, x_2, x_3)}
 \end{eqnarray}
 where $\mathsf{q}_{1}(x_1, x_2, x_3)=\displaystyle\frac{m {x_2}^2+\ell {x_3}^2}{\ell m (n {x_1}^2+m {x_2}^2+ \ell {x_3}^2)}$, $\mathsf{q}_2(x_1, x_2, x_3)=\displaystyle\frac{n {x_1}^2+\ell {x_3}^2}{\ell n (n {x_1}^2+m {x_2}^2+ \ell {x_3}^2)}$ and $\mathsf{q}(x_1,  x_2, x_3)=\displaystyle\frac{n {x_1}^2+m{x_2}^2}{m n (n {x_1}^2+m {x_2}^2+ \ell {x_3}^2)}$, respectively.
We also set
\[
T_i(t, x_1, x_2, x_3):=\frac{\Ric^{\check g}_i(x_1, x_2, x_3)}{1+t\cdot \mathsf{q}_{i}(x_1, x_2, x_3)}\,, \quad i=1, 2, 3\,,
\]
such that $T_i(0, x_1, x_2, x_3)=\Ric^{\check g}_i(x_1, _2, x_3)$ and $T_i(1, x_1, x_2, x_3)=\mathsf{t}_i(x_1, x_2,  x_3)$, for any $i=1, 2, 3$. Since the Ricci tensor of  $\check g$ is not zero and $\mathsf{q}_{i}(x_1, x_2, x_3)>0$ for  any $i$,  we see that 
 \[
 (T_1(t, x_1, x_2, x_3), T_2(t, x_1, x_2, x_3), T_3(t, x_1, x_2, x_3)) \neq (0, 0, 0)\,.
 \] 
Hence, for each $t \in  [0, \, 1]$ we may introduce a map    $F_t : {\mathbb R}_+^3  \to \Ss^{2}$, where $\Ss^2$ is the 2-sphere, given by
 \[
 F_t(x_1, x_2, x_3):= \frac{1}{\sqrt{{T_1}^2+{T_2}^2+{T_3}^2}} (T_1(t, x_1, x_2, x_3), T_2(t, x_1, x_2, x_3), T_3(t, x_1, x_2, x_3))\,.
 \] 
 This is a well defined homotopy, with 
 \begin{eqnarray*}
 F_{0}&=&\frac{1}{\sqrt{(\Ric^{\check g}_1)^{2}+(\Ric^{\check g}_2)^{2}+(\Ric^{\check g}_3)^{2}}}(\Ric^{\check g}_1, \Ric^{\check g}_2, \Ric^{\check g}_3)\,,\\
 F_{1}&=&\frac{1}{\sqrt{(\mathsf{t}_1)^{2}+(\mathsf{t}_2)^{2}+(\mathsf{t}_3)^{2}}}(\mathsf{t}_1, \mathsf{t}_2, \mathsf{t}_3)\,.
 \end{eqnarray*}

\begin{prop} \label{characterEIN}
Put $p_0:=(1/\sqrt{3},  1/\sqrt{3}, 1/\sqrt{3})$. Then \\ 
 1) A point  $({x_1}, {x_2}, {x_3})\in\bb{R}^{3}_{+}$ corresponds to a $G$-invariant Einstein metric on the flag manifold  $F_{\ell, m, n}=G/K$, if and only if $({x_1}, {x_2}, {x_3}) \in {F_0}^{-1}(p_0)$. \\
 2) A point $({x_1}, {x_2}, {x_3})\in\bb{R}^{3}_{+}$ corresponds to a $G$-invariant Einstein metric on the $C$-space $M_{\ell, m, n}=G/H$, if and only if $({x_1}, {x_2}, {x_3}) \in {F_1}^{-1}(p_0)$. 
 \end{prop}
\begin{proof}
  Note  that the Ricci tensor of the $G$-invariant metric $\check{g}$  on $F_{\ell, m, n}=G/K$ cannot be negative definite, since $G/K$ is simply connected and compact. This implies that the inequalities $T_i(t, x_1, x_2, x_3) < 0$  are impossible, for any $i=1, 2, 3$ and $t\in[0, 1]$.  In particular,   
\[
(T_1(t, x_1, x_2, x_3),  T_2(t, x_1, x_2, x_3), T_3(t, x_1, x_2, x_3)) \neq  -p_0 = (-1/\sqrt{3},  -1/\sqrt{3}, -1/\sqrt{3})\,.
\]
Now, a point $({x_1}, {x_2}, {x_3})$ belongs to  the level set   $F_0^{-1}(p_{0})$, if and only if
\[
\Ric_1^{\check g}({x_1}, {x_2}, {x_3}) = \Ric_2^{\check g}({x_1}, {x_2}, {x_3})=\Ric_3^{\check g}({x_1}, {x_2}, {x_3})\,.
\]
If a point  $({x_1}, {x_2}, {x_3})\in\bb{R}^{3}_{+}$ corresponds  to a $G$-invariant Einstein metric on the C-space $(M_{\ell, m, n}=G/H,  g=\langle \ , \ \rangle)$, then we obtain $\mathsf{t}_1(x_1, x_2, x_3)=\mathsf{t}_2(x_1, x_2, x_3)=\mathsf{t}_3(x_1, x_2, x_3)$ and thus $F_{1}(x_1, x_2, x_3)=p_{0}$. Conversely, if a point $({x_1}, {x_2}, {x_3})$  belongs to the level set $F^{-1}(p_{0})$, then we have $\mathsf{t}_1(x_1, x_2, x_3)=\mathsf{t}_2(x_1, x_2, x_3)=\mathsf{t}_3(x_1, x_2, x_3)=\lambda$, and by (\ref{eq_v4}) and (\ref{eq_v5}) we see that (\ref{eq_ric1}), (\ref{eq_ric2}) and (\ref{eq_ric3}) hold. Therefore,  the point  $({x_1}, {x_2}, {x_3})\in\bb{R}^{3}_{+}$ corresponds to a $G$-invariant Einstein metric on  $M_{\ell, m, n}$.
\end{proof}

\begin{remark}\label{homogpolyn}
  The rational polynomials $\mathsf{t}_1, \mathsf{t}_2, \mathsf{t}_3$ defined by (\ref{t_1(x_1, x_2, x_3)}), (\ref{t_2(x_1, x_2, x_3)}) and (\ref{t_3(x_1, x_2, x_3)}), respectively, are homogeneous of degree -1. Thus, for some $\mu > 0$ and for any $i=1, 2, 3$ we obtain  the relation 
\[
T_i(t, \mu x_1, \mu x_2, \mu x_3) = \mu^{-1}T_i(t, x_1, x_2, x_3)\,.
\]
  This means that along the procedure of describing  $G$-invariant Einstein metrics on the $C$-space $M_{\ell, m, n}=G/H$ and the flag manifold  $F_{\ell, m, n}=G/K$, one may assume (without loss of generality) the normalization  $x_3 = 1$.
\end{remark}
An application of this method on the base space $F_{\ell, m, n}$   gives that    (see also \cite{Kim} for the Einstein metrics)
\begin{prop}\label{flag_einstein} 
There are four $G$-invariant Einstein metrics on the flag manifold  $F_{\ell, m, n}=G/K = \SU(\ell+m+n)/\Ss(\U(\ell)\times\U(m)\times\U(n))$, given by
{\small
\[
 \left(\frac{\ell+m}{m+n}, \frac{\ell+n}{m+n}, 1\right)\,, \, \left(\frac{\ell+m}{m+n},\frac{\ell+2 m+n}{m+n},1\right)\,, \,  \left(\frac{\ell+m+2 n}{m+n}, \frac{\ell+n}{m+n}, 1\right)\,, \, \left(\frac{l+m}{2 \ell+m+n}, \frac{\ell+n}{2 \ell+m+n}, 1\right).
\]}
Moreover, the Einstein constants $\lambda$ for these Einstein metrics are given by 
 \begin{eqnarray*} \ \ 
\lambda = \frac{m \ell^2+n \ell^2+m^2 \ell+n^2 \ell+4 m n \ell+m n^2+m^2 n}{2 (\ell+m) (\ell+n) (\ell+m+n)}\,,\,\,  \frac{m+n}{2 (\ell+m+n)}\,, \,\,  \frac{2 \ell+m+n}{2 (\ell+m+n)},\, \,  \frac{m+n}{2 (\ell+m+n)}\,,
  \end{eqnarray*} 
  respectively.
\end{prop}
From now on let us use the notation $D_{+}\subset\R^3$ for the domain  $ \{ ( x_1, x_2, 1) \in {\mathbb R}^3\,  |\,   x_1 > 0, \, x_2 > 0 \}$. 
Consider the rotation matrix  $(\det P=1)$
\[
P := \begin{pmatrix}
1/\sqrt{2} & -1/\sqrt{2} & 0\\
1/\sqrt{6} & 1/\sqrt{6} & -2/\sqrt{6}\\
1/\sqrt{3} & 1/\sqrt{3} & 1/\sqrt{3}
\end{pmatrix}
\]
which maps the point $p_0=(1/\sqrt{3},  1/\sqrt{3}, 1/\sqrt{3})$ to $(0, 0, 1)$. Consider also  the
stereographic projection $\psi : \Ss^2-{(0,0,-1)} \to {\mathbb R}^2$ defined by $\psi(X, Y, Z) = \left(\frac{X}{1+Z}, \frac{Y}{1+Z}\right)$. 
By using $\psi$ we  define a map $f_t^{} : D_{+} \to {\mathbb R}^2$, by
\[
f_t^{}(x_1, x_2, 1):= \psi\circ P \circ F_t(x_1, x_2, 1)\,. 
\]
By Proposition \ref{characterEIN} and the definition of $f_{t}$ it follows that
\begin{corol}
1) A point $ ({x_1}, {x_2}, 1)\in D_{+}$  corresponds to a $G$-invariant Einstein metric  on the flag manifold $F_{\ell, m ,n}=G/K$,  if and only if $f_0^{}({x_1}, {x_2}, 1) =(0, 0)$.  \\
2)  A point $ ({x_1}, {x_2}, 1)\in D_{+}$ corresponds to a $G$-invariant Einstein metric  on the C-space  $M_{\ell, m, n}=G/H$, if and only if  $f_1^{}({x_1}, {x_2}, 1) =(0, 0)$.
 \end{corol}

We can now apply mapping degree theory with respect to $f_{t}$, to show existence of Einstein metrics on the C-space $M_{\ell, m, n}=G/H=\SU(\ell+m+n)/\SU(\ell)\times\SU(m)\times\SU(n)$.
But first,  let  us  recall by \cite{Ruiz}   necessary material from  mapping degree theory.

Let $D$ be a bounded open subset in ${\mathbb R}^n$,  $f : {\overline D}  \to {\mathbb R}^n$ a continuous map and $ y \in {\mathbb R}^n - f(\partial D)$. Such a triple ($f, D, y$) is called {\it admissible} and we may  consider a   map  $\dg : \{(f, D, y) :  \text{admissible}\} \to {\mathbb Z}$, 
such that:

(1) {\it Homotopy invariance}: For every bounded open set $D \subset  {\mathbb R}^n $ and all continuous mappings $F: [0, \, 1]\times {\overline D} \to {\mathbb R}^n$ and $\gamma : [0, \, 1] \to {\mathbb R}^n$ satisfying 
$$\gamma(t) \in {\mathbb R}^n - F(\{t\}\times \partial D) \quad \mbox{for} \quad  0 \leq t \leq 1\,,$$ 
the following formula holds:
$$\dg(F(t, \cdot), D, \gamma(t)) = \dg(F(0, \cdot), D, \gamma(0))\,,\quad \mbox{for \ any} \quad  0 \leq t \leq 1 \,.$$

(2) {\it Normality}: For every bounded open set $D \subset  {\mathbb R}^n $ and every point $p \in D$, 
$$\dg( \mbox{Id}_{\overline D}, D, p) = 1\,. $$

(3) {\it Additivity}: For every bounded open set $D \subset  {\mathbb R}^n$,  every pair of two disjoint open sets $D_1$,  $D_2 \subset D$, every continuous mapping $f : {\overline D}  \to {\mathbb R}^n$ , and every point $p \not\in {\mathbb R}^n - f(\overline D -  D_1 \cup D_2)$, 
$$\dg(f, D, p) = \dg(f, D_1, p) + \dg(f, D_2, p)\,.$$

(4) {\it Existence of solutions}: For every bounded open set $D \subset  {\mathbb R}^n$,  every continuous mapping $f : {\overline D}  \to {\mathbb R}^n$ , and every point $p \in {\mathbb R}^n - f(\partial D)$, such that $\dg(f, D, p) \neq 0$, the equation $f(x)= p$ has a solution  in $D$. 

Moreover, let $f : {\overline D}  \to {\mathbb R}^n$ be a smooth mapping,  and let $a \in   {\mathbb R}^n - f(\partial D)$ be a regular value of $f\vert_D$. Then $f^{-1}(a)$ is a finite set (possibly empty), and  we have 
\[
\dg(f, D, a) =
\left\{
\begin{tabular}{lr}
$\sum_{x \in f^{-1}(a)} {\rm sign}_x(f)$\,, &  if  \ $f^{-1}(a)\neq\emptyset$\,,\\
0\,,  &  if \ $f^{-1}(a)=\emptyset$\,,
\end{tabular}\right.
\]
where ${\rm sign}_x(f)$ is the  sign of $f$ at $x$,  which by definition is the sign $\pm  1$ of the Jacobian $\det \left( \frac{\partial  f_i}{\partial x_j}(x)\right)$.

\begin{lemma}\label{degree_flag_einstein} 
The point $(0, 0) \in \mathbb{R}^2$ is a regular value of the mapping $f_0 : D_{+} \to {\mathbb R}^2$. 
The Jacobians of $f_0$ at points  ${f_0}^{-1}((0, 0))$ are given as follows: 
{\small \begin{eqnarray*} \frac{\ell m n (m+n)^3}{6 \sqrt{3} \left(\ell^2 (m+n)+\ell \left(m^2+4 m n+n^2\right)+m n (m+n)\right)^2}\,,&\text{for}&  (x_1,x_2) = \left(\frac{\ell+m}{m+n}, \frac{\ell+n}{m+n}\right)\,,\\
-\frac{m (m+n)^2}{6 \sqrt{3} (\ell+m) (\ell+2 m+n)^2}  \,,&\text{for}& (x_1,x_2) = \left(\frac{\ell+m}{m+n}, \frac{\ell+2 m +n}{m+n}\right)\,,\\
-\frac{\ell (2 \ell+m+n)}{6 \sqrt{3} (\ell+m) (\ell+n)} \,,&\text{for}&  (x_1,x_2) = \left(\frac{\ell+m}{2 \ell+m+n}, \frac{\ell +n}{2 \ell+m+n}\right)\,,\\ 
-\frac{n (m+n)^2}{6 \sqrt{3} (\ell+n) (\ell+m+2 n)^2}  \,,&\text{for}&  (x_1,x_2) = \left(\frac{\ell+m+2 n}{m+n}, \frac{\ell +n}{m+n}\right)\,. 
 \end{eqnarray*}   }
\end{lemma}
\begin{proof} 
Put $h(x_1, x_2) =P \circ F_0(x_1, x_2, 1)$.  Then, a direct computations shows  that $h(x_1, x_2)$ is expressed by
\begin{eqnarray*}
h(x_1, x_2)&=& \frac{1}{\sqrt{\sum_{i=1}^3{\Ric_i^{\check g}}^2}}\left(\frac{1}{\sqrt{2}}(\Ric_1^{\check g}-\Ric_2^{\check g})\,, \frac{1}{\sqrt{6}}(\Ric_1^{\check g}+\Ric_2^{\check g}-2 \Ric_3^{\check g})\,,  \frac{1}{\sqrt{3}}(\Ric_1^{\check g}+\Ric_2^{\check g}+\Ric_3^{\check g})\right)\\
&=&\big(h_{1}(x_1, x_2)\,, h_{2}(x_1, x_2)\,, h_{3}(x_1, x_2)\big)\,.
\end{eqnarray*}
  Moreover, by the definition of $f_{t}$ we have $f_{0}:=\psi\circ P\circ F_0(x_1, x_2, 1)=\psi\circ h(x_1, x_2)$, and hence 
\[
f_0^{}(x_1, x_2, 1)  = \left( \frac{h_1(x_1, x_2)}{1+ h_3(x_1, x_2)}, \frac{h_2(x_1, x_2)}{1+ h_3(x_1, x_2)}\right).
\]
For a point $(x_1, x_2, 1) \in {f_0}^{-1}((0, 0))$,  it  holds $\Ric_1^{\check g}(x_1, x_2, 1) = \Ric_2^{\check g}(x_1, x_2, 1)=\Ric_3^{\check g}(x_1, x_2, 1)$ and $h(x_1, x_2) = (0,  0, 1)$.  Thus, the Jacobian 
{\small
\[
\mathscr{J}:=\frac{\partial}{\partial x_1}\Big(\frac{h_1(x_1, x_2)}{1+h_3(x_1, x_2)}\Big)\frac{\partial}{\partial x_2}\Big(\frac{h_2(x_1, x_2)}{1+h_3(x_1, x_2)}\Big)-\frac{\partial}{\partial x_2}\Big(\frac{h_1(x_1, x_2)}{1+h_3(x_1, x_2)}\Big)\frac{\partial}{\partial x_1}\Big(\frac{h_2(x_1, x_2)}{1+h_3(x_1, x_2)}\Big)
\]}
of $f_{0}$ at the point   $\displaystyle (x_1,x_2) = \left(\frac{\ell+m}{m+n}, \frac{\ell+n}{m+n}\right)$   is given by
{\small
\[
\mathscr{J}=\frac{1}{4} \left(\frac{\partial h_1}{\partial x_1}\frac{\partial h_2}{\partial x_2}-\frac{\partial h_1}{\partial x_2}\frac{\partial h_2}{\partial x_1}\right)
    =\frac{\ell m n (m+n)^3}{6 \sqrt{3} \left(\ell^2 (m+n)+\ell \left(m^2+4 m n+n^2\right)+m n (m+n)\right)^2}\,.
  \]}
In a  similar way we obtain the values of the Jacobian of $f_0$ at the rest points. 
\end{proof}

Now, for $\varepsilon > 0$ and $L > 0$,  we now define a subdomain $D_{\varepsilon, L}$ of $D_+$ by
\[
D_{\varepsilon, L}:= \{ ( x_1, x_2, 1) \in {\mathbb R}^3\,  |\,  \varepsilon <  x_1 <L,  \, \,  \varepsilon <  x_2 < L \}\,.
\]
The boundary  $\partial D_{\varepsilon, L}$  of the domain  $D_{\varepsilon, L}$  has the form
\begin{eqnarray*}
 \partial D_{\varepsilon, L}&=&\{ ( \varepsilon, x_2, 1) \in {\mathbb R}^3\,  |\,  \varepsilon \leq x_2 \leq L \}\cup 
\{ ( L, x_2, 1) \in {\mathbb R}^3\,  |\, \varepsilon \leq  x_2 \leq L \}\cup \\&& \{ ( x_1, \varepsilon, 1) \in {\mathbb R}^3\,  |\, \varepsilon \leq  x_1 \leq L \}\cup \{ ( x_1, L, 1) \in {\mathbb R}^3\,  |\, \varepsilon \leq  x_1 \leq L \}\,.
\end{eqnarray*}
Based on this description of $\partial D_{\varepsilon, L}$, and by  combining basic  elimination theory (e.g. resultants of algebraic equations, see \cite[Ch.~3]{Cox}), we obtain the following 

\begin{prop}\label{exist_varepsilon_L} 
  For positive integers $\ell, m, n$, 
there exist $\varepsilon = \varepsilon(\ell, m, n) > 0$ and $L= L(\ell, m, n) > 0$ such that 
smooth mappings 
$f_t^{}: [0, \, 1]\times {\overline D_{\varepsilon, L}} \to {\mathbb R}^2$ satisfies 
\[
(0, 0) \not\in f_t^{}(\partial D_{\varepsilon, L}) \quad \mbox{for} \quad  0 \leq t \leq 1\,.
\] 
\end{prop}
\begin{proof}
Note that $(0, 0) \in f_t^{}(\partial D_{\varepsilon, L})$ if and only if there exist a point $(x_1,x_2,1) \in \partial D_{\varepsilon, L}$ such that $T_1(t, x_1, x_2, 1) = T_2(t, x_1, x_2, 1)= T_3(t, x_1, x_2, 1)$. We consider the equations 
\begin{eqnarray}\label{T1-T2_T2-T3} 
T_1(t, x_1, x_2, 1) - T_2(t, x_1, x_2, 1)=0, \quad T_2(t, x_1, x_2, 1)-T_3(t, x_1, x_2, 1)=0. 
\end{eqnarray}
Then, a direct computation shows that the equations (\ref{T1-T2_T2-T3}) are equivalent to the equations 
{\small
\begin{eqnarray*}
T_{12}&:=& -m n {x_2}^2 \left(\ell^2 m+\ell^2 n-\ell m^2+\ell m
   n+(2 \ell -m) t\right)+m n {x_1}^2 \left(\ell^2 m+\ell^2
   n+\ell m n-\ell n^2+(2 \ell -n) t\right)\\
 &&  +\ell^2 m n (m-n)-m^2 n
   {x_2}^4 (\ell m+\ell n+t)+2 \ell m^2 n {x_2}^3
   (\ell+m+n)+m n^2 {x_1}^4 (\ell m+\ell n+t)
   \\ 
   && -2 \ell m n^2
   {x_1}^3 (\ell+m+n)+m n {x_1}^2 {x_2}^2
   (m-n) (\ell m+\ell n+t)+2 m n {x_1}^2 {x_2}
   (\ell+m+n) (\ell n+t) 
   \\ && -2 m n {x_1} {x_2}^2
   (\ell+m+n) (\ell m+t)-2 \ell n {x_1} (\ell+m+n) (\ell m+t)+2
   \ell m {x_2} (\ell+m+n) (\ell n+t)=0\,,\\
     T_{23}&:=&   \ell m {x_1}^2 \left(\ell^2 n-\ell m n-\ell n^2-m n^2+(\ell -2 n) t\right)-\ell^2
   m (\ell n+m n+t)+2 \ell^2 m n {x_1} (\ell+m+n)
   \\
    &&+\ell m
   {x_1}^2 {x_2}^2 \left(\ell m n+\ell n^2-m^2 n+m n^2+(2 n -m) t\right)+\ell m^2 {x_2}^4 (\ell
   n+m n+t)\\
   &&-2 \ell m^2 n {x_1} {x_2}^3 (\ell+m+n)+\ell
   m n^2 {x_1}^4 (\ell-m)-2 m n {x_1}^3
   {x_2} (\ell+m+n) (\ell n+t)
   \\ &&+2 \ell n {x_1}^3
   (\ell+m+n) (m n+t)+2 \ell m {x_1} {x_2}^2
   (\ell+m+n) (m n+t)
   \\&&-2 \ell m {x_1} {x_2} (\ell+m+n)
   (\ell n+t)+\ell m {x_2}^2 (\ell-m) (\ell n+m n+t)=0\,,
   \end{eqnarray*}}
respectively. We will show that there exist (small) $\varepsilon=\varepsilon(\ell, m, n) > 0$ and (large) $L=L(\ell, m, n) > 0$, such that  $T_{12}=0$ and $T_{23}=0$ have no common roots on the boundary  $\partial D_{\varepsilon, L}$. 
   For this, let us denote by $R(T_{12},T_{23}, x_2)$ the resultant  of $T_{12}$ and $T_{23}$, with respect to  $x_2$.  Note that if $T_{12}=0$ and  $T_{23}=0$ have a common root $(x_1, x_2)=(a_1, a_2)$, then $R(T_{12},T_{23}, x_2)=0$ at  $x_1=a_1$ (see \cite{Cox}). With the aid of Mathematica, we compute 
\[
R(T_{12},T_{23}, x_2) = t^2\sum_{k=0}^{16}a_k(\ell, m, n, t){x_1}^k\,,
\]
   where $a_k(\ell, m, n, t)$ are polynomials of $\ell, m, n$ and  $t$, with
{\small \begin{eqnarray*}
a_0(\ell, m, n, t) &=& -4 \ell^8 m^{10} (\ell+m)^2 (\ell+m+2 n) \left(4 \ell m (\ell+m+2n)+(\ell +m) n^2\right) ((\ell+m) n+t)^6\,,\\
a_{16}(\ell, m, n, t) &=&  -4 m^{10} n^8 (m+n)^2 (2 \ell+m+n) \left(\ell^2 (m+n)+4 m n   (2 \ell+m+n)\right) (\ell (m+n)+t)^6\,.
   \end{eqnarray*}}
\noindent  Since $a_0(\ell, m, n, t) \neq 0$ and $a_{16}(\ell, m, n, t) \neq 0$, there exist (small) $\varepsilon'(\ell, m, n) > 0$ and (large) $L'(\ell, m, n) > 0$ such that $R(T_{12},T_{23}, x_2) \neq 0$ for $0 < t \leq1$,  $0< x_1 \leq \varepsilon'(\ell, m, n)$ and $L'(\ell, m, n)  \leq x_1$.  
  Let us also  denote by $R(T_{12},T_{23}, x_1)$ the resultant  of $T_{12}$ and $T_{23}$,   with respect to $x_1$. Then we compute
$$R(T_{12},T_{23}, x_1) = t^2\sum_{k=0}^{16}b_k(\ell, m, n, t){x_2}^k $$
   where $b_k(\ell, m, n, t)$ are polynomials of $\ell, m, n$ and  $t$, with
  {\small
  \begin{eqnarray*}
  b_0(\ell, m, n, t) &=& -4 \ell^6 m^2 n^6 (\ell+n)^3 (\ell+2 m+n) \left(m^2
   (\ell+n)+4 \ell n (\ell+2 m+n)\right)^2 (\ell n+t)^2 (m (\ell+n)+t)^4\,, \\
  b_{16}(\ell, m, n, t) &=&  -4 m^{10} n^8 (m+n)^2 (2 \ell+m+n) \left(\ell^2 (m+n)+4
   m n (2 \ell+m+n)\right) (\ell (m+n)+t)^6\,.
   \end{eqnarray*}}
   Since $b_0(\ell, m, n, t) \neq 0$ and $b_{16}(\ell, m, n, t) \neq 0$, there exist (small) $\varepsilon''(\ell, m, n) > 0$ and (large) $L''(\ell, m, n) > 0$ such that $R(T_{12},T_{23}, x_1) \neq 0$ for $0 < t \leq1$,  $0< x_2 \leq \varepsilon''(\ell, m, n)$ and $L''(\ell, m, n)  \leq x_2$.  
   For $t=0$ we have four solutions by Proposition \ref{flag_einstein}. 
Thus we see that there exist $\varepsilon=\varepsilon(\ell, m, n) > 0$ and $L=L(\ell, m, n) > 0 $ such that  the equations 
$T_1(t, \varepsilon, x_2, 1) = T_2(t, \varepsilon, x_2, 1)= T_3(t, \varepsilon, x_2, 1)$, 
$T_1(t, L, x_2, 1) = T_2(t, L, x_2, 1)= T_3(t, L, x_2, 1)$,
$T_1(t, x_1,\varepsilon, 1) = T_2(t,x_1, \varepsilon, 1)= T_3(t, x_1,\varepsilon, 1)$ and 
$T_1(t, x_1, L,  1) = T_2(t, x_1, L, 1)= T_3(t, x_1, L, 1)$ 
have no solutions for any $0 \leq t\leq 1$.  
\end{proof}

\begin{theorem}\label{MAINTHEM}
There exists at least one $\SU(\ell+m+n)$-invariant Einstein metric on the indecomposable C-space $\SU(\ell+m+n)/\SU(\ell)\times\SU(m)\times\SU(n)$. 
\end{theorem}
\begin{proof}
 By Proposition \ref{exist_varepsilon_L}, we conclude that there exist numbers $\varepsilon = \varepsilon(\ell, m, n) > 0$ and $L= L(\ell, m, n) > 0$ for which we can apply mapping degree theory to $f_t : {\overline D_{\varepsilon, L}} \to {\mathbb R}^2$.  By Lemma  \ref{degree_flag_einstein} we have that $\dg(f_0, D_{\varepsilon, L}, (0, 0)) = -2$  and hence by homotopy invariance we also obtain  $\dg(f_1, D_{\varepsilon, L}, (0, 0)) = -2$. Therefore, there exists a solution $(x_1, x_2) \in  D_{\varepsilon, L}$ of the equation $f_1(x_1, x_2, 1) = (0, 0)$. 
\end{proof}
\begin{remark}
 Note that  $(0, 0)$ might be a singular value,  hence the condition $d(f_1, D_{\varepsilon, L}, (0, 0)) = -2$ does not provide the existence  of a second invariant Einstein metric on $M_{\ell, m, n}$.  If $(0, 0)$ is a regular value for $f_1$, then there exist at least two  solutions. In fact, for $\ell = m = n =1$, we see that $(0, 0)$ is  a singular value for $f_1$ and the  equation $f_1(x_1, x_2, 1) = (0, 0)$ has only one solution.  This means that on the C-space $M_{1, 1, 1}=\SU(3)$, among invariant metrics $g$ defined by {\rm (\ref{ggg})},  only the bi-invariant metric is an Einstein metric, a fact  which we will prove also below.
\end{remark}


\subsection{More invariant Einstein metrics on  C-spaces}

\subsubsection{Case A}\label{caseaaa}
Assume that $\ell, m, n$ are different. For small values of $\ell,  m, n$ it is possible to obtain the  numerical form of all non-isometric homogeneous Einstein metrics on the corresponding C-space \[
M_{\ell, m, n}=G/H=\SU(\ell+m+n)/\SU(\ell)\times\SU(m)\times\SU(n)\,.
\]
 In particular, we see that   there exist  two  or four non-isometric $G$-invariant Einstein metrics which we shall denote by $g_\al, g_\be$ and   $g_\al, g_\be, g_\gamma, g_\delta$, respectively. Their numerical values are listed in Table \ref{Table1}, where  we write $g=(x_1, x_2, x_3, v_4, v_{5})$ for an invariant Einstein metric, with $x_3=1$.
\begin{table}[ht]
\centering
{\small \caption{The case of two non-isometric  invariant Einstein metrics on $M_{\ell, m, n}$  for small  different $\ell,  m, n$}\label{Table1}
\vspace{0.1cm}
\renewcommand\arraystretch{1.4}
\begin{tabular}{ c | c | l | l }
\hline
$M_{\ell, m, n}$	& $\dim_{\R}$ &	 $g_\al$								& $g_\be$								 	\\ 
\hline 
$M_{1, 2, 3}$ & $24$ & $(0.472295, 1.19781, 1, 1.77808, 0.60798)$    &  $(1.49887, 0.714536, 1, 1.14012,  1.55945)$\\
$M_{1, 2, 4}$ & $30$ & $(1.5978,   0.76303, 1, 1.26653, 1.63504)$ & $(0.379311,   1.13315, 1, 1.83194, 0.490535)$ \\
$M_{1, 2, 5}$ & $36$ & $(1.66213,  0.796466, 1, 1.36024, 1.6853)$ & $(0.31734, 1.09462, 1, 1.86425,  0.411959)$\\
$M_{1, 3, 4}$ & $40$ & $(0.48286, 1.30095, 1, 1.88783,  0.685127)$ & $(1.48800, 0.636510, 1,  1.21459,   1.47125)$ \\
$M_{1, 3, 5}$ & $ 48$ & $(1.5613, 0.681659, 1, 1.3168, 1.5272)$   &    $(0.417584,  1.24436, 1, 1.91656,  0.593683)$\\
$M_{2, 3, 4}$ &  $54$ & $(0.676785, 1.49686, 1,  1.9581,  1.03866)$ & $(1.70003,  0.833603, 1, 1.26452, 
   2.01911)$ \\
$M_{2, 3, 5}$ & $64$ &  $(1.75345,  0.855002, 1, 1.33712, 2.02138)$ & $(0.586034, 1.41566, 1, 1.98963, 
0.899876)$ 
 \end{tabular}}
\end{table}
 In Table \ref{Table21},  for simplicity we pose only 
the numerical values of $x_1, x_2$.
\begin{table}[ht]
\centering
{\small \caption{The case of four non-isometric  invariant Einstein metrics}\label{Table21}
\vspace{0.1cm}
\renewcommand\arraystretch{1.4}
\begin{tabular}{ c | c | l | l| l | l }
\hline
$M_{\ell,m,n}$	& $\dim_{\R}$ & $g_\al$& $g_\be$ &	 $g_\gamma$ &  $g_\delta$	\\ 
\hline 
$M_{3, 4, 5}$ & $96$ & $(0.514582, 0.594076)$    &  $(0.727423, 0.847601)$ & $(0.761962,  1.65282)$ &  $(1.79298,  0.879305)$ \\
$M_{3, 4, 6}$ & $110$ & $(0.480679,   0.628472)$    &  $(0.646952,  0.857152)$ & $(0.682517,   1.57948)$ &  $(1.82385,  0.891611)$ \\
$M_{4, 5, 6}$ & $150$ & $(0.499825,   0.558034)$    &  $(0.793644,   0.891443)$ & $(0.809993,  1.73784)$ &  $(1.84458, 0.904275)$ \\
$M_{5, 6, 7}$ & $216$ & $(0.495154, 0.541631)$    &  $(0.832054,   0.913651)$ & $(0.841356,  1.79029)$ &  $(1.87675,  0.92032)$ \\
 \end{tabular}}
\end{table}
 \begin{remark}
  In order to examine the isometric problem for  $g_{\al}$ and  $g_{\be}$ in Table \ref{Table1}, we apply the following method.  First we  use the numerical values of $g_{\al}$ and $g_{\be}$ to compute the corresponding Einstein constants.  Then, we normalize the    Einstein metrics such that the associated Einstein constants become equal to one. By comparing the obtained  normalized Einstein metrics one can easily deduce that they are not isometric.
  \end{remark}
  \begin{remark}
Consider the spaces $M_{1, 2, 3}$, $M_{1, 3, 2}$, $M_{2, 1, 3}$, $M_{2, 3, 1}$, $M_{3, 1, 2}$ and $M_{3, 2, 1}$. All of them define the homogeneous space $\SU(6)/\SU(2)\times\SU(3)\simeq \SU(6)/\SU(3)\times\SU(2)$. By normalizing the corresponding Einstein metrics such that the Einstein constant is equal to one, it is not hard to see that the Einstein metrics on these manifolds are isometric each other. Indeed, the action of the Weyl group of $\SU(6)$ interchanges  the  metric parameters of the diagonal part, while its action on   the abelian part $\fr{f}_{0}$ gives us  invariant metrics $( \ , \ )_{0}$ whose matrices   have the same eigenvalues (for a more explicit  description of the Weyl action see for instance the proof of Theorem \ref{them3}).  For example, consider $M_{1, 2, 3}$ and $M_{1, 3, 2}$, endowed with the invariant Einstein metrics  
\begin{eqnarray*}
g_{\al}^{1}&=&(x_1, x_2, 1, v_4, v_5)=(0.472295, 1.19781, 1, 1.77808, 0.60798)\,,\\
g_{\al}^{2}&=&(x_1, x_2, 1, v_4, v_5)=(1.19781,  0.472295, 1, 0.854456, 1.26518)\,,
\end{eqnarray*}
 respectively.  Note that for $M_{1, 2, 3}$, the real number $c$ is given by  $\displaystyle c =\frac{-1 + {x_2}^2}{1 + 2 {x_2}^2}\approx 0.112353$, 
  while for $M_{1, 3, 2}$ we obtain 
  $\displaystyle  c=\frac{-1 +{x_2}^2}{1 + 3{x_2}^2}\approx -0.465459$. 
 Then, for both the pairs $(M_{1, 2, 3}, g_{\al}^{1})$ and  $(M_{1, 3, 2}, g_{\al}^{2})$ we compute  the same Einstein constant, $\lambda=0.399622$. Thus, the corresponding normalized Einstein metrics  satisfying $\lambda=1$,  are given by
 \begin{equation}\label{gagb}
\left. \begin{tabular}{l}
 $(g_{\al}^{1})'=(x_1', x_2', x_3', v_4', v_5')=(0.188739, 0.478671, 0.399622, 0.710559, 0.242962)$ \\
 $(g_{\al}^{2})'=(x_1', x_2', x_3', v_4', v_5')=(0.478671, 0.188739, 0.399622, 0.341459, 0.505592)$
 \end{tabular}\right\}.
 \end{equation}
 Note   that for $M_{1, 2, 3}$  we have $d_1=\dim\fr{f}_1=2$, $d_{2}=\dim\fr{f}_2=3$ and $d_3=\dim\fr{f}_3=6$, while for $M_{1, 3, 2}$ we have $d_1=3$, $d_2=2$ and $d_3=6$, hence we have interchanging of the first two paremeters of the diagonal part of these  metrics, as one can see also by (\ref{gagb}).  For the abelian part, recall by Proposition \ref{parametrize}, that the matrix of the inner product $( \ , \ )_{0}$ is given  by
\[
\begin{pmatrix}
1 & 0\\
\gamma & 1
\end{pmatrix}^{T}
\begin{pmatrix}
v_4' & 0\\
0 &  v_5'
\end{pmatrix}
\begin{pmatrix}
1 & 0\\
\gamma & 1
\end{pmatrix}=\begin{pmatrix}
v_4'+c^2v_5'   & c v_{5}' \\
c v_5'  & v_5'
\end{pmatrix}\,.
\]
By using the above values of $(g_{\al}^{1})'$ and $(g_{\al}^{2})'$, we see that the corresponding matrices have the same eigenvalues, namely, $\lambda_1=0.715204$, $\lambda_2=0.241384$.
Since the corresponding eigenvector matrices belong to $\Oo(2)$, this allows us to  conclude that the invariant Einstein metrics on $M_{1, 2, 3}$ and $M_{1, 3, 2}$, are isometric. Similarly for the other cases,   while in an  analogous way can be  treated the rest manifolds appearing in Table \ref{Table1} (or in Table \ref{Table21}).  \end{remark}

  
  \subsubsection{Case B}\label{casebbb}
  Let us now assume that two of the parameters $\ell, m, n$ coincide, i.e. $\ell=m$. 
   Then, the rational polynomials $\mathsf{t}_{i}$ $(i=1, 2, 3)$ introduced in Section \ref{existence section}, take the form
   {\small
 \[  \mathsf{t}_1(x_1, x_2, x_3) = \frac{\displaystyle\frac{1}{2{x_1}}+\frac{n}{4 (2 m+n)}\frac{{x_1}^2-{x_2}^2-{x_3}^2}{ {x_1} {x_2} {x_3}}}{\displaystyle 1+\frac{{x_2}^2+{x_3}^2}{m \left(n {x_1}^2+m {x_2}^2+m {x_3}^2\right)}}\,,  \quad \mathsf{t}_2(x_1, x_2, x_3) = \frac{\displaystyle \frac1{2x_{2}} + \frac{m}{4 (2 m+n)}
\frac{{x_{2}}^{2}-{x_{1}}^{2}-{x_{3}}^{2}}{x_1x_2x_3}}{\displaystyle 
 1+ \frac{n {x_1}^2+m{x_3}^2}{m n \left(n {x_1}^2+m {x_2}^2+m {x_3}^2\right)}}\,, \]}
 and 
 {\small 
 \[
 \  \mathsf{t}_3(x_1, x_2, x_3) = \frac{\displaystyle \frac1{2x_{3}} + 
+ \frac{m}{4 (2 m+n)}
\frac{{x_{3}}^{2}-{x_{1}}^{2}-{x_{2}}^{2}}{x_1x_2x_3}}{\displaystyle 1+
 \frac{n{x_1}^2+m{x_2}^2}{m n \left(n {x_1}^2+m {x_2}^2+m {x_3}^2\right)}}\,,
\]}
  respectively. By  normalizing the equations $\mathsf{t}_1(x_1, x_2, 1) =\mathsf{t}_2(x_1, x_2, 1) =\mathsf{t}_3(x_1, x_2, 1)$ with $x_3=1$, 
 we deduce  that  the  homogeneous Einstein equation is equivalent to the following system of equations 
{\small    \begin{eqnarray*} 
h_1(x_1, x_2)&=& -n^2 {x_1}^4 \left(m^2+m n+1\right)-n
   {x_1}^2 {x_2}^2 (m-n) \left(m^2+m
   n+1\right)\\ 
   &&+2 \left(m^2+1\right) n {x_1}
   {x_2}^2 (2 m+n)+2 \left(m^2+1\right) n
   {x_1} (2 m+n)+m n {x_2}^4 \left(m^2+m
   n+1\right)\\ 
   &&-2 m^2 n {x_2}^3 (2 m+n)-m^2 n
   (m-n)-n {x_1}^2 \left(m^3+2 m^2 n-m n^2+2
   m-n\right)\\
    &&
  +2 m n^2 {x_1}^3 (2 m+n) -2 n {x_1}^2 {x_2} (2 m+n) (m n+1)+m n
   {x_2}^2 (2 m n+1)\\
   &&-2 m {x_2} (2 m+n) (mn+1)=0\,,\\ 
  ({x_2}-1)  h_2(x_1, x_2)&=&({x_2}-1) \big(-2 m^2 n {x_1}
   {x_2}^2 (2 m+n) -2 m^2 n {x_1} (2 m+n)+m^2 {x_2}^3 (2 m n+1)  \\
    & &+m^2 {x_2}^2
   (2 m n+1)+m^2 {x_2} (2 m n+1)+m^2 (2 m n+1)+m {x_1}^2 {x_2} \left(2 m n^2-m+2 n\right)
   \\ 
   &&+m {x_1}^2 \left(2 m n^2-m+2 n\right)-2 n {x_1}^3 (2 m+n) (m n+1) +2 m {x_1} {x_2} (2 m+n)\big)=0\,.
\end{eqnarray*} }
Let us treat the case   ${x_2}=1$. Then, we obtain the following polynomial $H_1(x_1)$ from the polynomial  $h_1(x_1, x_2)$: 
 $H_1(x_1)= -n^2 (m^2+m n+1){x_1}^4+2 m n^2(2 m+n) {x_1}^3 -m n (2 m^2+6 m n+7) {x_1}^2   +4(m^2+1) n(2 m+n) {x_1} -4 m^2 (2 m n+1)$. 
 We compute $H_1(0) = -4 m^2 (2 m n+1)  < 0$, $H_1(2) =-(4 m^2+ 12 m n+8 n^2)  < 0$ and 
$H_1(1) =(n-m)(2 m^2 n+m n^2+4 m+3 n) > 0$ for $n > m$. Thus, for $n > m$,   there exist at least two solutions $x_1=\al$ and $x_1=\be$ of $H_1(x_1)=0$ with $0 < \al < 1$ and $1< \be <2$.

Let  us now consider the polynomial ring $R= {\mathbb Q}[m, n][ x_{1}, x_{2}] $ and the ideal $I$  generated by the polynomials  $\{h_1(x_1,x_2), \, h_2(x_1,x_2) \}$.   We take a lexicographic ordering $>$   with $ x_{1} > x_{2}$ for a monomial ordering on $R$. Then, with the aid of computer we deduce  that a  Gr\"obner basis for the ideal $I$ contains   polynomials $G_1(x_2)$ and  $G_2(x_2)$ of  $x_{2}$, which are given by 
{\small \begin{eqnarray*}
   G_1(x_2) &=& m n^2 {x_2}^4 (m^2+m n+1)^2+(m+n) (m n+1)^2(m^2+13 m n+4 n^2)\\ 
   &&   +2 m n {x_2}^2 (m n+1)(m^3+8 m^2 n+7 m n^2+m+2 n^3-n)\,,\\
   G_2(x_2)&=& n^2 (m+n) (3 m+n) (m^2+n m+1)^3 (m^2+13 n m+4 n^2)
   {x_2}^8
   -2 n (m+n)\times  \\ 
   &&(m^2+n m+1)^2 (2 n m^5+38 n^2
   m^4+2 m^4+42 n^3 m^3+29 n m^3+22 n^4 m^2  -15 n^2 m^2 +4 n^5 m\\
   &&
    -14 n^3 m-2
   n^4) {x_2}^7+(-m^2-n m-1) (8 n^2 m^8+40 n^3
   m^7-168 n^4 m^6  -72 n^2 m^6+4 m^6 \\&&
   -484 n^5 m^5-620 n^3 m^5+48 n m^5-400
   n^6 m^4-1352 n^4 m^4+77 n^2 m^4 
   -132 n^7 m^3 \\ 
   &&
   -1092 n^5 m^3-289 n^3
   m^3-16 n^8 m^2-352 n^6 m^2-456 n^4 m^2-40 n^7 m-188 n^5 m-24
   n^6) {x_2}^6  \\
    &&
   -2 (6 n^2 m^{10}+132 n^3 m^9-2 n m^9+558
   n^4 m^8-7 n^2 m^8+8 m^8+956 n^5 m^7+442 n^3 m^7 +140 n m^7 \\ &&
   +838 n^6
   m^6+1054 n^4 m^6+798 n^2 m^6+8 m^6+412 n^7 m^5+934 n^5 m^5+1320 n^3
   m^5+230 n m^5 
\\ 
&&
   +110 n^8 m^4+449 n^6 m^4+706 n^4 m^4+579 n^2 m^4+12 n^9
   m^3+136 n^7 m^3+66 n^5 m^3+366 n^3 m^3 \\
    &&
   +18 n^8 m^2-14 n^6 m^2-3 n^4
   m^2-40 n^5 m-6 n^6) {x_2}^5-2 (11 n^2 m^{10}+100 n^3
   m^9+235 n^4 m^8 \\
    &&
   +9 n^2 m^8+12 m^8-8 n^5 m^7+215 n^3 m^7+236 n m^7-355
   n^6 m^6-574 n^4 m^6+1176 n^2 m^6 \\
    &&
   +12 m^6-304 n^7 m^5-1630 n^5 m^5+798
   n^3 m^5+416 n m^5-99 n^8 m^4-1259 n^6 m^4-931 n^4 m^4 \\
    &&
   +634 n^2 m^4-12
   n^9 m^3-389 n^7 m^3-1230 n^5 m^3+68 n^3 m^3-44 n^8 m^2-441 n^6 m^2-311
   n^4 m^2  \\
    &&
   -52 n^7 m-151 n^5 m-20 n^6) {x_2}^4-2 (6 n^2
   m^{10}+132 n^3 m^9-2 n m^9+558 n^4 m^8-7 n^2 m^8+8 m^8 \\ &&
   +956 n^5 m^7+442
   n^3 m^7+140 n m^7+838 n^6 m^6+1054 n^4 m^6+798 n^2 m^6+8 m^6+412 n^7
   m^5 \\ 
   &&
   +934 n^5 m^5+1320 n^3 m^5+230 n m^5+110 n^8 m^4+449 n^6 m^4+706 n^4
   m^4+579 n^2 m^4  +12 n^9 m^3 
\\
&&+136 n^7 m^3+66 n^5 m^3+366 n^3 m^3+18 n^8
   m^2-14 n^6 m^2-3 n^4 m^2-40 n^5 m-6 n^6)
   {x_2}^3 \\ 
   &&
   +(-m^2-n m-1) (8 n^2 m^8+40 n^3 m^7-168 n^4
   m^6-72 n^2 m^6+4 m^6-484 n^5 m^5-620 n^3 m^5\\
   &&
   +48 n m^5-400 n^6 m^4-1352
   n^4 m^4+77 n^2 m^4-132 n^7 m^3-1092 n^5 m^3-289 n^3 m^3\\
   &&
   -16 n^8 m^2-352
   n^6 m^2-456 n^4 m^2-40 n^7 m-188 n^5 m-24 n^6) {x_2}^2
 -2 n (m+n)\times \\
 && (m^2+n m+1)^2
    (2 n m^5+38 n^2 m^4+2 m^4+42 n^3
   m^3+29 n m^3+22 n^4 m^2
   -15 n^2 m^2\\&&+4 n^5 m  -14 n^3 m-2 n^4)
   {x_2}+n^2 (m+n) (3 m+n) (m^2+n m+1)^3 (m^2+13 n
   m+4 n^2)\,,
  \end{eqnarray*}}
  respectively.  Since $G_1(x_2) > 0$,    there are no real solutions  of the equation $G_1(x_2)=0$. 
  
  For  $G_2(x_2)$  we see that $G_2(0) = n^2 (m+n) (3 m+n) (m^2+m n+1)^3(m^2+13 m n+4 n^2) > 0$ and for $m > n$, based on a computation by   Mathematica, we  see that $G_2(1)<0$ and $G_2(5)>0$.   Thus, for $m>n$, we have at least two solutions $x_2=\gamma$ and $x_2=\delta$ of  $G_2(x_2) =0$, with $0 <  \gamma <1$ and $1 < \delta < 5$, respectively.
 Note that the coefficients $a_j(p, q)$ of the polynomial $G_{2}(x_2)$ satisfy $a_j(p, q) = a_{8-j}(p, q)$ for  any $j = 0,1, \ldots, 8$,  and hence if     $x_2=\gamma$ is a solution of $G_2(x_2) =0$, then so is $x_2=1/\gamma$.
 When $G_2(x_2)=0$, consider the  ideal $J$  generated by the polynomials  $\{h_1(x_1,x_2),  h_2(x_1,x_2),  G_2(x_2) \}$.  By computing the  Gr\"obner basis for the ideal $J$ with lexicographic ordering $ >$   with $  x_{1} > x_{2}$    in the above ring $R$,  we  see that the ideal $J$  contains    a polynomial of the form  $ p(m, n)x_1 = \sum_{k=0}^{7}q_k(m, n){x_2^{}}^k$, 
   where $p(m, n)$  is the following polynomial of $m, n$:  
    {\small  
    \begin{eqnarray*} 
            p(m, n)&=&  -8 n^2 (2 m+n)^3 \left(m^2+m n+1\right)^2 \left(m^2+2 m n+2\right)
   \big(10 m^{10} n^3+2 m^{10} n+159 m^9 n^4\\&&+54 m^9 n^2+m^9+607 m^8
   n^5+517 m^8 n^3+36 m^8 n+983 m^7 n^6+2201 m^7 n^4+489 m^7 n^2
   +5 m^7
   \\&&+811 m^6 n^7+3617 m^6 n^5+2670 m^6 n^3+141 m^6 n+362 m^5 n^8+2891
   m^5 n^6+4444 m^5 n^4\\&&
   +1093 m^5 n^2+84 m^4 n^9+1226 m^4 n^7+3420 m^4
   n^5+1821 m^4 n^3+8 m^3 n^{10}+268 m^3 n^8\\&&
   +1376 m^3 n^6+1342 m^3 n^4+24
   m^2 n^9+284 m^2 n^7+512 m^2 n^5+24 m n^8+100 m n^6+8 n^7\big)\,.
      \end{eqnarray*}}
 Moreover, $q_k(m, n)$  ($k = 0, 1, \ldots, 7$) are polynomials of $m, n$ with integer coefficients. 
Since $p(m, n)\neq 0$, we deduce that  if $x_2$ is a real solution of $G_{2}(x_2)=0$, then the solution $x_1$ of the equations $h_1(x_1, x_2)=0$ and $h_2(x_1, x_2)=0$ defined above, is also real. 
Take a lexicographic ordering $>$   with $  x_{2} > x_{1}$ for a monomial ordering on the ring $R$. Then,  a  Gr\"obner basis for the ideal $J$ contains a  polynomial $H_2(x_1)$ of  $x_{1}$,  which is given by 
{\small  \begin{eqnarray*} 
   H_2(x_1)&=& n^4 (m+n) (3 m+n) (m^2+n m+1)^3 (m^2+13 n m+4 n^2)
   {x_1}^8 \\ && 
   -2 m n^3 (m+n) (2 m+n) (m^2+n m+1)^2 (2 n
   m^3+44 n^2 m^2+2 m^2+14 n^3 m+35 n m+11 n^2) {x_1}^7\\ &&
   +m^2 n^2
   (m^2+n m+1) (6 n m^7+54 n^2 m^6+352 n^3 m^5 +11 n m^5 \\&&+596
   n^4 m^4+462 n^2 m^4
   +4 m^4+362 n^5 m^3+902 n^3 m^3+149 n m^3 \\&&+70 n^6
   m^2+594 n^4 m^2+278 n^2 m^2+119 n^5 m+208 n^3 m+45 n^4)
   {x_1}^6 \\ &&
   -2 m^2 n^2 (2 m+n) (10 n m^8+88 n^2 m^7+4 m^7+188 n^3
   m^6+63 n m^6+172 n^4 m^5 \\ && 
   +263 n^2 m^5+2 m^5+66 n^5 m^4+313 n^3 m^4+122 n
   m^4+4 n^6 m^3+143 n^4 m^3 \\ &&
   +179 n^2 m^3+10 m^3+6 n^5 m^2+93 n^3 m^2+39 n
   m^2+19 n^2 m-2 n^3) {x_1}^5\\&&
   +m^2 n (120 n^2 m^9+896 n^3
   m^8+68 n m^8+1516 n^4 m^7+1152 n^2 m^7+8 m^7\\&&
   +1092 n^5 m^6+2488 n^3
   m^6+292 n m^6+364 n^6 m^5+1996 n^4 m^5+1294 n^2 m^5 \\&&
   -4 m^5+44 n^7 m^4+696
   n^5 m^4+1281 n^3 m^4+200 n m^4+80 n^6 m^3+506 n^4 m^3\\&&
   +251 n^2 m^3+71 n^5
   m^2+130 n^3 m^2+4 n^6 m+31 n^4 m+4 n^5) {x_1}^4 \\&&
   -2 m^3 n (2 m+n)
   (2 m n+1) (32 n m^6+118 n^2 m^5+92 n^3 m^4+152 n m^4+22 n^4 m^3\\ &&
   +167n^2 m^3+20 m^3+51 n^3 m^2+62 n m^2+2 n^4 m+33 n^2 m+5 n^3)
   {x_1}^3\\ &&
   +m^4 (2 m+n)^2 (2 m n+1) (52 n^2 m^4+56 n^3 m^3+6 n m^3+4
   n^4 m^2 \\&&
   +90 n^2 m^2-4 m^2+16 n^3 m+25 n m+7 n^2) {x_1}^2 \\ &&
   -4 m^5 n
   (2 m+n)^2 (2 m n+1)^2 (6 m^2+2 n m+1) {x_1}+4 m^6 (2 m+n)^2 (2 m n+1)^3. 
  \end{eqnarray*}}
Note that the coefficients of even degree of $H_2(x_1)$ are positive and the coefficients of odd degree of $H_2(x_1)$ are negative. Thus,  if $x_1$ is a real solution of $H_{2}(x_1)=0$, then it must be positive. Note that the solutions $x_1= \nu$ and $x_1= \mu$  obtained above, from the solutions $\gamma$ and $1/\gamma$ of  $G_2(x_2) =0$, respectively, are also the solutions of  $H_2(x_1) =0$.  Hence, we have proved  the following
 \begin{theorem}\label{theorem4.9} For $m < n$
there exist at least two $\SU(m+m+n)$-invariant Einstein metrics on the C-space $M_{m, m, n}=\SU(m+m+n)/\SU(m)\times\SU(m)\times\SU(n)$,   given by   $(x_1, x_2, x_3) = (\al, 1, 1)$ and  $(x_1, x_2, x_3) = (\be, 1, 1)$,   with $0 < \al < 1$ and $1< \be <2$, respectively. For $m > n$
there also exist at least two $\SU(m+m+n)$-invariant Einstein metrics on  $M_{m, m, n}$, which are of the form   $
 (x_1, x_2, x_3) = (\nu, \gamma,  1)$ and $(x_1, x_2, x_3) = (\mu, 1/\gamma, 1)$,   with $0 <  \gamma <1$.
\end{theorem}

\begin{remark}
 For  the invariant   Einstein metrics given above it is hard to solve the isometric problem. However, motivated by particular cases (see below), we conjecture that  the invariant Einstein metrics corresponding to   $(x_1, x_2, x_3) = (\nu, \gamma,  1), (x_1, x_2, x_3) = (\mu, 1/\gamma, 1)$  are isometric, where $x_1=\nu > 0$, $x_1=\mu >0$ are the real solutions of $H_2(x_1)=0$, which are also obtained from the equation $p(m, n)x_1 = \sum_{k=0}^{7}q_k(m, n){x_2^{}}^k$ by substituting $x_2 = \gamma$, 
 $x_2 = 1/\gamma$, respectively.
 \end{remark}
For small $m, n$ it is  again   possible to present the  numerical forms of all  invariant Einstein metrics on the C-space $M_{m, m, n}=G/H$, and improve Theorem \ref{theorem4.9}.   In particular, the conditions $n > m$ or $m > n$ given in  Theorem \ref{theorem4.9} are not optimum. 
For example, we have members of $M_{m, m, n}$ admitting   two, one, or three  $G$-invariant Einstein metrics, up to isometry.   These cosets and the numerical values of the corresponding Einstein metrics are  given in   Tables \ref{Table31},    \ref{Table32} and  \ref{Table41}, respectively. 

\begin{table}[ht]
\centering
{\small \caption{The case of two non-isometric  invariant Einstein metrics}\label{Table31}
\vspace{0.1cm}
\renewcommand\arraystretch{1.4}
\begin{tabular}{ c | c | l | l }
\hline
$M_{m, m, n}$	& $\dim_{\R}$ &	 $g_\al$								& $g_\be$								 	\\ 
\hline 
$M_{1, 1, 2}$ & $12$  & $(1.61237, 1, 1, 1.11629, 1.61237)$ & $(0.387628,  1, 1,  1.48371,  0.387628)$ \\
$M_{1, 1, 3}$ & $16$  & $(1.7303,  1, 1, 1.26935,  1.7303)$ & $(0.269703, 1, 1, 1.64493,  0.269703)$\\
$M_{1, 1, 4}$ & $20$  & $(1.79057, 1, 1,  1.37987, 1.79057)$ & $(0.209431,  1,  1, 1.73124,  0.209431)$\\
 \end{tabular}}
\end{table}
\begin{table}[ht]
\centering
{\small \caption{The case of only one  invariant Einstein metric (up to isometry)}\label{Table32}
\vspace{0.1cm}
\renewcommand\arraystretch{1.4}
\begin{tabular}{ c | c | l | l }
\hline
$M_{m, m, n}$	& $\dim_{\R}$ &	 $g_\al$								& $g_\be$								 	\\ 
\hline 
$M_{2, 2, 1}$ & $18$ & $(1.586, \gamma\approx 2.089, 1, 1.473,  2.307)$ & $(0.7589, 1/\gamma \approx 0.4785, 1,  0.7052,  1.1037)$ \\ &&
 $(1, 1.31775, 0.630577, 0.929305, 1.45443)$ & $( 1, 0.630577, 1.31775, 0.929305, 1.45443)$ 
\\
$M_{3, 3, 2}$ & $44$ &  $(1.244, \gamma\approx 2.001, 1, 1.847, 1.975)$&  $(0.6219, 1/\gamma \approx 0.4997,  1, 0.923,   0.9871)$   \\ &&
 $(1, 1.60802, 0.803557, 1.4839, 1.58721)$& $(1, 0.803557, 1.60802, 1.4839, 1.58721)$ 
 \end{tabular}}
\end{table}

\begin{table}[ht]
\centering
{\small \caption{The case of three non-isometric  invariant Einstein metrics }\label{Table41}
\vspace{0.1cm}
\renewcommand\arraystretch{1.4}
\begin{tabular}{ c | c | l | l  }
\hline
$M_{m, m, n}$	& $\dim_{\R}$ & &  	\\ 
\hline 
$M_{2, 2, 3}$ & $34$ & $g_\al \approx (0.70564, 1, 1, 1.6260, 1.0316)$ &   $g_\be \approx (1.7749, 1, 1, 1.3074, 2.1434)$ 
 \\ & & $g_\gamma \approx (0.5547, 0.7405, 1, 1.3726, 0.8039)$ & $g_\delta \approx (0.7491, 1.3504, 1, 1.8535, 1.0856)$ 
 \\ && $(g_\gamma)' \approx ( 1, 1.3349, 1.8027, 2.4743, 1.4492)$ & $(g_\delta)' \approx (1, 1.8027, 1.3349, 2.4743, 1.4492)$ \\
$M_{3, 3, 4}$ & $68$ & $g_\al \approx (0.8206, 1, 1, 1.6631, 1.2882)$    &  $g_\be \approx (1.8673, 1, 1, 1.34463, 2.3504)$ 
 \\ & & $g_\gamma \approx (0.5086 , 0.6038, 1, 1.2232, 0.7877)$ & $g_\delta \approx (0.8423, 1.6561, 1, 2.0256, 1.3046)$ 
 \\ && $(g_\gamma)'  \approx (1, 1.1872, 1.9661, 2.4047, 1.5488)$ & $(g_\delta)' \approx (1, 1.9661, 1.1872, 2.4047, 1.5488)$ \\
$M_{4, 4, 3}$ & $82$ &$g_\al \approx (1.1969, 1, 1, 1.4815, 1.898)$ &   $g_\be \approx (1.8027, 1, 1, 1.2629, 2.544)$ 
 \\ & & $g_\gamma \approx (0.57292, 0.49478, 1, 0.97531, 0.9344)$ & $g_\delta \approx (1.1579, 2.0211, 1, 1.9712, 1.8884)$ 
 \\ && $(g_\gamma)'  \approx ( 1, 0.8636, 1.74545, 1.7024, 1.6309)$ & $(g_\delta)' \approx (1, 1.7455, 0.8636, 1.7024, 1.6309)$ 
 \end{tabular}}
\end{table}

\begin{remark}
Note that the metrics $g_{\al}, g_{\be}$ in  Table \ref{Table31} do not necessarily exhaust all invariant Einstein metrics on  $M_{1, 1, 2}, M_{1, 1, 3}$ and  $M_{1, 1, 4}$, since for these C-spaces we have $\ell m=1$ (see Remark \ref{remarkinv}).
In  Table \ref{Table32}, we have two solutions of the Einstein equation, which are isometric. To see this, in  the second line of each   $M_{m, m, n}$  we  include the  corresponding values after applying the  normalization  $x_1=1$.
 In  Table \ref{Table41},  the metrics $g_{\gamma}$ and $g_{\delta}$ are isometric. For this, again in the  third line of each C-space $M_{m, m, n}$ we  present the  corresponding values by  applying the normalization $x_1=1$. 
\end{remark}


   \subsubsection{Case C} \label{casecccc}
   Let as assume now that $\ell=m=n$. 
 In this case the rational polynomials $\mathsf{t}_{i}$ $(i=1, 2, 3)$ of    Section \ref{existence section}, are given by 
  \begin{eqnarray} 
\mathsf{t}_1(x_1, x_2, x_3) &=& \frac{\displaystyle\frac1{2x_{1}} + \frac{1}{12} \left(
\frac{{x_{1}}^{2}-{x_{2}}^{2}-{x_{3}}^{2}}{x_1x_2x_3}\right) }{\displaystyle 
  1+ \frac{ {x_2}^2+ {x_3}^2}{ n^2( {x_1}^2+{x_2}^2+  {x_3}^2)}}\,, \label{ttt1} \\
\mathsf{t}_2(x_1, x_2, x_3) &=& \frac{\displaystyle \frac1{2x_{2}} + \frac{1}{12}\left(
\frac{{x_{2}}^{2}-{x_{1}}^{2}-{x_{3}}^{2}}{x_1x_2x_3}\right) }{\displaystyle 
 1+ \frac{{x_1}^2+ {x_3}^2}{n^2 ( {x_1}^2+{x_2}^2+ {x_3}^2)}}\,, \label{ttt2}\\
\mathsf{t}_3(x_1, x_2, x_3) &=& \frac{\displaystyle \frac1{2x_{3}} + \frac{1}{12}\left(
\frac{{x_{3}}^{2}-{x_{1}}^{2}-{x_{2}}^{2}}{x_1x_2x_3}\right)}{\displaystyle 1+ \frac{  {x_1}^2+ {x_2}^2}{n^2 ( {x_1}^2+  {x_2}^2+ {x_3}^2)}}\,.\label{ttt3}
 \end{eqnarray} 
 By  normalizing the equations $\mathsf{t}_1(x_1, x_2, 1) = \mathsf{t}_2(x_1, x_2, 1) =\mathsf{t}_3(x_1, x_2, 1)$ with $x_3=1$, 
 we deduce that the Einstein equation is equivalent to the following system of equations 
  \begin{eqnarray*} 
g_1(x_1,x_2)&=&({x_1}-{x_2}) \Big((2
   n^2+1) {x_1}^3+{x_2}
   ((2 n^2+1) {x_1}^2+6 {x_1}+(2 n^2+1) ) -6 n^2 {x_1}^2\\ 
   & &    +{x_2}^2  ( (2  n^2+1 ) {x_1}-6 n^2 )+ (2
   n^2+1) {x_1}+ (2 n^2+1)
   {x_2}^3-6
   (n^2+1)\Big)=0\,,\\   
 g_2(x_1,x_2)&=&   ({x_2}-1)
  \Big(6 (n^2+1)   {x_1}^3-{x_2} ((2
   n^2+1) {x_1}^2+(2 n^2+1)+6 {x_1})  - (2 n^2+1){x_1}^2  \\
   & &     +{x_2}^2  (6 n^2
   {x_1}-(2 n^2+1) )+6 n^2
   {x_1}- (2 n^2+1 ) {x_2}^3-(2 n^2+1)\Big)=0.
    \end{eqnarray*} 
 Let us consider the polynomial ring $R= {\mathbb Q}[n][ x_{1}, x_{2}] $ and the ideal $I$  generated by the polynomials  $\{g_1(x_1,x_2), \, g_2(x_1,x_2) \}$.   We take a lexicographic ordering $>$   with $  x_{1} > x_{2}$ for a monomial ordering on $R$. Then, by the help of Mathematica, we see that a  Gr\"obner basis for the ideal $I$ contains    a  polynomial $({x_2}-1) h_1(x_2) h_2(x_2)$ of   $x_{2}$, where  
    \[
    h_1(x_2)=(2 n^2+1){x_2}^3-(2 n-1)(2 n+1){x_2}^2+4(n^2+2){x_2}-4(2 n^2+1)
    \]
    and
    \[
    h_2(x_2) = 4  (2 n^2+1 ) {x_2}^3-4 (n^2-2) {x_2}^2+(2 n-1) (2 n+1) {x_2}- (2 n^2+1)\,.
    \]  
  For $ {x_2}=1$,  we compute $0= g_1(x_1, 1) =(x_1-1)\big((2 n^2+1){x_1}^3-(2 n -1)(2 n+1){x_1}^2+4(n^2+2){x_1}-4(2 n^2+1)\big)=(x_1-1) h_1(x_1)$.  Since $h_1(1) = -6( n^2-1) < 0 $ for $n \geq 2$, $h_1(2) = 24$ and  $h_1(x_1)$ is a  monotone increasing function,  there exist only one solution $x_1= \al$  of $h_1(x_1)=0$,  with $1 < \al < 2$.

 For  the case of $h_1(x_2)=(2 n^2+1){x_2}^3-(2 n-1) (2 n+1) {x_2}^2+4 (n^2+2) {x_2}-4 (2 n^2+1)=0$,  by computing a  Gr\"obner basis for the ideal $I_1$ generated by the polynomials  $\{g_1(x_1,x_2), \, g_2(x_1,x_2), h_1(x_2) \}$, we see that the Gr\"obner basis contains the polynomial $(n-1)(x_1-1)$. Hence,  for $n \geq 2$, we obtain only one solution $(x_1, x_2) = (1, \al)$  with $h_1(\al )=0$ and $1 < \al < 2$. 

   For  the case of $h_2(x_2) = 4  (2 n^2+1 ) {x_2}^3-4 (n^2-2) {x_2}^2+(2 n-1) (2 n+1) {x_2}- (2 n^2+1)=0$, by computing a  Gr\"obner basis, we see that the Gr\"obner basis contains the polynomial $(n-1)(x_1-x_2)$. Since $h_2(x_2) = -{x_2}^3 h_1(1/x_2)$, we see that 
there exist only one solution $x_2= \be$  of $h_2(x_2)=0$, with $1/2 < \be < 1$. Note that $\be = 1/\al$. 
Hence, for $n \geq 2$, we obtain only one solution $(x_1, x_2) = (\be, \be)$  with $h_2(\be)=0$ and $1/2 < \be < 1$.  Hence we have proved that
    
     \begin{theorem} \label{them3}
 1) For $n \geq 2$,
there exist four $\SU(3n)$-invariant Einstein metrics on the C-space $M_{n, n, n}=\SU(3n)/\SU(n)\times\SU(n)\times\SU(n)$, given by 
  \begin{eqnarray*}  & & (x_1, x_2, x_3) = (1, 1, 1), \quad (x_1, x_2, x_3) = (1, \al, 1), \\ & &  (x_1, x_2, x_3) = (\al, 1, 1), \quad (x_1, x_2, x_3) = (1/\al, 1/\al, 1),
  \end{eqnarray*} 
where $\al$ is the solution of $h_1(\al)=(2 n^2+1){\al}^3-(2 n-1) (2 n+1) {\al}^2+4 (n^2+2) {\al}-4 (2 n^2+1)=0$. 
In particular, the invariant Einstein metrics defined by $(x_1, x_2, x_3) = (1, \al, 1)$,    $(x_1, x_2, x_3) = (\al, 1, 1)$ and $(x_1, x_2, x_3) = (1, 1, \al)$ are isometric each other, and  thus, up to isometry and scale, $M_{n, n, n}$ admits exactly two $\SU(3n)$-invariant Einstein metrics. 

 2) For  $n=1$ and the C-space $M_{1, 1, 1}=\SU(3)$,  among the $\SU(3)$-invariant metrics $g$ defined by {\rm (\ref{ggg})}, only the bi-invariant metric is  Einstein.
\end{theorem}
\begin{proof}
The existence part was completed above, so it remains  to examine the isometric problem  and prove the given assertion in 1).  Note that the action of the  Weyl group $\mc{W}$  of $\SU(3n)$ interchanges  the  metric parameters of the diagonal part, while on  the abelian part $\fr{f}_{0}$ induces an isometry of invariant metrics $( \ , \ )_{0}$. 
In particular, for the invariant Einstein metrics  $(x_1, x_2, x_3) = (1, \al, 1)$  and   $(x_1, x_2, x_3) = (\al, 1, 1)$,
we can take the element $\sigma\in\mc{W}$ given by the permutation 
$$\sigma = \left(
\begin{array}{ccccccccccc}
  1& 2 &\cdots & n & n+1 &\cdots & 2 n-1& 2n &  2n+1 & \cdots & 3 n \\
  1 & 2 &\cdots & n &  3 n  & \cdots & 2 n +2 & 2 n+1 &  2n &\cdots  & n+1
  \end{array}
  \right).
$$ 
This also satisfies the relation $\sigma(\fr{f}_{0}) = \fr{f}_{0}$. 

For the invariant Einstein metrics  $(x_1, x_2, x_3) = (1, \al, 1)$  and   $(x_1, x_2, x_3) = ( 1, 1, \al,)$,
we consider the element  $\sigma\in\mc{W}$ given by the permutation 
$$\sigma = \left(
 \begin{array}{ccccccccccc}
  1& 2 &\cdots & n & n+1 &\cdots & 2 n-1& 2n &  2n+1 & \cdots & 3 n \\
  2n & 2n-1 &\cdots & n+1 &   n  & \cdots &  2 & 1 &  2n+1 &\cdots  & 3 n
  \end{array}
   \right),
$$  
and 
for the invariant Einstein metrics  $(x_1, x_2, x_3) = (1, 1, \al)$  and   $(x_1, x_2, x_3) = (\al, 1, 1)$,
we can take  $\sigma\in\mc{W}$ as the permutation 
$$\sigma = \left(
\begin{array}{ccccccccccc}
  1& 2 &\cdots & n & n+1 &\cdots & 2 n-1& 2n &  2n+1 & \cdots & 3 n \\
  3n & 3n-1 &\cdots & 2n+1 &   2 n  & \cdots & n+ 2 & n+1 &  n &\cdots  & 1
  \end{array}
  \right).
$$

For the invariant Einstein metrics  $(x_1, x_2, x_3) = (1, \al, 1)$,    $(x_1, x_2, x_3) = (\al, 1, 1)$ and $(x_1, x_2, x_3) = (1, 1, \al)$,  the Einstein constants $\lambda$ are equal each other and can be  computed by applying each of (\ref{ttt1}), (\ref{ttt2}) or (\ref{ttt3}). We obtain  (note  that the given values of $\lambda$ coincide after  replacing $\al$, where $\al$ is  the solution of $h_1(\al)=0$)
\[
\lambda = \frac{(2 + \al^2) (4 + \al^2) n^2}{12 \al (2 + 2 n^2 + \al^2 n^2)}\,,  \quad (\text{or} \ \lambda= \frac{(6-\al) (2 + \al^2) n^2}{12 (1 + \al^2 + 2 n^2 + \al^2 n^2)})\,.
\]
Now, by the relations (\ref{eq_v4}), (\ref{eq_v5}) and (\ref{newww}),  the values for $v_4$, $v_5$  and $c$ are given as follows:

\noindent for  $(x_1, x_2, x_3) = (1, \al, 1)$,
 \begin{eqnarray*} 
  v_4 =  \frac{2 \al (2 + \al^2) (4 + \al^2) n^2}{3 (1 + \al^2) (2 + 2 n^2 + \al^2 n^2)}\,, \,\,\,
   \displaystyle v_5 =  \frac{(1 + \al^2) (4 + \al^2) n^2}{2 \al (2 + 2 n^2 + \al^2 n^2)}\,,   \,\,\, c=\frac{\al^2-1}{\sqrt{3} (\al^2+1)}\,, 
   \end{eqnarray*} 
 for  $(x_1, x_2, x_3) = (\al, 1, 1)$,  
\begin{eqnarray*} 
 \hspace{-50pt} v_4 = \frac{(2 + \al^2) (4 + \al^2) n^2}{3 \al (2 + 2 n^2 + \al^2 n^2)}\,, \,\,\,
   \displaystyle v_5 = \frac{\al (4 + \al^2) n^2}{2 + 2 n^2 + \al^2 n^2}\,,  \,\,\, c= 0\,, 
    \end{eqnarray*}
 for  $(x_1, x_2, x_3) = ( 1, 1, \al)$,  
    \begin{eqnarray*}  
  v_4 = \frac{2 \al (2 + \al^2) (4 + \al^2) n^2}{3 (1 + \al^2) (2 + 2 n^2 + \al^2 n^2)}\,, \,\,\,
 \displaystyle   v_5 = \frac{(1 + \al^2) (4 + \al^2) n^2}{2 \al (2 + 2 n^2 + \al^2 n^2)}\,, \,\,\, c=\frac{\al^2-1}{\sqrt{3} (\al^2+1)}\,.  
  \end{eqnarray*} 
  For the abelian part, recall by Proposition \ref{parametrize} that the matrix of the inner product $( \ , \ )_{0}$ defined by (\ref{prod}), is given  by
\[
\begin{pmatrix}
1 & 0\\
c & 1
\end{pmatrix}^{T}
\begin{pmatrix}
v_4 & 0\\
0 &  v_5
\end{pmatrix}
\begin{pmatrix}
1 & 0\\
c & 1
\end{pmatrix}=\begin{pmatrix}
v_4+c^2v_5   & c v_{5} \\
c v_5  & v_5
\end{pmatrix}.
\]
 Then, by replacing  the previous values of $v_4, v_5$ and $c$, we see that the corresponding matrices have always  the same eigenvalues:
 \[
 \lambda_1 =  \frac{(2 + \al^2) (4 + \al^2) n^2}{3 \al (2 + 2 n^2 + \al^2 n^2)}\,, \,\,\,
 \lambda_2 = \frac{\al (4 + \al^2) n^2}{2 + 2 n^2 + \al^2 n^2}\,.
 \]
Since the corresponding eigenvector matrices belong to  $\Oo(2)$,  this allows us to  conclude that the invariant Einstein metrics  are isometric each other. 
\end{proof}

\appendix
\section{Non-Kahler C-spaces $M=G/H$ of an exceptional Lie group $G$}\label{sectionapen}
In this appendix we present the classification of all  indecomposable non-K\"ahlerian C-spaces $M=G/H$ corresponding to an exceptional Lie group $G$.
So, let $G$ be one of the Lie groups  $G\in\{\G_2, \F_4, \E_6, \E_7, E_8\}$, with root system $R$ and a basis of simple roots $\Pi=\{\al_1, \ldots, \al_{\ell}\}$. Next, in Table \ref{Table6}  we shall denote by $G(\al_{{j_{1}}}, \ldots, \al_{{j_{u}}})\equiv G({j_{1}}, \ldots, {j_{u}})$ $(1\leq j_{1}<\ldots < j_{u}\leq \ell)$ the exceptional flag manifold $F=G/K$ for which $\Pi_{K}:=\{\al_{{j_{1}}}, \ldots, \al_{{j_{u}}}\}$ are the {\it white simple roots} in terms of painted Dynkin diagrams (these simple roots correspond to the semisimple part of $K$, while the rest simple roots which are painted black, generate the center of $K$, see \cite{Forger, Chry2, AC1}).  In such terms, $G(0)$ denotes the full flag manifold corresponding to $G$. Also, we write  ${\rm T}^{n}=\U(1)^{n}$ for the $n$-torus.  

A list with the explicit presentations of all non-isomorphic exceptional flag manifolds $F=G/K$ is given in \cite{Forger, Graev, AC1}.  In the following table we present all indecomposable non-K\"ahlerian C-spaces associated to an exceptional flag manifold. To indicate the type we use the notation: ``ss'' for semistrict, ``s''  for strict  and ``a''  for abelian.  Note that there is a misprint in \cite{AC1} about $\E_7(1, 2, 5, 6, 7)$, while the flag 
$\E_6(1,2,4,5,6)$ is missing from the published version (although it appears in the versions of arXiv). This is  flag manifold with $b_2=1$, so it does not play some role for  the given classification below.

\smallskip
{\bf A note on inequivalent C-spaces.}  For the exceptional Lie group $\F_4$,   the groups    $A_{\ell}^{l}$ (resp. $A_{\ell}^{s}$)  appearing in Table \ref{Table6} are the   groups of type $A_{\ell}$ defined by the long (resp. short) simple root(s) of $\F_4$.  Note that these groups are non-conjugate in $\F_4$, and  hence the associated flag manifolds are {\it inequivalent} as homogeneous spaces. As a consequence, the isotropy representations of such exceptional flags are always different. For example, $\F_{4}(1, 2)=\F_4/A_2^{l}\cdot{\rm T}^{2}$ has 9 isotropy summands, while $\F_{4}(3, 4)=\F_4/A_2^{s}\cdot{\rm T}^{2}$ has 6. Moreover, $\F_{4}(1, 2)$ and $\F_{4}(3, 4)$ are not biholomorphic as homogeneous complex manifolds (since the invariant complex structure is encoded by the PDD, see \cite{AP}).  The non-K\"ahlerian C-spaces fibered over  such flag spaces,  e.g. $\F_4/A_2^{l}$ and $\F_4/A_2^{s}$,  are also non-biholomorphic complex homogeneous spaces with different isotropy representations. Finally we use the notation ``type A'' and  ``type B''  to emphasize on  some non-isomorphic flag  manifolds of $\E_7$ (see also \cite{Forger, Graev, AC1}). The same type we assert to  the non-K\"ahlerian C-spaces associated to these cosets, which  are also inequivalent.
 \begin{table}[ht]
\centering
{\small \caption{The classification of exceptional indecomposable non-K\"ahlerian C-spaces}\label{Table6}}
\vspace{0.1cm}
\end{table}
\[
\begin{tabular}{c |  l  |  l | l | l | c | c }
$G$ & C-space $M=G/H$  & type & flag manifold $F=G/K$  & fiber & $b_{2}(F)$ & $b_{2}(M)$ \\
\thickline
$\G_2$  &  $\G_2$  & ss &  $\G_{2}(0)=\G_2/{\rm T}^{2}$  &  ${\rm T}^{2}$ & 2 &  0 \\
\hline
$\F_4$  &   $\F_4$                                                    & ss            &   $\F_4(0)=\F_4/{\rm T}^{4}$                                                     & ${\rm T}^{4}$ & 4 & 0\\
             &   $\F_4/{\rm T}^{2}$                                  & a        & $\F_4(0)$                                                                                  & ${\rm T}^{2}$ & 4 & 2 \\
             &   $\F_4/A_{1}^{l}\cdot{\rm  T}^{1}$          & s           & $\F_4(1)=\F_4/A_{1}^{l}\cdot{\rm T}^{3}$                                & ${\rm T}^{2}$ & 3 & 1 \\
             &   $\F_4/A_{1}^{s}\cdot{\rm  T}^{1}$         & s            & $\F_4(4)=\F_4/A_{1}^{s}\cdot{\rm T}^{3}$                              & ${\rm T}^{2}$ & 3 & 1 \\
             &  $\F_4/A_{2}^{l}$                                       & ss           &  $\F_{4}(1, 2)=\F_4/A_2^{l}\cdot{\rm T}^{2}$                          & ${\rm T}^{2}$ & 2 & 0 \\
             &  $\F_4/A_{2}^{s}$                                      & ss          &  $\F_{4}(3, 4)=\F_4/A_2^{s}\cdot{\rm T}^{2}$                         & ${\rm T}^{2}$ & 2 & 0 \\
             &   $\F_4/A_1\times A_1$                           & ss           &  $\F_{4}(1, 4)=\F_4/(A_1\times A_1)\cdot{\rm T}^{2}$               & ${\rm T}^{2}$ & 2 & 0 \\
             &  $\F_4/B_2$                                             & ss            &  $\F_{4}(2, 3)=\F_4/B_2\cdot{\rm T}^{2}$                                & ${\rm T}^{2}$ & 2 & 0 \\
\hline             
$\E_6$ &   $\E_6$                                                    & ss           &  $\E_{6}(0)=\E_6/{\rm T}^{6}$                                                   & ${\rm T}^{6}$ &  6 & 0 \\
             &   $\E_6/{\rm  T}^{2}$                                 & a        &  $\E_{6}(0)$                                                                              & ${\rm T}^{4}$ &  6 & 2 \\
             &   $\E_6/{\rm  T}^{4}$                                 & a        &  $\E_{6}(0)$                                                                              & ${\rm T}^{2}$ &  6 & 4 \\
             &   $\E_6/A_1\cdot{\rm T}^{1}$                  &  s           & $\E_6(1)=\E_6/A_1\cdot{\rm T}^{5}$                                       & ${\rm T}^{4}$ &  5 & 1 \\
              &   $\E_6/A_1\cdot{\rm T}^{3}$                  & s            & $\E_6(1)$                                                                                 & ${\rm T}^{2}$ &  5 & 3 \\
             &   $\E_6/(A_1)^{2}$                                     & ss           & $\E_6(3, 5)=\E_6/(A_1)^{2}\cdot{\rm T}^{4}$                          & ${\rm T}^{4}$ &  4 & 0 \\
             &   $\E_6/(A_1)^{2}\cdot{\rm T}^{2}$          & s         & $\E_6(3, 5)$                                                                               & ${\rm T}^{2}$ &  4 & 2 \\
             &   $\E_6/A_2$                                             & ss          & $\E_6(4, 5)=\E_6/A_2\cdot{\rm T}^{4}$                                   & ${\rm T}^{4}$ &  4 & 0 \\
             &   $\E_6/A_2\cdot{\rm T}^{2}$                   & s          & $\E_6(4, 5)$                                                                                & ${\rm T}^{2}$ &  4 & 2 \\
             &   $\E_6/(A_1)^{3}\cdot{\rm T}^{1}$           & s       & $\E_6(1, 3, 5)=\E_6/(A_1)^{3}\cdot{\rm T}^{3}$                         & ${\rm T}^{2}$ &  3 & 1 \\
             &   $\E_6/(A_2\times A_1)\cdot{\rm T}^{1}$   & s         & $\E_6(2, 4, 5)=\E_6/(A_2\times A_1)\cdot{\rm T}^{3}$               & ${\rm T}^{2}$ &  3 & 1 \\
             &   $\E_6/A_3\cdot{\rm T}^{1}$                   & s        & $\E_6(3, 4, 5)=\E_6/A_3\cdot{\rm T}^{3}$                                  & ${\rm T}^{2}$ &  3 & 1 \\
             &   $\E_6/A_4$  					       & ss         & $\E_6(2, 3, 4, 5)=\E_6/A_4\cdot{\rm T}^{2}$                              & ${\rm T}^{2}$   &  2 & 0 \\
              &   $\E_6/A_3\times A_1$  		               & ss        & $\E_6(1, 3, 4, 5)=\E_6/(A_3\times A_1)\cdot{\rm T}^{2}$              & ${\rm T}^{2}$   &  2 & 0 \\
              &   $\E_6/A_2\times A_2$  		               & ss       & $\E_6(1, 2, 4, 5)=\E_6/(A_2\times A_2)\cdot{\rm T}^{2}$              & ${\rm T}^{2}$   &  2 & 0 \\
              &   $\E_6/A_2\times (A_1)^{2}$  		       & ss      & $\E_6(2, 4, 5, 6)=\E_6/(A_2\times (A_1)^{2})\cdot{\rm T}^{2}$     & ${\rm T}^{2}$   &  2 & 0 \\
              &   $\E_6/D_4$  		                                & ss       & $\E_6(2, 3, 4, 6)=\E_6/D_4\cdot{\rm T}^{2}$                               & ${\rm T}^{2}$   &  2 & 0 \\
              \hline
$\E_7$    &    $\E_7/{\rm T}^{1}$                             & a          &  $\E_{7}(0)=\E_7/{\rm T}^{7}$                                                   & ${\rm T}^{6}$ &  7 & 1 \\
               &    $\E_7/{\rm T}^{3}$                             & a          &  $\E_{7}(0)$                                                                               & ${\rm T}^{4}$ &  7 & 3 \\
               &    $\E_7/{\rm T}^{5}$                             & a            &  $\E_{7}(0)$                                                                               & ${\rm T}^{2}$ &  7 & 5 \\
               &    $\E_7/A_1$                                       & ss               &  $\E_{7}(1)=\E_7/A_1\cdot{\rm T}^{6}$                                    & ${\rm T}^{6}$ &  6 & 0 \\
               &    $\E_7/A_1\cdot{\rm T}^{2}$            & s             &  $\E_{7}(1)$                                                                                & ${\rm T}^{4}$ &  6 & 2 \\
               &    $\E_7/A_1\cdot{\rm T}^{4}$            & s               &  $\E_{7}(1)$                                                                                & ${\rm T}^{2}$ &  6 & 4 \\
               &    $\E_7/(A_1)^{2}\cdot{\rm T}^{1}$     &s               &  $\E_{7}(4,6)=\E_7/(A_1)^{2}\cdot{\rm T}^{5}$                       & ${\rm T}^{4}$ &  5 & 1 \\
               &    $\E_7/(A_1)^{2}\cdot{\rm T}^{3}$     &s                &  $\E_{7}(4,6)$                                                                          & ${\rm T}^{2}$ &  5 & 3 \\
               &    $\E_7/A_2\cdot{\rm T}^{1}$             & s              &  $\E_{7}(5,6)=\E_7/A_2\cdot{\rm T}^{5}$                               & ${\rm T}^{4}$ &  5 & 1 \\
               &    $\E_7/A_2\cdot{\rm T}^{3}$             & s              &  $\E_{7}(5,6)$                                                                          & ${\rm T}^{2}$ &  5 & 3 \\
               &    $\E_7/(A_1)^{3}$  (type A)                 & ss               &  $\E_{7}(1, 3, 5)=\E_7/(A_1)^{3}\cdot{\rm T}^{4}$                   & ${\rm T}^{4}$ &  4 & 0 \\
              &    $\E_7/(A_1)^{3}\cdot{\rm T}^{2}$ (type A)   & s       &  $\E_{7}(1, 3, 5)$ \ \  \ (type A)                                             & ${\rm T}^{2}$ &  4 & 2 \\
              &    $\E_7/(A_1)^{3}$   (type B)                   & ss               &  $\E_{7}(1, 3, 7)=\E_7/(A_1)^{3}\cdot{\rm T}^{4}$                    & ${\rm T}^{4}$ &  4 & 0 \\
              &    $\E_7/(A_1)^{3}\cdot{\rm T}^{2}$ (type B) & s       &  $\E_{7}(1, 3, 7) $ \ \ \  (type B)                                              & ${\rm T}^{2}$ &  4 & 2 \\
              &    $\E_7/A_2\times A_1$                      & ss               &  $\E_{7}(3, 5,6)=\E_7/(A_2\times A_1)\cdot{\rm T}^{4}$                 & ${\rm T}^{4}$ &  4 & 0 \\
       &    $\E_7/(A_2\times A_1)\cdot{\rm T}^{2}$   & s              &  $\E_{7}(3, 5,6)$                                                                            & ${\rm T}^{2}$ &  4 & 2 \\
       &    $\E_7/A_3$                                             & ss                &  $\E_{7}(4, 5,6)=\E_7/A_3\cdot{\rm T}^{4}$                                & ${\rm T}^{4}$ &  4 & 0 \\
      &    $\E_7/A_3\cdot{\rm T}^{2}$                     & s               &  $\E_{7}(4, 5,6)$                                                                            & ${\rm T}^{2}$ &  4 & 2 \\
      \hline
                     \end{tabular}
\]
(continued)
\[
\begin{tabular}{c |  l  |  l | l | l | c | c }
$G$ & C-space $M=G/H$  & type & flag manifold $F=G/K$  & fiber & $b_{2}(F)$ & $b_{2}(M)$ \\
\thickline
$\E_7$    &  $\E_7/A_4\cdot{\rm T}^{1}$      & s                  & $\E_7(1, 2, 3, 4)=\E_7/A_4\cdot{\rm T}^{3}$                                & ${\rm T}^{2}$  & 3 & 1 \\  
 &  $\E_7/(A_3\times A_1)\cdot{\rm T}^{1}$ (type A)      & s              & $\E_7(1, 2, 3, 5)=\E_7/(A_3\times A_1)\cdot{\rm T}^{3}$                & ${\rm T}^{2}$  & 3 & 1 \\                                                      
 &  $\E_7/(A_3\times A_1)\cdot{\rm T}^{1}$  (type B)     & s               & $\E_7(1, 2, 3, 7)=\E_7/(A_3\times A_1)\cdot{\rm T}^{3}$                & ${\rm T}^{2}$  & 3 & 1 \\
 &  $\E_7/(A_2)^{2}\cdot{\rm T}^{1}$            & s               & $\E_7(1, 2, 4, 5)=\E_7/(A_2)^{2}\cdot{\rm T}^{3}$                         & ${\rm T}^{2}$  & 3 & 1 \\                                                      
&  $\E_7/(A_2\times (A_1)^{2})\cdot{\rm T}^{1}$      & s               & $\E_7(1, 2, 4, 6)=\E_7/(A_2\times (A_1)^{2})\cdot{\rm T}^{3}$             & ${\rm T}^{2}$  & 3 & 1 \\
&  $\E_7/(A_1)^{4}\cdot{\rm T}^{1}$            & s                & $\E_7(1, 3, 5, 7)=\E_7/(A_1)^{4}\cdot{\rm T}^{3}$                         & ${\rm T}^{2}$  & 3 & 1 \\  
&  $\E_7/D_4\cdot{\rm T}^{1}$            & s                 & $\E_7(3, 4, 5, 7)=\E_7/D_{4}\cdot{\rm T}^{3}$                         & ${\rm T}^{2}$  & 3 & 1 \\  
&  $\E_7/A_5$     (type A)                   & ss                &  $\E_{7}(1, 2, 3, 4, 5)=\E_7/A_5\cdot{\rm T}^{2}$                                & ${\rm T}^{2}$ &  2 & 0 \\
&  $\E_7/A_5$     (type B)                   & ss                &  $\E_{7}(1, 2, 3, 4, 7)=\E_7/A_5\cdot{\rm T}^{2}$                                & ${\rm T}^{2}$ &  2 & 0 \\
&  $\E_7/A_4\times A_1$       & ss                 & $\E_7(1, 2, 3, 4, 6)=\E_7/(A_{4}\times A_1)\cdot{\rm T}^{2}$                         & ${\rm T}^{2}$  & 2 & 0 \\ 
&  $\E_7/A_3\times A_2$       & ss                 & $\E_7(1, 2, 3, 5, 6)=\E_7/(A_{3}\times A_2)\cdot{\rm T}^{2}$                         & ${\rm T}^{2}$  & 2 & 0 \\ 
&  $\E_7/A_3\times (A_1)^{2}$       & ss                 & $\E_7(1, 2, 3, 5, 7)=\E_7/(A_3\times (A_1)^{2})\cdot{\rm T}^{2}$                         & ${\rm T}^{2}$  & 2 & 0 \\ 
&  $\E_7/D_4\times A_1$       & ss                 & $\E_7(1, 3, 4, 5, 7)=\E_7/(D_4\times A_1)\cdot{\rm T}^{2}$                         & ${\rm T}^{2}$  & 2 & 0 \\ 
&  $\E_7/(A_2)^{2}\times A_1$       & ss                 & $\E_7(1, 2, 5, 6, 7)=\E_7/((A_2)^{2}\times A_1)\cdot{\rm T}^{2}$                         & ${\rm T}^{2}$  & 2 & 0 \\ 
&  $\E_7/A_2\times (A_1)^{3}$       & ss                 & $\E_7(1, 3, 5, 6, 7)=\E_7/(A_2\times (A_1)^{3})\cdot{\rm T}^{2}$                         & ${\rm T}^{2}$  & 2 & 0 \\ 
&  $\E_7/D_5$       & ss                 & $\E_7(3, 4, 5, 6, 7)=\E_7/(D_5\times A_1)\cdot{\rm T}^{2}$                                          & ${\rm T}^{2}$  & 2 & 0 \\ 
\hline
$\E_8$ &   $\E_8$                                                    & ss           &  $\E_{8}(0)=\E_8/{\rm T}^{8}$                                               & ${\rm T}^{8}$ &  8 & 0 \\
             &   $\E_8/{\rm  T}^{2}$                                 & a        &  $\E_{8}(0)$                                                                               & ${\rm T}^{6}$ &  8 & 2 \\
             &   $\E_8/{\rm  T}^{4}$                                 & a        &  $\E_{8}(0)$                                                                               & ${\rm T}^{4}$ &  8 & 4 \\    
             &   $\E_8/{\rm  T}^{6}$                                 & a        &  $\E_{8}(0)$                                                                               & ${\rm T}^{2}$ &  8 & 6 \\    
              &   $\E_8/A_1\cdot{\rm T}^{1}$                  &  s           & $\E_8(1)=\E_8/A_1\cdot{\rm T}^{7}$                                   & ${\rm T}^{6}$ &  7 & 1 \\
              &   $\E_8/A_1\cdot{\rm T}^{3}$                  & s            & $\E_8(1)$                                                                               & ${\rm T}^{4}$ &  7 & 3 \\  
               &   $\E_8/A_1\cdot{\rm T}^{5}$                  & s            & $\E_8(1)$                                                                              & ${\rm T}^{2}$ &  7 & 5 \\   
               &   $\E_8/A_2$                                          &  ss          & $\E_8(1, 2)=\E_8/A_2\cdot{\rm T}^{6}$                                & ${\rm T}^{6}$ &  6 & 0 \\   
               &   $\E_8/A_2\cdot{\rm T}^{2}$                &  s          & $\E_8(1, 2)$                                                                            & ${\rm T}^{4}$ &  6 & 2 \\ 
                &   $\E_8/A_2\cdot{\rm T}^{4}$                &  s          & $\E_8(1, 2)$                                                                            & ${\rm T}^{2}$ &  6 & 4 \\   
                  &   $\E_8/A_1\times A_1$                      &  ss          & $\E_8(1, 3)=\E_8/(A_1)^{2}\cdot{\rm T}^{6}$                        & ${\rm T}^{6}$ &  6 & 0 \\ 
    &   $\E_8/(A_1\times A_1)\cdot{\rm T}^{2}$           &  s         & $\E_8(1, 3)$                                                                           & ${\rm T}^{4}$ &  6 & 2 \\      
    &   $\E_8/(A_1\times A_1)\cdot{\rm T}^{4}$           &  s          & $\E_8(1, 3)$                                                                           & ${\rm T}^{2}$ &  6 & 4 \\ 
     &   $\E_8/A_3\cdot{\rm T}^{1}$                           &  s           & $\E_8(1, 2, 3)=\E_8/A_3\cdot{\rm T}^{5}$                            & ${\rm T}^{4}$ &  5 & 1 \\
     &   $\E_8/A_3\cdot{\rm T}^{3}$                           &  s           & $\E_8(1, 2, 3)$                                                                        & ${\rm T}^{2}$ &  5 & 3 \\
     &   $\E_8/(A_2\times A_1)\cdot{\rm T}^{1}$           &  s         & $\E_8(1, 2, 4)=\E_8/(A_2\times A_1)\cdot{\rm T}^{5}$              & ${\rm T}^{4}$ &  5 & 1 \\   
     &   $\E_8/(A_2\times A_1)\cdot{\rm T}^{3}$           &  s         & $\E_8(1, 2, 4)$                                                                         & ${\rm T}^{2}$ &  5 & 3 \\  
     &   $\E_8/(A_1)^{3}\cdot{\rm T}^{1}$                   &  s         & $\E_8(1, 3, 5)=\E_8/(A_1)^{3}\cdot{\rm T}^{5}$                              & ${\rm T}^{4}$ &  5 & 1 \\  
     &   $\E_8/(A_1)^{3}\cdot{\rm T}^{3}$                   &  s         & $\E_8(1, 3, 5)$                                                                         & ${\rm T}^{2}$ &  5 & 3 \\  
     &   $\E_8/A_4$                                                     &  ss          & $\E_8(1, 2, 3, 4)=\E_8/A_4\cdot{\rm T}^{4}$                         & ${\rm T}^{4}$ &  4 & 0 \\
     &   $\E_8/A_4\cdot{\rm T}^{2}$                           &  s           & $\E_8(1, 2, 3, 4)$                                                                      & ${\rm T}^{2}$ &  4 & 2 \\
     &   $\E_8/A_3\times A_1$                                     & ss         & $\E_8(1, 2, 3, 5)=\E_8/(A_3\times A_1)\cdot{\rm T}^{4}$           & ${\rm T}^{4}$ &  4 & 0 \\  
     &   $\E_8/(A_3\times A_1)\cdot{\rm T}^{2}$         & s         & $\E_8(1, 2, 3, 5)$                                                                         & ${\rm T}^{2}$ &  4 & 2 \\    
     &   $\E_8/A_2\times A_2$        			   &  ss         & $\E_8(1, 2, 4,  5)=\E_8/(A_2\times A_2)\cdot{\rm T}^{4}$              & ${\rm T}^{4}$ &  4 & 0 \\   
   &   $\E_8/(A_2\times A_2)\cdot{\rm T}^{2}$        &  s             & $\E_8(1, 2, 4,  5)$            								  & ${\rm T}^{2}$ &  4 & 2 \\  
    &   $\E_8/A_2\times(A_1)^{2}$      			  &  ss             & $\E_8(1, 2, 4,  6)=\E_8/(A_2\times(A_1)^{2})\cdot{\rm T}^{4}$   & ${\rm T}^{4}$ &  4 & 0 \\  
     &   $\E_8/(A_2\times(A_1)^{2})\cdot{\rm T}^{2}$  & s            & $\E_8(1, 2, 4,  6)$  									 & ${\rm T}^{2}$ &  4 & 2 \\  
     &   $\E_8/(A_1)^{4}$                        		&  ss                 & $\E_8(1, 3, 5, 7)=\E_8/(A_1)^{4}\cdot{\rm T}^{4}$                       & ${\rm T}^{4}$ &  4 & 0 \\ 
          &   $\E_8/(A_1)^{4}\cdot{\rm T}^{2}$     &  s                      & $\E_8(1, 3, 5, 7)$						                           & ${\rm T}^{2}$ &  4 & 2 \\ 
          &   $\E_8/D_4$                                             &  ss             & $\E_8(4, 5, 6, 8)=\E_8/D_4\cdot{\rm T}^{4}$                         & ${\rm T}^{4}$ &  4 & 0 \\
            &   $\E_8/D_4\cdot{\rm T}^{2}$              &  s                  & $\E_8(4, 5, 6, 8)$           					              & ${\rm T}^{2}$ &  4 & 2     \\
            \hline  
                     \end{tabular}
\]
(continued)
\[
\begin{tabular}{c |  l  |  l | l | l | c | c }
$G$ & C-space $M=G/H$  & type & flag manifold $F=G/K$  & fiber & $b_{2}(F)$ & $b_{2}(M)$ \\
\thickline
$\E_8$       &   $\E_8/A_5\cdot{\rm T}^{1}$                   &  s          & $\E_8(1, 2, 3, 4, 5)=\E_8/A_5\cdot{\rm T}^{3}$                            & ${\rm T}^{2}$ &  3 & 1 \\
                 &   $\E_8/(A_4\times A_1)\cdot{\rm T}^{1}$  &  s         & $\E_8(1, 2, 3, 4, 6)=\E_8/(A_4\times A_1)\cdot{\rm T}^{3}$         & ${\rm T}^{2}$ &  3 & 1 \\
                 &   $\E_8/(A_3\times A_2)\cdot{\rm T}^{1}$  &  s         & $\E_8(1, 2, 3, 5, 6)=\E_8/(A_3\times A_2)\cdot{\rm T}^{3}$         & ${\rm T}^{2}$ &  3 & 1 \\
                &   $\E_8/(A_3\times (A_1)^{2})\cdot{\rm T}^{1}$  &  s         & $\E_8(1, 2, 3, 5, 7)=\E_8/(A_3\times (A_1)^{2})\cdot{\rm T}^{3}$    & ${\rm T}^{2}$ &  3 & 1      \\
             & $\E_8/((A_2)^{2}\times A_1)\cdot{\rm T}^{1}$     & s   &$\E_8(1, 2, 4, 5, 7)=\E_8/((A_2)^{2}\times A_1)\cdot{\rm T}^{3}$  & ${\rm T}^{2}$ &  3 & 1\\
             & $\E_8/(A_2\times (A_1)^{3})\cdot{\rm T}^{1}$     & s   &$\E_8(1, 2, 4, 6, 8)=\E_8/(A_2\times (A_1)^{3})\cdot{\rm T}^{3}$  & ${\rm T}^{2}$ &  3 & 1\\
             & $\E_8/(D_4\times A_1)\cdot{\rm T}^{1}$            & s   &$\E_8(1, 4, 5, 6, 8)=\E_8/(D_4\times A_1)\cdot{\rm T}^{3}$              & ${\rm T}^{2}$ &  3 & 1\\
              & $\E_8/D_5\cdot{\rm T}^{1}$              & s   &$\E_8(4, 5, 6, 7, 8)=\E_8/D_5\cdot{\rm T}^{3}$                                               & ${\rm T}^{2}$ &  3 & 1\\
               & $\E_8/A_6$                                      & ss   &$\E_8(1,  2, 3, 4, 5, 6)=\E_8/A_6\cdot{\rm T}^{2}$                            & ${\rm T}^{2}$ &  2 & 0\\
                & $\E_8/A_5\times A_1$                   & ss   &$\E_8(1, 2, 3,  4, 5, 7)=\E_8/(A_5\times A_1)\cdot{\rm T}^{2}$              & ${\rm T}^{2}$ &  2 & 0\\
                & $\E_8/A_4\times A_2$                  & ss   &$\E_8(1, 2, 3,  4, 6, 7)=\E_8/(A_4\times A_2)\cdot{\rm T}^{2}$              & ${\rm T}^{2}$ &  2 & 0\\
                   & $\E_8/A_3\times A_3$               & ss   &$\E_8(1, 2, 3,  5, 6, 7)=\E_8/(A_3\times A_3)\cdot{\rm T}^{2}$              & ${\rm T}^{2}$ &  2 & 0\\
                   & $\E_8/A_4\times (A_1)^{2}$        & ss   &$\E_8(1, 2, 3,  4, 6, 8)=\E_8/(A_4\times (A_1)^{2})\cdot{\rm T}^{2}$              & ${\rm T}^{2}$ &  2 & 0\\
              & $\E_8/D_4\times A_2$                     & ss   &$\E_8(1, 2, 4,  5, 6, 8)=\E_8/(D_4\times A_2)\cdot{\rm T}^{2}$              & ${\rm T}^{2}$ &  2 & 0\\
              & $\E_8/D_5\times A_1$                      & ss   &$\E_8(1, 4, 5,  6, 7, 8)=\E_8/(D_5\times A_1)\cdot{\rm T}^{2}$              & ${\rm T}^{2}$ &  2 & 0\\
              & $\E_8/D_6$                                       & ss   &$\E_8(2, 3, 4,  5, 6, 8)=\E_8/D_6\cdot{\rm T}^{2}$                                     & ${\rm T}^{2}$ &  2 & 0\\
              & $\E_8/A_3\times A_2\times A_1$ & ss   &$\E_8(1, 2, 3,  6, 7, 8)=\E_8/(A_3\times A_2\times A_1)\cdot{\rm T}^{2}$              & ${\rm T}^{2}$ &  2 & 0\\
              & $\E_8/(A_2)^{2}\times(A_1)^{2}$ & ss   &$\E_8(1, 2, 4,  6, 7, 8)=\E_8/((A_2)^{2}\times(A_1)^{2})\cdot{\rm T}^{2}$              & ${\rm T}^{2}$ &  2 & 0\\
               & $\E_8/\E_6$                                       & ss   &$\E_8( 3, 4,  5, 6, 7, 8)=\E_8/\E_6\cdot{\rm T}^{2}$                                     & ${\rm T}^{2}$ &  2 & 0\\
               \hline              
     \end{tabular}
\]

\end{document}